\documentclass[10pt]{article}
 \pdfmapline{+eufm10 EUFM10 <eufm10.pfb}
 \pdfmapline{+eufm7 EUFM7 <eufm7.pfb}
 \pdfmapline{+eufm5 EUFM5 <eufm5.pfb}
 \pdfmapline{+eufb10 EUFB10 <eufb10.pfb}
 \pdfmapline{+eufb7 EUFB7 <eufb7.pfb}
 \pdfmapline{+eufb5 EUFB5 <eufb5.pfb}
 
\usepackage[
a4paper,
left=2.1cm,
right=2.1cm,
top=2.1cm,
bottom=2.1cm
]{geometry}

\usepackage[utf8]{inputenc}

\usepackage{amsmath,amsthm,amssymb,amsopn,bm,bbm}
\usepackage{indentfirst}
\usepackage[section]{placeins}
\numberwithin{equation}{section}

\usepackage[table,svgnames]{xcolor}
\usepackage{graphicx}
\usepackage{tikz}
\usetikzlibrary{patterns}
\usetikzlibrary{calc} 
\usepackage{enumitem,multicol}

\usepackage{caption}
\usepackage{subcaption}
\usepackage{booktabs}

\usepackage[hyperfootnotes=false]{hyperref}

\usepackage{hyperref}

\newtheorem{thm}{Theorem}
\newtheorem{theorem}[thm]{Theorem}
\numberwithin{thm}{section}
\newtheorem{lemma}[thm]{Lemma}

\newtheorem{proposition}[thm]{Proposition}
\newtheorem{rem+}[thm]{Remark}
\newtheorem{remark}[thm]{Remark}

\newtheorem{definition}[thm]{Definition}

\newcommand{\ds}{\displaystyle}

\usepackage{mathtools} 
\newcommand{\defeq}{\coloneqq}   

\usepackage{braket}

\usepackage{setspace}

\def\ed{{\varepsilon\delta}}

\def\d{{\delta}}
\def\1{{\bf 1}}

\def\D{{\mathbb{D}}}
\def\A{{\mathbb{A}}}
\def\R{{\mathbb{R}}}
\def\N{{\mathbb{N}}}
\def\Z{{\mathbb{Z}}}
\def\Q{{\mathbb{Q}}}

\def\H{{\mathbb{H}}}

\def\W{{\mathbb{W}}}

\def\p{{\partial}}

\def\Ga{{\bf a}}

\def\Ge{{\bf e}}

\def\Gu{{\bf u}}

\def\GA{{\bf A}}
\def\GB{{\bf B}}

\def\GE{{\bf E}}

\def\GH{{\bf H}}
\def\GI{{\bf I}}

\def\GM{{\bf M}}

\def\GO{{\bf O}}

\def\GR{{\bf R}}
\def\GS{{\bf S}}

\def\GU{{\bf U}}
\def\GV{{\bf V}}

\def\cS{{\cal S}}

\def\cQ{{\cal Q}}

\def\O{{\Omega}}
\def\o{{\omega}}
\def\cJ{{\cal J}}
\def\cU{{\cal U}}
\def\cR{{\cal R}}

\def\cA{{\cal A}}
\def\cV{{\cal V}}
\def\cC{{\cal C}}

\def\cO{{\cal O}}

\def\cY{{\cal Y}}

\def\cH{{\cal  H}}

\def\cT{{\cal  T}}

\def\cH{{\cal  H}}
\def\cK{{\mathcal K}}
\def\cV{{\mathcal V}}

\def\fv{\mathfrak{v}}
\def\fs{\mathfrak{s}}

\def\fI{\mathfrak{I}}
\def\fG{\mathfrak{G}}

\def\Ted{\cT_\ed}

\def\dia{{\diamond}}
\def\bt{{\blacktriangle}}

\def\wt{\widetilde}
\def\wh{\widehat}
\def\X{{\times}}
\def\e{\varepsilon}
\def\ov{\overline}

\def\fu{\mathfrak{u}}
\def\fU{\mathfrak{U}}

\def\fF{\mathfrak{F}}

\def\fV{\mathfrak{V}}

\def\fr{\mathfrak{r}}
\def\fv{\mathfrak{v}}

\def\fV{\mathfrak{V}}

\def\a{{\alpha}}
\def\b{{\beta}}

\def\bt{\blacktriangle}

\definecolor{skyblue}{RGB}{135,206,235}
\definecolor{deepskyblue}{RGB}{0, 191, 255}

\title{Homogenization of  a thin linear elastic plate reinforced with a   periodic mosaic of small rigid plates}

\author{%
	Amartya Chakrabortty\thanks{SMS, Fraunhofer ITWM, Kaiserslautern 67663, Germany. \texttt{amartya.chakrabortty@itwm.fraunhofer.de}}\hskip 10mm
	Georges Griso\thanks{Laboratoire Jacques-Louis Lions (LJLL), Sorbonne Universit\'e, CNRS, Universit\'e de Paris, F-75005 Paris, France. \texttt{griso@ljll.math.upmc.fr}}\hskip 10mm 
	Julia Orlik\thanks{SMS, Fraunhofer ITWM, Kaiserslautern 67663, Germany. \texttt{julia.orlik@itwm.fraunhofer.de}}
}

\date{}

\makeatletter
\renewcommand{\maketitle}{%
	\begin{center}
		{\LARGE \@title \par}
		\vskip 0.8em
		{\large \@author}
	\end{center}
	\vskip -1em
	\@thanks  
}
\makeatother

\newcommand{\keywords}[1]{\par\vspace{1ex}\noindent\textbf{Keywords:} #1}
\newcommand{\subjclass}[1]{\par\noindent\textbf{MSC 2020:} #1}

\begin{document}
	
	\maketitle
	
	\begin{abstract}		
In the framework of linearized elasticity, we study thin elastic composite plates with thickness $\delta$. The plates contain small, rigid rectangular plates distributed periodically along $\e$. Between two neighboring rigid plates is an elastic beam with thickness $\delta < \varepsilon/3 < 1$.   Through a simultaneous process of homogenization and dimension reduction, we obtain the limit model. Our analysis yields Korn-type inequalities adapted to the rigid-elastic geometry of the structure and provides a precise characterization of the limit deformation and displacement fields. 
In the 2D limit problem, the mixed second derivative of the out-of-plane displacement vanishes, so the bending field decomposes as the sum of two functions each depending on a single in-plane variable. As a result, the two directional homogenized energies decouple at the level of bending, while the membrane contributions remain coupled through a single common displacement field. Each directional problem involves three degrees of freedom: shear along the strip axis, cross-contraction/extension, and cross-bending. 
The corresponding correctors are defined in the same way in the periodicity cell. 
In the linearized setting, all the correctors are decomposed.

\end{abstract}

\keywords{Homogenization, dimension reduction, Kirchhoff-Love plate theory, rigid inclusions,  non-penetration condition.}

\subjclass{74K20, 74Q05, 74B05, 35B27, 74E30, 35Q74}

\section{Introduction}

The mathematical analysis of thin elastic structures has a long history. Starting from the classical Kirchhoff--Love and Reissner--Mindlin theories, a variety of reduced models have been derived from three-dimensional elasticity by means of asymptotic analysis or $\Gamma$-convergence methods (see, e.g., \cite{ciarlet1997mathematical,friesecke2002theorem,Ciar01}). In parallel, the homogenization of heterogeneous elastic materials with periodic structures has been extensively studied (see \cite{bensoussan1978asymptotic,allaire1992homogenization,nguetseng1989,UFO1,Ole}), leading to the concept of effective elastic tensors that capture the influence of micro-scale heterogeneity.

When both small thickness and periodic heterogeneity are present, the mechanical behavior of the structure depends on the interplay between the two small parameters: the plate thickness $\delta$ and the periodicity $\varepsilon$. The simultaneous limit $(\varepsilon,\delta) \to (0,0)$ therefore requires a careful multiscale analysis, since the corresponding limits do not commute \cite{Nedelec01}. Several studies have addressed related problems for thin periodic plates, shells, textiles, tubes and general $3$D structures (see, for instance, \cite{Migunova01, Migunova02, Julia02, Neu1, Friedrich01,GGDecomplPlate,larysa01,larysa02,larysa03,GFOW01,GOW01,hauck01,Panasenko01,Panasenko02}), but the presence of rigid inclusions with a non-penetration constraint introduces additional challenges. In particular, the rigidity constraint in the inclusions restricts the admissible displacement fields and affects the limit behavior in a nontrivial way. 

For homogenization of high-contrast composites, see \cite{High01,High02}. For homogenization without dimension reduction of composites with rigid inclusions, we refer to \cite{Kreisbeck01,Kreisbeck02,Kreisbeck03}. Simultaneous homogenization and dimension reduction have been addressed in \cite{Amartya01,Schmidt01} for high-contrast composites with soft inclusions, and in \cite{Amartya02} for thin composites with rigid inclusions subjected to nonlinear elasticity in the von K\'arm\'an regime.

In this work, we focus on thin elastic composites containing periodically distributed rigid inclusions; such composites play a crucial role in modern engineering applications (see \cite{julia01,glass}) because of their ability to conform to prescribed shapes without folding. Understanding their effective (or homogenized) two-dimensional mechanical behavior is therefore of significant theoretical and practical importance for design optimization.

\medskip

This paper presents a derivation of a limit model obtained through the simultaneous homogenization and dimension reduction of thin elastic plates reinforced with rigid substructures, within the framework of linearized elasticity. 
The structure under consideration, denoted by $\Omega_\delta$, consists of a connected elastic matrix forming a periodic framework of thin, straight rods intersecting along two mutually orthogonal in-plane directions, thereby creating a periodic structural frame $\Omega^S_{\varepsilon\delta}$ (the connected soft part). The resulting perforations are filled with rigid inclusions $\Omega^H_{\varepsilon\delta}$ (the disconnected hard part), see Figure~\ref{fig:01}. The rigidity of the hard inclusions is enforced by requiring that the infinitesimal (linearized) strain tensor $e(u)$ vanishes in $\Omega^H_{\varepsilon\delta}$.

Unlike \cite{larysa01,larysa02,GFOW01,GOW01}, which deal with beam structures with free lateral boundaries (homogeneous Neumann conditions on the lateral beam surfaces), in our setting the pores are filled with rigid (strain-free) subdomains that preserve distances. This corresponds to imposing inhomogeneous Dirichlet conditions on the holes, i.e., on the corresponding parts of the lateral beam surfaces. These Dirichlet boundary data are obtained by restricting an interpolation of the rigid displacements of the inclusions to their boundaries. A similar concept of inclusions in contact with the matrix, controlled via tangential jumps, was developed in \cite{CDO,Julia02,Migunova01}. We also refer to \cite{masonary1,masonary2} for related results on masonry structures.

This study concerns three main aspects: (i) the thickness and periodicity parameters $\delta$ and $\varepsilon$, respectively, satisfying
\begin{equation}\label{AS0}
	0 < 3\delta < \varepsilon < 1,\qquad  \e\to 0 \quad \text{and} \quad {\d\over \e}\to 0,
\end{equation}
throughout this article, and particularly in the context of asymptotic processes, we consistently assume the aforementioned convergences regarding the two small parameters $\varepsilon$ and $\delta$\\[1mm]
(ii) rigidity on the disconnected hard parts, and (iii) a non-penetration constraint. With this perspective, we assume that two consecutive rigid plates in the in-plane directions satisfy a non-penetration condition (NPC). Due to this constraint, the elasticity problem is formulated as a variational inequality. In the linearized setting, this assumption requires some justification, since violation of the non-penetration condition can occur only in rather degenerate situations, such as loss of coercivity of the linear material or special choices of the Poisson ratio that allow the inter-layer to shear until its thickness effectively collapses to a line. However, in the context of large deformations, as considered in our later studies, a multiaxial deformation of the elastic part may lead to contact between contiguous rigid substructures. We therefore impose the non-penetration conditions on the admissible fields and formulate the problem from the outset as a variational inequality (similar to \cite{CDO,Julia02,GOW01,GFOW01}), so as to obtain a formulation that remains valid for other regimes, such as the von K\'arm\'an and large-deformation regimes. 

Next, we recall a few mathematical tools, mostly developed by the second author, that are used in this paper.  
The orthotropic character of the structure allows us to improve the Kirchhoff--Love plate decomposition from \cite{GGDecomplPlate}, \cite[Chapter~11]{PUM}, \cite{larysa03}, and \cite{GGKL} by cancelling the shear terms (see the proof in the Appendix).
By an interpolation technique, the rigid displacements of the disconnected rigid substructures are extended to the whole domain. Such extensions are often used for perforated structures with Neumann conditions on the perforations; see, for example, \cite{PUM,larysa03,larysa01,larysa02,Migunova02}. In contrast, under homogeneous Dirichlet boundary conditions one usually extends by zero.

In this paper, we also present polynomial approximations (in this case even the exact solutions $\cU^\bt$ and $\cU^{\bt\bt}$), which strongly converge to the main limiting fields and correctors. Similar tools can be found in \cite{lattice,anisotropic,africangirls}.

We also point out that, in our setting, the inter-layer between two contiguous plates is thin and elastic, i.e.~softer than the plates themselves. If we fix $\varepsilon$ and let $\delta \to 0$, then the corrector problems and their corresponding macroscopic strain components coincide with those obtained in \cite{geymonat} and \cite{Migunova01}. The third corrector corresponds to a cylindrical hinge between two rigid plates.

The main mathematical novelty of this work lies in the analytical framework developed to handle the interplay between dimension reduction, high-contrast homogenization, and unilateral (non-penetration) constraints within linearized elasticity. In particular, the paper introduces a new decomposition of displacements consistent with the rigidity constraint, establishes uniform Korn-type estimates adapted to the hard/soft composite structure, and employs rescaled unfolding operators to capture the multiscale limit behavior.
These tools allow for the derivation of the limit variational equation governing the effective plate model. Below, we summarize the main analytical techniques and results.

\begin{itemize}
	\item In Section~\ref{S04}, we establish that every admissible displacement can be decomposed into a Kirchhoff--Love (KL) type displacement and a residual part satisfying the rigidity constraint, i.e.,
	\[
	u(x) = U_{KL}(x) + \wt u(x)
	= \begin{pmatrix}
		\ds \fU_1(x') - x_3 \p_1 \fU_3(x') \\[1.5mm]
		\ds \fU_2(x') - x_3 \p_2 \fU_3(x') \\[1.5mm]
		\fU_3(x') + \fu(x')
	\end{pmatrix} + \wt u(x),
	\quad \text{for a.e. } x = (x',x_3) \in \Omega_\delta,
	\]
	such that
	\[
	e(u) = e(U_{KL})=0, \quad \text{and} \quad \wt u = 0 \quad \text{a.e. in } \Omega^H_{\varepsilon\delta}.
	\]
	The membrane displacement $\fU_m = \fU_1 \mathbf{e}_1 + \fU_2 \mathbf{e}_2$ and the bending displacement $\fU_3$ satisfy the following Korn-type estimates:
	\[
	\|\fU_3\|_{H^1(\omega_\delta)} \leq \frac{C}{\varepsilon^{1/2}\delta} \|e(u)\|_{L^2(\Omega_\delta)}, 
	\quad 
	\|\fU_m\|_{L^2(\omega_\delta)} \leq \frac{C}{\varepsilon^{1/2}} \|e(u)\|_{L^2(\Omega_\delta)},
	\]
	which are sharper than those obtained using classical tools such as the Poincar\'e inequality.
	
	\item We also derive estimates of different orders for higher derivatives:
	\[
	\|\partial^2_{\alpha\beta}\fU_3\|_{L^2(\omega_\delta)} \leq \frac{C}{\delta^{3/2}} \|e(u)\|_{L^2(\Omega_\delta)}, 
	\quad 
	\|\partial_\alpha \fU_m\|_{L^2(\omega_\delta)} \leq \frac{C}{\delta^{1/2}} \|e(u)\|_{L^2(\Omega_\delta)}.
	\]
	Due to the difference in sharpness of these estimates, in Section~\ref{S05} we construct additional global fields using $\mathcal{Q}_1$-interpolation to ensure sufficient regularity of the unfolded limits of $\fU_m$ and $\fU_3$. In particular, we construct $\cU^\diamond$ and $\cR^\diamond$ in Subsection~\ref{SS42}, using $\mathcal{Q}_1$-interpolation on cells isomorphic to $(0,\varepsilon)^2$, to ensure sufficient regularity for the limit macroscopic fields. In Subsections~\ref{SS43}--\ref{SS44}, we define $\cR$, $\cU$, $\cU^\bt$, and $\cU^{\bt\bt}$ to obtain the exact unfolded limits presented in Section~\ref{S07}, specifically in Lemma~\ref{lem61}. For this construction, we divide each cell $(-\tfrac{\delta}{2}, \varepsilon - \tfrac{\delta}{2})^2$ into the union of two squares and two rectangles (see Figure~\ref{FigXX}), and perform $\mathcal{Q}_1$-interpolation on each of these subdomains.
	
	\item The main tool used to derive the explicit form of the limit strain tensor in Section~\ref{S07} relies on two rescaled unfolding operators (introduced in Section~\ref{S06}), together with the KL-decomposition and the global fields. First, in Lemma~\ref{lem61}, we derive the macroscopic displacement fields. Then, using the two-dimensional rescaled unfolding operator, we obtain the unfolded limits of these fields. Finally, in Lemma~\ref{L64}, we establish the unfolded limit of the infinitesimal (linearized) strain tensor by means of the three-dimensional rescaled unfolding operator. All microscopic fields depend on the direction corresponding to the two in-plane axes.
	
	\item The construction of global fields in Section~\ref{S05} also provides insight into the definition of the recovery sequence. The recovery sequence itself is constructed in Section~\ref{SS81}. In particular, in Theorems~\ref{lemAp5}--\ref{lemAp7}, we show that the spaces $W^{1,\infty}(\omega_\delta)$ with vanishing gradient, and $W^{2,\infty}((-\delta/2, L+\delta/2))$ with vanishing second derivative on the hard part, are dense in $W^{1,\infty}(\omega)$ and $W^{2,\infty}(0,L)$, respectively, in the sense of the unfolding operators.
	
	\item In Section~\ref{S09} (see Subsection~\ref{SS73} for the derivation), we define the set of limit displacements, denoted by $\D_0$. In Theorem~\ref{Th82}, we formulate the limit unfolded problem and establish the strong convergence of the strain tensor via the convergence of the energy, given by
	\[
	\lim\frac{1}{\e\d^4}\int_{\O^S_{\e\d}}\A_{\ed,ijkl}(x)e_{ij}(u_\ed)e_{kl}(u_\ed)\,dx
	= \sum_{\a=1}^2\int_{\o}\cA^{(\a)}E^{(\a)}(\fU)\cdot E^{(\a)}(\fV)\,dx',
	\]
	where
	\begin{equation*}
		\begin{aligned}
			E^{(1)}(\fV)&=\begin{pmatrix*}
				\p_2\fV_1\\
				\p_2\fV_2\\
				\p^2_{22}\fV_3
			\end{pmatrix*},\qquad
			E^{(2)}(\fV)=\begin{pmatrix*}
				\p_1\fV_2\\
				\p_1\fV_1\\
				\p^2_{11}\fV_3
			\end{pmatrix*},\qquad\forall\,\fV\in\D_0.
		\end{aligned}
	\end{equation*}
	Then, by solving the associated cell problems and introducing suitable correctors, which are the solutions of the following auxiliary problems
	\begin{equation*}
		\left.\begin{aligned}
			\int_{\cY_\a}\A_{ijkl}^{(\a)}(y)\big(\GM_{r,ij}^{(\a)}+E^{(\a)}_{ij}(0,\chi^{(\a)}_r)\big)E^{(\a)}_{kl}(0,\wh v^{(\a)})\,dy= 0,		
		\end{aligned}\right.\quad\forall\,\wh v^{(\a)}\in \wh{D}^{(\a)}(\o \X Y_\a \X \fI ),
	\end{equation*}
	where $\GM^{(\a)}_r$ for $\a=1,2$ and $r=1,2,3$ are $3\times3$ symmetric matrices given by 
	\[
	\begin{aligned}
		\GM^{(1)}_1&=\begin{pmatrix}
			0 & \ds\frac{1}{ 2} & 0\\
			\ds\frac{1}{2} & 0 & 0\\
			0 & 0 & 0
		\end{pmatrix},\quad &&
		\GM^{(1)}_2=\begin{pmatrix}
			0 & 0 & 0\\
			0 & 1 & 0\\
			0 & 0 & 0
		\end{pmatrix},\quad &&
		\GM^{(1)}_3=\begin{pmatrix}
			0 & 0 & 0\\
			0 & -y_3 & 0\\
			0 & 0 & 0
		\end{pmatrix},\\
		\GM^{(2)}_1&=\begin{pmatrix}
			1 & 0 & 0\\
			0 & 0 & 0\\
			0 & 0 & 0
		\end{pmatrix},\quad &&
		\GM^{(2)}_2=\begin{pmatrix}
			0 & \ds\frac{1}{2} & 0\\
			\ds\frac{1}{2} & 0 & 0\\
			0 & 0 & 0
		\end{pmatrix},\quad &&
		\GM^{(2)}_3=\begin{pmatrix}
			-y_3 & 0 & 0\\
			0 & 0 & 0\\
			0 & 0 & 0
		\end{pmatrix}.
	\end{aligned}
	\]
	Observe that the correctors have zero values on the lateral boundary parts with the normal belonging to the plate midplane and have zero Neumann conditions on the lateral parts with the normal in the $x_3$-direction.
	
	Finally, we characterize the homogenized problem and compute the effective (homogenized) tensor, given by
	\begin{equation*}
		\begin{aligned}
			\cA_{rs}^{(\a)}&=\int_{\cY_\a}\A^{(\a)}_{ijkl}(y)\left[\GM_{r,ij}^{(\a)}+E^{(\a)}_{ij}(0,\chi_r^{(\a)})\right]\left[\GM_{s,kl}^{(\a)}+E^{(\a)}_{kl}(0,\chi_s^{(\a)})\right]\,dy
		\end{aligned}
	\end{equation*}
	for $r,s\in\{1,2,3\}$.
	
	\item Technical proofs are presented in Appendix~\ref{AppA01}--\ref{SAppC}.
\end{itemize}

We end the introduction by drawing the reader's attention to the limit problem, in which the out-of-plane limit field decomposes into two components (i.e.~$\p^2_{12}\fU_3=0$ a.e.~in $\o$), and for each $\a$ the corresponding component depends only on the single in-plane variable $x_{3-\a}$. 
The membrane displacements $\mathfrak{U}_1$, $\mathfrak{U}_2$ remain coupled across the two directional energies through the single limit displacement field $\mathfrak{U}\in\D_0$. In the linear setting this decomposition extends to all quantities, including the three correctors. In future studies of nonlinear problems on the same topology, this decomposition will yield natural macroscopic constraints, which are particularly relevant for applications such as $3$D-printing technologies.

\subsection*{Notations} 
The following notation will be used:
\begin{itemize}
\item  $\o\doteq(0,L)^2$, $\gamma\doteq\{0\}\X(0,l)\cup (0,l)\X\{0\}$ with $0<l< L$, $\o_\d\doteq\ds\Big(-{\d\over 2},L+{\d\over 2}\Big)^2$,
\item  $\ds \fI\doteq \Big(-{1\over 2},{1\over 2}\Big)$, $\ds \fI_\kappa\doteq \Big(-{\kappa\over 2},{\kappa\over 2}\Big)$, $\kappa>0$, 
\item $Y_1\doteq(0,1)\X\ds \fI$, $Y_2\doteq\ds \fI\X(0,1)$, and $\cY_\alpha\doteq Y_\alpha\X\ds \fI$, $\alpha\in\{1,2\}$,
\item we choose $\e$ such that $\ds {L\over \e}\doteq N_\e\in \N^*$, $\ds {l\over \e}\doteq n_\e\in \N^*$\footnotemark,
\footnotetext{$ $ This implies that $\ds {L\over l}\in \Q^*$.}
\item  $\cK_\e\doteq \{0,\ldots,N_\e\}^2$,\quad   $\cK^*_\e\doteq \{0,\ldots,N_\e-1\}^2$,\quad $\ov{\cK}_\e\doteq \{-1,\ldots,N_\e\}^2$,\\
 $\cK_\e^{(1)}\doteq\{0,\ldots,N_\e-1\}\X\{0,\ldots,N_\e\}$,\quad $\cK_\e^{(2)}\doteq\{0,\ldots,N_\e\}\X\{0,\ldots,N_\e-1\}$,
\item  $\o_{pq}\doteq (p\e,p\e+\e)\X(q\e,q\e+\e)$,\\[1mm]
 $\ds \wt{\o}_{pq}\doteq \Big(p\e-{\d\over 2},(p+1)\e+{\d\over 2}\Big)\X \Big(q\e-{\d\over 2},(q+1)\e+{\d\over 2}\Big)$, \\[1mm]
$\ds \o^H_{pq}\doteq \Big(p\e+{\d\over 2},(p+1)\e-{\d\over 2}\Big)\X \Big(q\e+{\d\over 2},(q+1)\e-{\d\over 2}\Big)$,   $(p,q)\in \Z^2$,
\item $\GM_3$ is the space of $3\times 3$  real matrices, $\GS_3$ the space space of real $3\X 3$ symmetric matrices,  $\GI_3$  the identity matrix of $\GM_3$,
\item $\{\Ge_i\}_{\{1\leq i\leq n\}}$ the standard basis of $\R^n$, $n\geq 1$, $|\cdot|$ the standard Euclidean norm in $\R^n$,  $n\geq 1$, and Frobenius norm in $\GM_3$,
\item  $x'\doteq(x_1,x_2)$ the current point in $\R^2$ and $x\doteq(x',x_3)\doteq(x_1,x_2,x_3)$ the current point in $\R^3$,
\item  $\ds \partial_i\doteq\frac{\partial}{\partial x_i}$, $\ds\p^2_{ij}\doteq{\p^2\over \p x_i\p x_j}$ for $i,\,j\in \{1,2,3\}$, and $\ds d_\alpha\doteq{d\over dx_\alpha},\quad d^2_{\a\a}\doteq{d^2\over dx_\alpha^2},\quad \alpha\in\{1,2\}$,
\item  $\ds \partial_{y_i}\doteq\frac{\partial}{\partial y_i}$, for $i\in \{1,2,3\}$, 
\item   for any $u\in H^1(\O_\d)^3$  the linearized strain tensor $e(u)$ is the $3\X 3$ symmetric matrix whose entries are
\begin{equation*}
e_{ij}(u)={1\over 2}\big(\p_i u_j+\p_j u_i\big),\qquad \forall (i,j)\in \{1,2,3\}^2.
\end{equation*}
\end{itemize} 
Throughout this paper, we choose the parameters $\e$ and $\d$ such that  $0< 3\d < \e$.\\
 The Greek letters  $\alpha,\, \beta$ belong to $\{1,2\}$ and the Latin letters  $i,\,j,\,k,\,l$ to $ \{1,2,3\}$ (if not specified). We use the Einstein convention of summation over repeated indices.

\paragraph{Notation philosophy for the fields in Section \ref{S04}--\ref{S05}.}
The main displacement fields arising from the Kirchhoff--Love decomposition are denoted in fraktur. 
Since these fields alone do not carry sufficient sharp estimates to pass to the limit (see estimates~\eqref{rem+}), 
we construct auxiliary fields in three layers. 
Fields built by pure $\cQ_1$-interpolation at cell vertices are marked with $\diamond$. 
Fields that are constant on the rigid inclusions and $\cQ_1$-interpolated on the elastic beams carry no superscript decoration. 
The component of the latter arising from the translational part $\Ga_{pq}$ of the rigid displacement is marked $\bt$, 
and the difference $\mathcal{U}^{\bt\bt} = \mathcal{U} - \mathcal{U}^{\bt}$ captures the rotational contribution. 
Rigid displacements on the inclusions are written in vector form as 
$u = \Ga_{pq} + \cR_{pq} \wedge x_{pq}$, 
where $\cR_{pq} \in \mathbb{R}^3$ is the rotation vector. Equivalently, the gradient of a rigid displacement 
is an antisymmetric matrix, i.e., $\nabla u = \mathrm{skew}(\nabla u)$, 
and in three dimensions every such matrix takes the form $\GA_{pq}x = \cR_{pq} \wedge x$ for some $\cR_{pq} \in \mathbb{R}^3$.

\section{Problem setting}\label{S03}
In this section, we describe the domain, the set of admissible displacements and the linearized elasticity problem.
\begin{figure}[ht]
	\centering
	
	\begin{subfigure}[t]{0.31\textwidth}
		\centering
		\resizebox{\linewidth}{!}{%
			\begin{tikzpicture}[line cap=round, line join=round]
				\tikzset{hard/.style={fill=DeepSkyBlue!70}, soft/.style={fill=red!50}}
				\def\rod{0.25}\def\plate{0.75}\def\W{12.25}\def\H{12}
				\clip (0,-\rod) rectangle (\W,\H);
				\foreach \i in {0,...,13}{\fill[soft] (\i,-\rod) rectangle (\i+\rod,\H);}
				\foreach \j in {0,...,12}{\fill[soft] (0,\j-\rod) rectangle (\W,\j);}
				\foreach \i in {0,...,20}{
					\foreach \j in {0,...,14}{
						\pgfmathsetmacro\x{0.25+\i}
						\pgfmathsetmacro\y{0.00+\j}
						\fill[hard] (\x,\y) rectangle (\x+\plate,\y+\plate);
					}
				}
		\end{tikzpicture}}
		\caption{Top view of $\Omega_\delta$}
	\end{subfigure}\hfill%
	\begin{subfigure}[t]{0.31\textwidth}
		\centering
		\resizebox{\linewidth}{!}{%
			\begin{tikzpicture}[line cap=round, line join=round]
				\tikzset{soft/.style={fill=red!50}}
				\def\rod{0.25}\def\W{12.25}\def\H{12}
				\clip (0,-\rod) rectangle (\W,\H);
				\foreach \i in {0,...,13}{\fill[soft] (\i,-\rod) rectangle (\i+\rod,\H);}
				\foreach \j in {0,...,12}{\fill[soft] (0,\j-\rod) rectangle (\W,\j);}
		\end{tikzpicture}}
		\caption{Union of thin rods}
	\end{subfigure}\hfill%
	\begin{subfigure}[t]{0.31\textwidth}
		\centering
		\resizebox{\linewidth}{!}{%
			\begin{tikzpicture}[line cap=round, line join=round]
				\tikzset{hard/.style={fill=DeepSkyBlue!70}}
				\def\plate{0.75}\def\rod{0.25}\def\W{12.25}\def\H{12}
				\clip (0,-\rod) rectangle (\W,\H);
				\foreach \i in {0,...,20}{
					\foreach \j in {0,...,14}{
						\pgfmathsetmacro\x{0.25+\i}
						\pgfmathsetmacro\y{0.00+\j}
						\fill[hard] (\x,\y) rectangle (\x+\plate,\y+\plate);
					}
				}
		\end{tikzpicture}}
		\caption{Union of small plates}
	\end{subfigure}
	
	\caption{Blue (hard) tiles placed periodically, and red (soft) rod grid.}
	\label{fig:01}
\end{figure}

\subsection{Domain description}
The rigid re-enforced composite plate in its reference free-stressed form is occupying the region given by 
$$
\O_{\d}\doteq \o_\d\X\fI_\d=\hbox{Interior}\big(\ov{\O^S_{\e\d}}\cup\ov{\O^H_{\ed}}\big).
$$
$\O_{\d}$ is the union of small plates and beams given by
$$
\begin{aligned}
	&\O^H_{pq}=\o^H_{pq}\X\fI_\d,\quad \forall\, (p,q)\in \cK^*_\e, \\
	& \o^H_{\ed}=\bigcup_{(p,q)\in\cK_\e^*}\o^H_{pq}, \qquad \o^S_{\ed}\doteq \o_\d \setminus \ov{\o^H_{\ed}},\qquad \O^H_\ed=\o^H_\ed \X\fI_\d,\qquad  \O^S_{\e\d}=\O_{\d}\setminus \ov{\O^H_{\ed}}=\o^S_\ed \X\fI_\d.
\end{aligned}
$$
We also define the set of beams of direction $\Ge_1$ and $\Ge_2$ by 
$$
\begin{aligned}
	&\O^{(1)}_{\e\d}\doteq \o^{(1)}_{\e\d}\X\fI_\d,\quad \o^{(1)}_{\e\d}\doteq\bigcup_{q=0}^{N_\e}\Big(-{\d\over 2},L+{\d\over 2}\Big)\X\Big(q\e-{\d\over 2},q\e+{\d\over 2}\Big),\\ 
	& \O^{(2)}_{\e\d}\doteq \o^{(2)}_{\e\d}\X\fI_\d,\quad \o^{(2)}_{\e\d}\doteq\bigcup_{p=0}^{N_\e}\Big(p\e-{\d\over 2},p\e+{\d\over 2}\Big)\X\Big(-{\d\over 2},L+{\d\over 2}\Big)\quad\hbox{so}\quad \O^S_{\e\d}=\O^{(1)}_{\e\d}\cup \O^{(2)}_{\e\d}.
\end{aligned}
$$
\subsection{Set of admissible displacements}\label{SS32}
In this paper, we deal with displacements $u\in H^1(\O_\d)^3$ satisfying the following conditions:
\begin{itemize}
\item on a part of the boundary of soft part, we have homogeneous boundary condition 
\begin{equation}\label{EQ31}
	u=0,\quad\text{a.e. on }\Gamma_\d=\Big[\fI_\d\X\Big(-{\d\over 2},l-{\d\over 2}\Big)\cup \Big(-{\d\over 2},l-{\d\over 2}\Big)\X\fI_\d\Big]\X\fI_\d=\gamma_\d\X\fI_\d.
\end{equation}
Note that $\Gamma_\d$ is a three-dimensional region in the thickened domain $\O_\d$, reflecting the fact that the boundary condition is imposed on the extended microscopic domain $\o_\d$ rather than $\o$; in the limit $\d \to 0$ this region collapses to the surface $\gamma$.
\item the displacements are rigid displacements in $\O^H_\ed$
\begin{equation}\label{EQ32}
e(u)=0\qquad \hbox{a.e. in } \O^H_\ed.
\end{equation} So, we have
$$
\begin{aligned}
&u(x)=\Ga_{pq}+\cR_{pq}\land \ov{x}_{pq}=\begin{pmatrix} 
	\Ga_{pq,1}+\cR_{pq,2}\ov{x}_{pq,3}-\cR_{pq,3}\ov{x}_{pq,2}\\
	\Ga_{pq,2}+\cR_{pq,3}\ov{x}_{pq,1}-\cR_{pq,1}\ov{x}_{pq,3}\\
	\Ga_{pq,3}+\cR_{pq,1}\ov{x}_{pq,2}-\cR_{pq,2}\ov{x}_{pq,1}
	\end{pmatrix}\qquad \forall (p,q)\in\cK_\e^*,\\
	& \ov{x}_{pq}=x-\e\Big(p+{1\over 2}\Big)\Ge_1-\e\Big(q+{1\over 2}\Big)\Ge_2,\quad \forall (p,q)\in\cK_\e^*,\qquad (\Ga_{pq},\cR_{pq})\in\R^3\X \R^3.
\end{aligned}
$$
\item  they satisfy the non-penetration condition between  two consecutive rigid bodies 
$$
\begin{aligned}
&\Big[\Big(q\e+{\d\over 2}\Big)+u_2\Big(p\e+z_1,q\e+{\d\over 2},z_3\Big)\Big]-\Big[\Big(q\e-{\d\over 2}\Big)+u_2\Big(p\e+z_1,q\e-{\d\over 2},z_3\Big)\Big] \geq 0,\\
&\hbox{for all }(p,q)\in \{1,\ldots,N_\e-1\}\X\{0,\ldots,N_\e-1\}\;  \hbox{and for a.e. } (z_1,z_3)\in\left({\d\over 2},\e-{\d\over 2}\right)\X\fI_\d,\\
&\Big[\Big(p\e+{\d\over 2}\Big)+u_1\Big(p\e+{\d\over 2},q\e+z_2,z_3\Big)\Big]-\Big[\Big(p\e-{\d\over 2}\Big)+u_1\Big(p\e-{\d\over 2},q\e+z_2,z_3\Big)\Big] \geq 0,\\
&\hbox{for all }(p,q)\in \{0,\ldots,N_\e-1\}\X\{1,\ldots,N_\e-1\}\; \hbox{and for a.e. } (z_1,z_3)\in\left({\d\over 2},\e-{\d\over 2}\right)\X\fI_\d.
\end{aligned}
$$   This is equivalent to 
\begin{equation}\label{NPCM}
\begin{aligned}
& u_2\Big(p\e+z_1,q\e+{\d\over 2},z_3\Big)-u_2\Big(p\e+z_1,q\e-{\d\over 2},z_3\Big) \geq -\d,\\
&\hbox{for all }(p,q)\in \{1,\ldots,N_\e-1\}\X\{0,\ldots,N_\e-1\}\;  \hbox{and for a.e. } (z_1,z_3)\in\left({\d\over 2},\e-{\d\over 2}\right)\X\fI_\d,\\
&u_1\Big(p\e+{\d\over 2},q\e+z_2,z_3\Big)-u_1\Big(p\e-{\d\over 2},q\e+z_2,z_3\Big) \geq -\d,\\
&\hbox{for all }(p,q)\in \{0,\ldots,N_\e-1\}\X\{1,\ldots,N_\e-1\}\; \hbox{and for a.e. } (z_1,z_3)\in\left({\d\over 2},\e-{\d\over 2}\right)\X\fI_\d.
\end{aligned}
\end{equation} 
\end{itemize}
Observe that
\begin{equation}\label{EQ34}
\nabla u=\begin{pmatrix} 
	0 & -\cR_{pq,3} &  \cR_{pq,2}\\
	\cR_{pq,3} & 0 & -\cR_{pq,1}\\
	 -\cR_{pq,2} & \cR_{pq,1} & 0
	\end{pmatrix}\;\text{a.e. in } \O^H_{pq}.
\end{equation}
Since the soft part is glued to the hard part, we have continuity of displacement and response function at interface.\\  
The set of admissible displacements $\GU_{\e\d}$ is a closed convex subset of $H^1(\O_\d)^3$ 
\begin{equation*}
\GU_{\e\d}\doteq \big\{u\in H^1(\O_{\d})^3\;|\; \text{such that $u$ satisfies }\;\eqref{EQ31}-\eqref{EQ32}-\eqref{NPCM}\big\}.
\end{equation*}

Note that   for any displacement $u$ belonging to $\GU_{\e\d}$, we have 
$$\big\|e(u)\big\|_{L^2(\O_\d)}=\big\|e(u)\big\|_{L^2(\O^S_\ed)}=\GE(u).$$
\subsection{Linearized elasticity problem}\label{LEP}
 We consider the following elasticity problem, given in the form of a variational inequality (see \cite{contact,Stampa01}):
  \begin{equation}\label{Pb}
 \left\{ \begin{aligned}
& \hbox{Find $u_\ed\in \GU_{\e\d}$ such that},\\
&\hskip 10mm \int_{\O^S_{\e\d}}\A_{\ed,ijkl}(x)e_{ij}(u_\ed)e_{kl}(u_\ed-v)\,dx\leq  \int_{\O_{\d}}f_{\e\d}\cdot (u_\ed-v)\,dx,\quad\forall\, v\in \GU_{\e\d}.
\end{aligned}\right.
\end{equation}
We assume that the Hooke's coefficients $\A_{\ed,ijkl}\in L^\infty(\O_\ed^S)$ satisfy: 
\begin{itemize}
\item there exists  $c_0>0$ such that
		\begin{equation}\label{CoeCon2}
		\A_{\ed,ijkl} \GS_{ij}\GS_{kl}\geq c_0|\GS|^2,\quad\forall\,\GS\in\GS_3,\quad \hbox{a.e. in } \O_\ed^S.
		\end{equation}  
\item $\A_{\ed,ijkl}=\A_{\ed,ijlk}=\A_{\ed,klij}$, for all $(i,j,k,l)\in \{1,2,3\}^4$.
\end{itemize}
For $f_\ed\in L^2(\O_\d)^3$, we have using Stampacchia's lemma in \cite{Stampa01}, that exist a unique weak solution (due to the boundary condition \eqref{EQ31}) $u_\ed\in\GU_\ed$ to the problem \eqref{Pb}.

\section{Decomposition of the displacements of the structure}\label{S04}

We consider the following subspace of $H^1(\O_\d)^3 $:
$$
  \cH^1(\O_\d)^3\doteq\Big\{  u\in H^1(\O_\d)^3\;|\;  e(u)=0 \quad \hbox{ a.e. in  } \O^H_\ed\; \Big\}. 
 $$ 
The theorem below is a variant of \cite[Theorem 2]{GGKL}.  
\begin{theorem}\label{TH42} Any  displacement $u$ in $\cH^1(\O_{\d})^3$ is decomposed as follows: 
\begin{equation}\label{EQ41D+}
u(x)= \begin{pmatrix}
\ds \fU_1(x')-x_3 \p_1\fU_3 (x')\\[1.5mm]
\ds\fU_2(x')-x_3 \p_2 \fU_3 (x')\\[1.5mm]
\fU_3(x')+\fu(x')\\
\end{pmatrix}+\wt{u}(x)\quad  \hbox{ for a.e. $x$ in }  \Omega_\delta.
\end{equation} where $\fU_\a \in H^1(\o_{\d})$, $\fU_3 \in H^2(\o_{\d})$, $\fu\in H^1(\o_\d)$  and  $\wt{u}\in H^1(\O_\d)^3$ satisfies 
\begin{equation}\label{WMC02}
\wt{u}=0\quad \hbox{a.e. in }\O^H_\ed,\qquad 	\int_{\fI_\d} \wt{u}(\cdot, x_3)\, dx_3=0\qquad \hbox{a.e. in }\o_\d.
\end{equation}
We set $\fU_m=\fU_1\Ge_1+\fU_2\Ge_2$. The following estimates hold:
	\begin{equation}\label{rem+}
	\begin{aligned}
	&\|e_{\alpha\beta}(\fU_m)\|_{L^2(\o_\d)}+\d\big\|\p^2_{\alpha\beta}\fU_3\big\|_{L^2(\o_\d)}\le {C\over \d^{1/2}}\|e(u)\|_{L^2(\O_\d)},\\
	&\|\fu\|_{L^2(\o_\d)}+\d \big\|\nabla \fu\big\|_{L^2(\o_\d)} \leq {C \d^{1/2}}\|e(u)\|_{L^2(\O_\d)},\\
	&\|\wt{u}\|_{L^2(\O_\d )}+\d\|\nabla \wt{u}\|_{L^2(\O_\d)}\le C \d \|e(u)\|_{L^2(\O_\d)}.
	\end{aligned}
	\end{equation}
	The constants  depend only on $\o$.\\
	Moreover, if $u=0$ a.e. on $\Gamma_\d$ then, the terms of the decomposition \eqref{EQ41D+} satisfy
$$\fU_m=0,\quad \fU_3=\fu=0,\quad \nabla\fU_3=0, \quad \wt{u}=0\qquad \hbox{a.e. on } \gamma.$$	
\end{theorem}
The proof of the above theorem is given in Appendix \ref{AppA01}. This kind of decomposition was first introduced in \cite{GDecomp} and later extended in \cite{PUM,GGKL}.
\subsection{Global estimates}
 Let $u$ be a displacement in $\GU_{\e\d}$. We decompose $u$ as \eqref{EQ41D+}, then this along with \eqref{EQ34} leads to the following equalities:
\begin{equation}\label{EQ40}
\left.\begin{aligned}
-&\p_1\fU_3=\cR_{pq,2},\quad -\p_2\fU_3=-\cR_{pq,1},\\
& \fU_1=\Ga_{pq,1}-\cR_{pq,3}\ov{x}_{pq,2},\quad \fU_2=\Ga_{pq,2}+\cR_{pq,3}\ov{x}_{pq,1},\\
& \fU_3=\Ga_{pq,3}+\cR_{pq,1}\ov{x}_{pq,2}-\cR_{pq,2}\ov{x}_{pq,1},
\end{aligned}\right\}\quad \hbox{a.e. in } \o^H_{pq},\quad \forall  (p,q)\in \cK^*_\e.
\end{equation}
Below, we give the first global estimates of the Kirchhoff-Love part of a displacement $u\in \GU_\ed$.
\begin{proposition} We have
\begin{equation}\label{EQ46}
\begin{aligned}
& \big\|\nabla\fU_3\|_{L^2(\o_\d)}\leq {C\over \e^{1/2}\d} \GE(u),\quad && \big\|\fU_3\|_{H^1(\o_\d)}\leq {C\over \e^{1/2}\d} \GE(u),\\
&\|\nabla \fU_m\|_{L^2(\o_\d)}\leq  {C\over \d^{1/2}} \GE(u),\quad && \|\fU_m\|_{L^2(\o_\d)}\leq  {C\over \e^{1/2}} \GE(u).
\end{aligned}
\end{equation}   The constants do not depend on $\e$ and  $\d$. 
\end{proposition}
\begin{proof}  The estimate \eqref{EQ46}$_1$ is a consequence of  \eqref{EQ40}$_{1,2}$, \eqref{rem+}$_{2}$ and Lemma \ref{lemAp1}, from which and the boundary condition  we obtain  \eqref{EQ46}$_2$. The $2D$ Korn inequality and \eqref{rem+}$_1$ yield \eqref{EQ46}$_3$. The proof of \eqref{EQ46}$_4$ is given after the proof of  Lemma \ref{lem45}.
\end{proof}
As a consequence of the above theorem, 
\begin{lemma}\label{lem42} We have
	\begin{equation}\label{EQ47}
		\begin{aligned}
		& \sum_{p=1}^{N_\e-1}\sum_{q=0}^{N_\e-1}|\cR_{pq,2}-\cR_{p-1q,2}|^2\leq {C\over \e\d^2}\GE(u)^2,\qquad \sum_{q=1}^{N_\e-1}\sum_{p=0}^{N_\e-1}|\cR_{pq,1}-\cR_{pq-1,1}|^2\leq {C\over \e\d^2}\GE(u)^2,\\
		&\sum_{p=1}^{N_\e-1}\sum_{q=0}^{N_\e-1}|\cR_{pq,1}-\cR_{p-1q,1}|^2\leq {C\over \e^3}\GE(u)^2,\qquad \sum_{q=1}^{N_\e-1}\sum_{p=0}^{N_\e-1}|\cR_{pq,2}-\cR_{pq-1,2}|^2\leq {C\over \e^3}\GE(u)^2,\\
	&\sum_{p=1}^{N_\e-1}\sum_{q=0}^{N_\e-1}|\cR_{pq,3}-\cR_{p-1q,3}|^2\leq {C\over \e^3}\GE(u)^2,\qquad \sum_{q=1}^{N_\e-1}\sum_{p=0}^{N_\e-1}|\cR_{pq,3}-\cR_{pq-1,3}|^2\leq {C\over \e^3}\GE(u)^2
		\end{aligned}
	\end{equation} and
\begin{equation}\label{EQ47+}
		\begin{aligned}
	& \sum_{q=0}^{N_\e-1}\sum_{p=1}^{N_\e-1}\big|\Ga_{pq}-\Ga_{p-1q}-{\e\over 2}(\cR_{pq}+\cR_{p-1q})\land\Ge_1\big|^2\leq {C\over \e}\GE(u)^2,\\
	&  \sum_{p=0}^{N_\e-1}\sum_{q=1}^{N_\e-1}\big|\Ga_{pq}-\Ga_{pq-1}-{\e\over 2}(\cR_{pq}+\cR_{pq-1})\land\Ge_2\big|^2\leq {C\over \e}\GE(u)^2.
		\end{aligned}
	\end{equation}
	 The constants do not depend on $\e$, $\d$ and $L$. 
\end{lemma}
\begin{proof} 
First, we set
	\begin{equation}\label{wto}
 \wt{\O}_{pq}\doteq \wt{\o}_{pq} \X\fI_\d,\quad\forall\, (p,q)\in \cK^*_\e.
	\end{equation}
	
The proof is organized in the following steps.
	
{\it Step 1.}  We recall the following classical result:\\
For any $\psi\in H^1(\wt{\o}_{pq})$ and $\Psi\in H^1(\wt{\O}_{pq})^3$, we have
\begin{equation}\label{EQ43}
\begin{aligned}
&\|\psi\|_{L^2(\wt{\o}_{pq})}\leq C\big(\|\psi\|_{L^2(\o^H_{pq})}+\d\|\nabla \psi\|_{L^2(\wt{\o}_{pq})}\big),\\
&\|\Psi\|_{L^2(\wt{\O}_{pq})}\leq C\big(\|\Psi\|_{L^2(\O^H_{pq})}+\d\|e(\Psi)\|_{L^2(\wt{\O}_{pq})}\big).
\end{aligned}
\end{equation} The constant does not depend on $\e$ and $\d$.\\[1mm]
Below, we give a brief proof of these inequalities. Theorem 2.3 in \cite{GDecomp} gives a rigid displacement $\GR_{pq}=\GA_{pq}+\GB_{pq}\land \ov{x}_{pq}$, $\GA_{pq},\; \GB_{pq}\in \R^3$ such that
$$\|\Psi-\GR_{pq}\|_{L^2(\wt{\O}_{pq})}\leq C\d \|e(\Psi)\|_{L^2(\wt{\O}_{pq})}\big).$$ 
Thus
$$\|\Psi\|^2_{L^2(\wt{\O}_{pq})}\leq C\big(\|\GR_{pq}\|^2_{L^2(\wt{\O}_{pq})}+\d^2\|e(\Psi)\|^2_{L^2(\wt{\O}_{pq})}\big)\leq  C\big((\e+\d)^2|\GA_{pq}|^2+(\e+\d)^4|\GB_{pq}|^2 +\d^2\|e(\Psi)\|^2_{L^2(\wt{\O}_{pq})}\big).$$
Besides, we have
$$(\e-\d)^2|\GA_{pq}|^2+(\e-\d)^4|\GB_{pq}|^2\leq C \|\GR_{pq}\|^2_{L^2(\O^H_{pq})}\leq 2\big(\|\Psi-\GR_{pq}\|^2_{L^2(\O^H_{pq})}+\|\Psi\|^2_{L^2(\O^H_{pq})}\big).$$ This ends the proof of \eqref{EQ43}$_2$. We obtain \eqref{EQ43}$_1$ with the field $\Psi=\psi\Ge_1$ and thanks to \eqref{EQ43}$_2$.\\[1mm]
{\it Step 2.} A preliminary result.

Let $v$ and $w$ be  two rigid displacements
$$
v(x)=\Ga+\cR\land x=\begin{pmatrix} 
	\Ga_{1}+\cR_{2}x_3-\cR_{3}x_2\\
	\Ga_{2}+\cR_{3}\big(x_1-{\e/2}\big)-\cR_{1}x_3\\
	\Ga_{3}+\cR_{1}x_2-\cR_{2}\big(x_1-{\e/2}\big)
	\end{pmatrix}\;\text{a.e. in } \Big(-{\d\over 2},\e\Big)\X\fI_\e\X \fI_\d,\quad (\Ga,\cR)\in\R^3\X \R^3
$$ and
$$
w(x)=\Ga^{'}+\cR^{'}\land x=\begin{pmatrix} 
	\Ga^{'}_{1}+\cR^{'}_{2}x_3-\cR^{'}_{3}x_2\\
	\Ga^{'}_{2}+\cR^{'}_{3}\big(x_1+ {\e/2}\big)-\cR^{'}_{1}x_3\\
	\Ga^{'}_{3}+\cR^{'}_{1}x_2-\cR^{'}_{2}\big(x_1+ {\e/2}\big)
	\end{pmatrix}\;\text{a.e. in } \Big(-\e,{\d\over 2}\Big)\X\fI_\e\X\fI_\d,\quad (\Ga^{'},\cR^{'})\in\R^3\X \R^3
$$ In $\ds \cO_{\ed}\doteq\fI_\d\X\fI_\e\X \fI_\d$ the components of $\Gu=v-w$ are$$
\begin{aligned}
\Gu_1&=\Ga_{1}-\Ga^{'}_{1}+(\cR_{2}-\cR^{'}_{2})x_3-(\cR_{3}-\cR^{'}_{3})x_2,\\
\Gu_2&=\Ga_{2}-\Ga^{'}_{2}-{\e\over 2}\big(\cR_{3}+\cR^{'}_{3}\Big)+(\cR_{3}-\cR^{'}_{3})x_1-(\cR_{1}-\cR^{'}_{1})x_3,\\
\Gu_3&=\Ga_{3}-\Ga^{'}_{3}+{\e\over 2}\big(\cR_{2}+\cR^{'}_{2}\big)+(\cR_{1}-\cR^{'}_{1})x_2-(\cR_{2}-\cR^{'}_{2})x_1.
\end{aligned}\qquad \hbox{in }\; \cO_{\ed}.
$$ The  $L^2$ norm of the components of $v-w$ in $\cO_\ed$ are
\begin{equation}\label{EQ471}
\begin{aligned}
\|\Gu\|^2_{L^2(\cO_{\ed})}=\ed^2\Big(|\Ga_{1}-\Ga^{'}_{1}|^2+\Big|\Ga_{2}-\Ga^{'}_{2}-{\e\over 2}\big(\cR_{3}+\cR^{'}_{3}\Big)\Big|^2+\Big|\Ga_{3}-\Ga^{'}_{3}+{\e\over 2}\big(\cR_{2}+\cR^{'}_{2}\big)\Big|^2\\
+ {\d^2\over 6}|\cR_{2}-\cR^{'}_{2}|^2+\Big({\e^2\over 12}+{\d^2\over 12}\Big)|\cR_{3}-\cR^{'}_{3}|^2+ \Big({\d^2\over 12}+{\e^2\over 12}\Big)|\cR_{1}-\cR^{'}_{1}|^2\Big).
\end{aligned}
\end{equation}
 {\it Step 3.} We prove \eqref{EQ47}-\eqref{EQ47+}.\\
We apply \eqref{EQ43}$_2$ with the displacements $v=u$ and $w_{pq}=\Ga_{pq}+\cR_{pq}\land \ov{x}_{pq}$ in  $\wt{\O}_{pq}$. Then, with $v=u$ and $w_{p-1q}=\Ga_{p-1q}+\cR_{p-1q}\land \ov{x}_{p-1q}$ in  $\wt{\O}_{p-1q}$. This allows to estimate $w_{pq}-w_{p-1q}$ in $\wt{\O}_{pq}\cap \wt{\O}_{p-1q}$  using \eqref{EQ471}. We obtain \eqref{EQ47}$_{1,3,5}$- \eqref{EQ47+}$_{1}$.\\
 Similarly, proceeding as above we estimate  $w_{pq}-w_{pq-1}$ in $\wt{\O}_{pq}\cap \wt{\O}_{pq-1}$, which along with \eqref{EQ471} give \eqref{EQ47}$_{2,4,6}$-\eqref{EQ47+}$_{2}$.
 This completes the proof.
\end{proof}
\section{The tools: Global fields using \texorpdfstring{$\cQ_1$}{Q1}-interpolations}\label{S05}
In this section, we construct several global displacement fields using $\mathcal{Q}_1$-interpolation. These fields serve as tools to obtain the unfolded limits of the macroscopic Kirchhoff--Love displacements $\fU$. 

More precisely, we first define $\cU^\diamond, \cR^\diamond \in H^1(\omega_\delta)^3$ to obtain sufficient regularity for the macroscopic limits of $\fU$. Then we define $\cU, \cR \in H^1(\omega_\delta)^3$ and $\cU_\alpha^\bt, \cU_\alpha^{\bt\bt} \in H^1(\omega_\delta)$ to derive the exact expressions of the unfolded limits of $\fU$. Details on the corresponding convergence results are provided in Lemma~\ref{lem61}.
\subsection{The fields \texorpdfstring{$\cU^\dia$ and $\cR^\dia$}{U and R (diamond)}} \label{SS42}

 Let $u$ be in $\GU_{\e\d}$, from the previous subsection, we decompose $u$ as \eqref{EQ41D+}.\\
 First, we need to define $\Ga_{pq}$ and $\cR_{pq}$ for all $(p,q)\in \ov{\cK}_\e\setminus \cK^*_\e$. We  set \footnote{$ $ We use \( \cR_{pq} \) and \( \cR_{p,q} \) interchangeably for \( (p,q) \in \overline{\mathcal{K}}_\varepsilon \), as they refer to the same field. Similarly, \( \Ga_{pq} \) and \( \Ga_{p,q} \) are used to denote the same quantity.}
$$
\begin{aligned}
&\Ga_{-1,q}=\Ga_{0,q},\quad \cR_{-1,q}=\cR_{0,q},\quad \Ga_{N_\e,q}=\Ga_{N_\e-1,q},\quad \cR_{N_\e,q}=\cR_{N_\e-1,q},\qquad q\in\{0,\ldots,N_\e-1\},\\
&\Ga_{p,-1}=\Ga_{p,0},\quad \cR_{p,-1}=\cR_{p,0},\quad \Ga_{p,N_\e}=\Ga_{p,N_\e-1},\quad \cR_{p,N_\e}=\cR_{p,N_\e-1},\qquad p\in\{-1,\ldots,N_\e\}.
\end{aligned}
$$ If the displacement $u$ vanishes on $\Gamma_\d$, then we set
$$
\begin{aligned}
&\Ga_{-1,q}=\cR_{-1,q}=\Ga_{0,q}=\cR_{0,q}=0\qquad  q\in\{0,\ldots, n_\e-1\},\\
&\Ga_{p,-1}=\cR_{p,-1}=\Ga_{p,0}=\cR_{p,0}=0\qquad p\in\{-1,\ldots,n_\e-1\}.
\end{aligned}
$$
In order to complete the asymptotic behavior  of the function $\fU$, we introduce the fields $\cR^\dia,\;\cU^\dia$ belonging to $H^1(\o_\d)^3$ by
$$
\cR^\dia(p\e,q\e)=\cR_{pq},\quad \cU^\dia(p\e,q\e)=\Ga_{pq}\qquad (p,q)\in \ov{\cK_\e}
$$ and then, in $\ov{\o}_{pq}$, $(p,q)\in \ov{\cK_\e}$, the fields $\cR^\dia$, $\cU^\dia$ are the $\cQ_1$ interpolation of their values at the vertices of this square. 
By construction, the field  $\cR^\dia$, $\cU^\dia$ belong to $H^1\big((-\e,L+\e)^2\big)^3$ and satisfy the following boundary conditions:
$$\cU^\dia=\cR^\dia=0\qquad\hbox{on}\;\; \gamma.$$
In the lemma below, we give the estimates of the restrictions of these fields to $\o_\d$.
\begin{lemma}\label{lem45}  We have
\begin{equation}\label{EQ49}
\begin{aligned}
&\|\nabla \cR^\dia\|_{L^2(\o_\d)}+ \|\cR^\dia\|_{L^2(\o_\d)}\leq {C\over \e^{1/2}\d}\GE(u),\\
&\|\nabla \cR^\dia_3\|_{L^2(\o_\d)}\leq {C\over \e^{3/2}}\GE(u),\quad  \|\cR^\dia_3\|_{L^2(\o_\d)}\leq  {C\over \e^{1/2}}\GE(u),\\
 &  \|\p_1\cR^\dia_1\|_{L^2(\o_\d)}+\|\p_2\cR^\dia_2\|_{L^2(\o_\d)} \leq {C\over \e^{3/2}}\GE(u)
 \end{aligned}
\end{equation} 
and 
 \begin{equation}\label{EQ49+}
\begin{aligned}
  &\big\|\partial_1\cU^\dia_1\big\|_{L^2(\o_\d)}+\big\|\partial_2\cU^\dia_2\big\|_{L^2(\o_\d)}\leq {C\over \e^{1/2}}\GE(u),\\
  &\big\|\partial_1\cU^\dia_2-\cR^\dia_3\big\|_{L^2(\o_\d)}+\big\|\partial_2\cU^\dia_1+\cR^\dia_3\big\|_{L^2(\o_\d)}\leq {C\over \e^{1/2}}\GE(u),\\
    & \big\|\partial_1\cU^\dia_3+\cR^\dia_2\big\|_{L^2(\o_\d)}+\big\|\partial_2\cU^\dia_3-\cR^\dia_1\big\|_{L^2(\o_\d)}\leq C{\e^{1/2}\over \d}\GE(u).
  \end{aligned}
\end{equation} 
Moreover, we have
\begin{equation}\label{EQ413}
\|\cU^\dia_\a\|_{H^1(\o_\d)}\leq {C\over \e^{1/2}}\GE(u),\qquad \|\cU^\dia_3\|_{H^1(\o_\d)}\leq {C\over \e^{1/2}\d}\GE(u). 
\end{equation}
 The constants do not depend on $\e$ and $\d$. 
\end{lemma}

\begin{proof}
	The estimates \eqref{EQ49}$_{1,3,5,6}$ are the immediate consequences of \eqref{EQ47}$_{1,2,3,4,5,6}$ and the $\cQ_1$-character of $\cR^\dia$. Then, the boundary condition satisfied by $\cR^\dia$ leads to \eqref{EQ49}$_{2}$.\\[1mm]
Now, in $\o_{pq}$, for any $(z_1,z_2)\in (0,\e)^2.$ we have
$$
\begin{aligned}
\cU^\dia(p\e+z_1,q\e+z_2)&=\Ga_{pq}{(\e-z_1)(\e-z_2)\over \e^2}+\Ga_{p+1q}{z_1(\e-z_2)\over \e^2}+\Ga_{pq+1}{(\e-z_1)z_2\over \e^2}+\Ga_{p+1q+1}{z_1 z_2\over \e^2},\\
\p_1\cU^\dia(p\e+z_1,q\e+z_2)&={\Ga_{p+1q}-\Ga_{pq}\over \e}{\e-z_2\over \e}+{\Ga_{p+1q+1}-\Ga_{pq+1}\over \e}{z_2\over \e},\\
\cR^\dia(p\e+z_1,q\e+z_2)&=\cR_{pq}{(\e-z_1)(\e-z_2)\over \e^2}+\cR_{p+1q}{z_1(\e-z_2)\over \e^2}+\cR_{pq+1}{(\e-z_1)z_2\over \e^2}+\cR_{p+1q+1}{z_1 z_2\over \e^2},\\
&=\Big({1\over 2}\big(\cR_{p+1q}+\cR_{pq}\big)+{z_1-\e/2\over \e}\big(\cR_{p+1q}-\cR_{pq}\big)\Big){\e-z_2\over \e}\\
&+\Big({1\over 2}\big(\cR_{p+1q+1}+\cR_{pq+1}\big)+{z_1-\e/2\over \e}\big(\cR_{p+1q+1}-\cR_{pq+1}\big)\Big){z_2\over \e}.
\end{aligned}
$$ Then, we obtain
$$
\begin{aligned}
\big(\partial_1\cU^\dia-\cR^\dia\land \Ge_1\big)(p\e+z_1,q\e+z_2)=&{1\over \e}\Big(\Ga_{p+1q}-\Ga_{pq}-{\e\over 2}(\cR_{p+1q}+\cR_{pq})\land\Ge_1\Big){\e-z_2\over \e} \\
+&{1\over \e}\Big(\Ga_{p+1q+1}-\Ga_{pq+1}-{\e\over 2}(\cR_{p+1q+1}+\cR_{pq+1})\land\Ge_1\Big){z_2\over \e} \\
+\big(\cR_{p+1q}-\cR_{pq}\big)\land & \Ge_1{2z_1-\e\over 2\e}{\e-z_2\over \e}+\big(\cR_{p+1q+1}-\cR_{pq+1}\big)\land\Ge_1{2z_1-\e\over 2\e}{z_2\over \e}.
\end{aligned}
$$ Thanks to the estimates \eqref{EQ47} this gives \eqref{EQ49+}$_{1,3,5}$. Similarly we prove  \eqref{EQ49+}$_{2,4,6}$.\\
The estimate \eqref{EQ413}$_2$ is direct consequence of \eqref{EQ49}$_{2}$-\eqref{EQ49+}$_{5,6}$ and Poincar\'e inequality. Estimates \eqref{EQ49+}$_{1,2,3,4}$ lead to
$$\|e_{\alpha\beta}(\cU^\dia)\|_{L^2(\o_\d)}\leq {C\over \e^{1/2}}\GE(u)$$
which along with the  $2$D Korn's inequality give \eqref{EQ413}$_1$ from which and \eqref{EQ49+}$_{3,4}$ we obtain \eqref{EQ49}$_{4}$.

Finally, from \eqref{EQ413}$_1$ and \eqref{EQ49+}$_3$ we deduce \eqref{EQ49}$_4$. 
 This ends the proof of the lemma.
\end{proof}

\begin{proof}[{\bf Proof of the estimate  \eqref{EQ46}$_4$}] From \eqref{EQ49}$_4$-\eqref{EQ413}$_1$ and the $\cQ_1$ character of $\cU^\dia_\a$, $\cR^\dia_3$ we obtain
\begin{equation}\label{EQ417}
\sum_{(p,q)\in \cK^*_\e} |\Ga_{pq,\a}|^2\e^2 \leq {C\over \e}\GE(u)^2,\qquad \sum_{(p,q)\in \cK^*_\e} |\cR_{pq,3}|^2\e^2\leq {C\over \e}\GE(u)^2.
\end{equation}
Therefore, we get
 $$\sum_{(p,q)\in \cK^*_\e}\|\fU_\a\|^2_{L^2(\o^H_{pq})}\leq {C\over \e}\GE(u)^2$$
 which along with \eqref{EQ46}$_3$ and inequality \eqref{EQ43}$_1$, give the required result \eqref{EQ46}$_4$.
\end{proof}
\begin{lemma}\label{lem47} We have
\begin{equation}\label{N3}
\begin{aligned}
&\big\|\p_1\fU_3+\cR^\dia_2\big\|_{L^2(\o_\d)}+\big\|\p_2\fU_3-\cR^\dia_1\big\|_{L^2(\o_\d)}\leq C{\e^{1/2}\over \d}\GE(u),\\
&\|\fU_\a-\cU^\dia_\a\|_{L^2(\o_\d)}\leq C{\e^{1/2}}\GE(u),\\
& \|\fU_3-\cU^\dia_3\|_{H^1(\o_\d)}\leq  C{\e^{1/2}\over \d}\GE(u).
\end{aligned}
\end{equation} The constants do not depend on $\e$ and $\d$.
\end{lemma}
\begin{proof} First, due to the $\cQ_1$ character of $\cR^\dia$ we have
\begin{equation}\label{LOL1}
	\sum_{(p,q)\in \cK^*_\e}\big\|\cR^\dia_\a-\cR_{pq,\a}\big\|^2_{L^2(\o_{pq})} \leq C\e^2\|\nabla \cR^\dia_\a\|^2_{L^2(\o)}.
\end{equation}
 The equality \eqref{EQ40}$_{1}$ and estimate \eqref{EQ43}$_1$ yield
$$
\sum_{(p,q)\in \cK^*_\e}\big\|\p_1\fU_3+\cR^\dia_2\big\|^2_{L^2(\wt\o_{pq})} \leq C\sum_{(p,q)\in \cK^*_\e}\Big\{\|\cR^\dia_2-\cR_{pq,2}\|^2_{L^2(\o^H_{pq})}+\d^2\big(\|\nabla \cR^\dia_2\|^2_{L^2(\wt\o_{pq})}+\|\nabla (\p_1\fU_3)\|^2_{L^2(\wt\o_{pq})}\big)\Big\}.
$$
Then,  \eqref{EQ49}$_{1}$ together with \eqref{rem+}$_2$ give \eqref{N3}$_1$. Similarly, we show \eqref{N3}$_2$.\\ 
From the equalities \eqref{EQ40}$_{3,4}$ and estimate \eqref{EQ417}$_1$ we obtain
$$\sum_{(p,q)\in \cK^*_\e}\big\|\fU_\a-\Ga_{pq,\a}\big\|^2_{L^2(\o^H_{pq})} \leq C\e\GE(u)^2.$$
 The $\cQ_1$-character of $\cU^\dia_\a$ together with estimate \eqref{EQ413}$_{1}$ give
$$
\sum_{(p,q)\in \cK^*_\e}\big\|\cU^\dia_\a-\Ga_{pq,\a}\big\|^2_{L^2(\o_{pq})} \leq C\e^2\|\nabla \cU^\dia_\a\|^2_{L^2(\o)}\leq C\e \GE(u)^2.
$$
The above inequalities along with \eqref{EQ43}$_1$, \eqref{EQ413}$_{1}$ and 	\eqref{EQ46}$_3$ give \eqref{N3}$_3$.\\[1mm]
Now, equality \eqref{EQ40}$_5$ and estimates \eqref{EQ49}$_2$  give
$$\sum_{(p,q)\in \cK^*_\e}\big\|\fU_3-\Ga_{pq,3}\big\|^2_{L^2(\o^H_{pq})} \leq C{\e\over \d^2}\GE(u)^2.$$
Besides, due to  the $\cQ_1$-character of $\cU^\dia_3$ together with estimate \eqref{EQ413}$_{2}$ give
\begin{equation}\label{EQ418}
\sum_{(p,q)\in \cK^*_\e}\big\|\cU^\dia_3-\Ga_{pq,3}\big\|^2_{L^2(\o_{pq})} \leq C\e^2\|\nabla \cU^\dia_3\|^2_{L^2(\o)}\leq C{\e\over \d^2}\GE(u)^2.
\end{equation}
Using the above estimates along with \eqref{EQ43}$_1$, \eqref{EQ413}$_{2}$ and \eqref{EQ46}$_2$ we obtain the estimate of the $L^2$ norm of $\fU_3-\cU^\dia_3$. The estimate of the gradient of this difference is a consequence of \eqref{EQ49+}$_{5,6}$ and \eqref{N3}$_{1,2}$.
\end{proof} 
\subsection{The fields \texorpdfstring{$\cU$ and $\cR$}{U and R}}\label{SS43}
Set
\begin{equation}\label{Rect}
	\begin{aligned}
		& R^1_{pq}\,:\, (p\e,q\e)+\fI_{\e-\d}\X\fI_\d,\qquad &&(p,q)\in \cK^{(1)}_\e,\\
		& R^2_{pq}\,:\, (p\e,q\e)+\fI_\d\X\fI_{\e-\d},\qquad &&(p,q)\in \cK^{(2)}_\e,\\
		& R^3_{pq}\,:\, (p\e,q\e)+\fI_\d^2,\qquad &&(p,q)\in \cK_\e.
	\end{aligned}
\end{equation}
Let $\phi_{pq}$, $(p,q)\in \ov{\cK}_\e$ be real numbers. To this family of real numbers we associate  a function $\phi$ defined as follows:
\begin{equation}\label{Dphi}
\begin{aligned}
		\text{in $\o_{pq}$:}\quad \phi(p\e+z_1,q\e+z_2)& =\phi_{pq},\quad (z_1,z_2)\in \fI_{\e-\d}^2\\
		\text{in $R^1_{pq}$:}\quad \phi(p\e+z_1,q\e+z_2)&=\phi_{pq}{\d+ 2z_2\over 2\d}+\phi_{pq-1}{\d-2z_2\over 2\d}, \quad (z_1,z_2)\in \fI_{\e-\d}\X\fI_\d,\\
		\text{in $R^2_{pq}$:}\quad \phi(p\e+z_1,q\e+z_2)&=\phi_{pq}{\d+ 2z_1\over 2\d}+\phi_{p-1q}{\d-2z_1\over 2\d},\quad (z_1,z_2)\in \fI_{\d}\X\fI_{\e-\d},\\
		\text{in $R^3_{pq}$:}\quad \phi(p\e+z_1,q\e+z_2)&=\phi_{pq}{\d+ 2z_1\over 2\d}{\d+ 2z_2\over 2\d}+\phi_{pq-1}{\d+ 2z_1\over 2\d}{\d-2z_2\over 2\d}\\
		&\;\;+\phi_{p-1q}{\d- 2z_1\over 2\d}{\d+ 2z_2\over 2\d}+\phi_{p-1q-1}{\d- 2z_1\over 2\d}{\d-2z_2\over 2\d},\quad (z_1,z_2)\in \fI^2_{\d}.
	\end{aligned}
\end{equation}
Now, we are going to define the global fields   $\cU\in H^1(\o_\d)^3$ and $\cR\in H^1(\o_\d)^3$, for that we set
\begin{equation}\label{H01}
\begin{aligned}
	&\cR(p\e+z_1,q\e+z_2)=\cR_{pq},\\
	&\cU(p\e+z_1,q\e+z_2)=\Ga_{pq}+ \cR_{pq}\land \begin{pmatrix} \ds z_1 \\  \ds z_2 \\  0 \end{pmatrix},\\
\end{aligned} \qquad \hbox{for all } z'=(z_1,z_2)\in \fI_{\e-\d}^2,\quad (p,q)\in \ov{\cK}_\e.
\end{equation}
\begin{figure}[ht]
	\centering
	\scalebox{0.88}{ 
		\begin{subfigure}[t]{0.48\textwidth}
			\centering
			\begin{tikzpicture}[scale=0.9, line cap=round, line join=round]
				\tikzset{
					hard/.style  ={draw=black, fill=RoyalBlue!70,  line width=1.6pt},
					softV/.style ={draw=black, fill=Crimson!60,    line width=1.6pt}, 
					softH/.style ={draw=black, fill=SeaGreen!60,   line width=1.6pt}, 
					pad/.style   ={draw=black, fill=Gold!85!yellow,line width=1.6pt}, 
					border/.style={draw=black, line width=1.6pt},
					jnode/.style ={circle, draw=black, fill=black, line width=1.6pt, minimum size=7pt, inner sep=0pt},
					labelbox/.style={font=\bfseries\footnotesize, text=black,
						fill=white, fill opacity=.9, text opacity=1,
						rounded corners=2pt, inner sep=1.5pt,
						draw=black, line width=.25pt}
				}
				
				\filldraw[hard]  (2,2) rectangle (6,6);  
				\filldraw[softV] (1,2) rectangle (2,6);  
				\filldraw[softH] (2,1) rectangle (6,2);  
				\filldraw[pad]   (1,1) rectangle (2,2);  
				
				\foreach \x in {1,2,6}   \draw[border] (\x,1)--(\x,6);
				\foreach \y in {1,2,6}   \draw[border] (1,\y)--(6,\y);
				
				\foreach \X in {1,2,6}{
					\foreach \Y in {1,2,6}{
						\node[jnode] at (\X,\Y) {};
					}
				}
				
				\node[labelbox] at (4.0,4.0) {$\omega^H_{pq}$};
				\node[labelbox] at (1.5,4.0) {$R^{2}_{pq}$}; 
				\node[labelbox] at (4.0,1.5) {$R^{1}_{pq}$}; 
				\node[labelbox] at (1.5,1.5) {$R^{3}_{pq}$}; 
			\end{tikzpicture}
			
			\caption{$\widetilde{\omega}_{pq}$}
		\end{subfigure}
		\hfill
		\begin{subfigure}[t]{0.48\textwidth}
			\centering
			\begin{tikzpicture}[scale=0.30, line cap=round, line join=round]
				\tikzset{
					hard/.style  ={draw=black, fill=RoyalBlue!70,  line width=1.6pt}, 
					softV/.style ={draw=black, fill=Crimson!60,    line width=1.6pt}, 
					softH/.style ={draw=black, fill=SeaGreen!60,   line width=1.6pt}, 
					pad/.style   ={draw=black, fill=Gold!85!yellow,line width=1.6pt}, 
					border/.style={draw=black, line width=1.6pt},
					labelbox/.style={font=\bfseries\footnotesize, text=black,
						fill=white, fill opacity=.9, text opacity=1,
						rounded corners=2pt, inner sep=1.5pt,
						draw=black, line width=.25pt}
				}
				
				\foreach \x in {2,7,12}{
					\foreach \y in {2,7,12}{
						\filldraw[hard] (\x,\y) rectangle (\x+4,\y+4);
					}
				}
				
				\foreach \xa/\xb in {2/6,7/11,12/16}{
					\filldraw[softH] (\xa, 1) rectangle (\xb, 2);
					\filldraw[softH] (\xa, 6) rectangle (\xb, 7);
					\filldraw[softH] (\xa,11) rectangle (\xb,12);
					\filldraw[softH] (\xa,16) rectangle (\xb,17);
				}
				
				\foreach \ya/\yb in {2/6,7/11,12/16}{
					\filldraw[softV] ( 1,\ya) rectangle ( 2,\yb);
					\filldraw[softV] ( 6,\ya) rectangle ( 7,\yb);
					\filldraw[softV] (11,\ya) rectangle (12,\yb);
					\filldraw[softV] (16,\ya) rectangle (17,\yb);
				}
				
				\foreach \X in {1,6,11,16}{
					\foreach \Y in {1,6,11,16}{
						\filldraw[pad] (\X,\Y) rectangle (\X+1,\Y+1);
					}
				}
				
				\foreach \x in {1,2,6,7,11,12,16,17} \draw[border] (\x,1)--(\x,17);
				\foreach \y in {1,2,6,7,11,12,16,17} \draw[border] (1,\y)--(17,\y);
				
				\foreach \X in {2,6,7,11,12,16}{
					\foreach \Y in {2,6,7,11,12,16}{
						\fill (\X,\Y) circle[radius=0.28];
					}
				}
				
				\node[labelbox] at ( 9, 4) {$\omega^H_{pq-1}$};
				\node[labelbox] at ( 4, 9) {$\omega^H_{p-1q}$};
				\node[labelbox] at ( 9, 9) {$\omega^H_{pq}$};
				\node[labelbox] at (14, 9) {$\omega^H_{p+1q}$};
				\node[labelbox] at ( 9,14) {$\omega^H_{pq+1}$};
			\end{tikzpicture}
			\caption{Part of $\omega_\delta$}
	\end{subfigure}}
	
	\caption{Covering a part of $\omega_\delta$.}
	\label{FigXX}
\end{figure}
the field $\cR$ is defined from the family of real numbers $\cR_{pq}$, $(p,q)\in \ov{\cK}_\e$.
$\cU$ is defined as the $\cQ_1$-interpolate of its values at the vertices of the above  rectangles or squares \eqref{Rect}.
At this point, these fields are defined in $\ds\Big(-{\d\over 2},L+{\d\over 2}\Big)^2$ and the above constructions give $\cR,\; \cU, \in H^1(\o_\d)^3$.  We have the following expressions of $\cU$:
{\footnotesize$$
 \begin{aligned}
\hbox{in }R^1_{pq} :\quad &\cU(p\e+z_1,q\e+z_2)\\
  &=  \Big(\Ga_{pq}+{1\over 2}\cR_{pq}\land \begin{pmatrix} \ds 2z_1 \\  \ds -\e+\d \\  0 \end{pmatrix}\Big){\d+ 2z_2\over 2\d}+\Big(\Ga_{pq-1}+{1\over 2}\cR_{pq-1}\land \begin{pmatrix} \ds 2z_1 \\  \ds \e-\d\\  0 \end{pmatrix}\Big){\d-2z_2\over 2\d} ,\quad (z_1,z_2)\in \fI_{\e-\d}\X\fI_\d,\\
\hbox{in }R^2_{pq} :\quad &\cU(p\e+z_1,q\e+z_2)\\
  &= \Big(\Ga_{pq}+{1\over 2}\cR_{pq}\land \begin{pmatrix} \ds -\e+\d \\  \ds 2z_2 \\  0 \end{pmatrix}\Big){\d+ 2z_1\over 2\d}+\Big(\Ga_{p-1q}+{1\over 2}\cR_{p-1q}\land \begin{pmatrix} \ds \e-\d \\  \ds 2z_2 \\  0 \end{pmatrix}\Big){\d-2z_1\over 2\d} ,\quad (z_1,z_2)\in \fI_{\d}\X\fI_{\e-\d},\\
\hbox{in }R^3_{pq} :\quad &\cU(p\e+z_1,q\e+z_2)\\
  &=\Big(\Ga_{pq}+{1\over 2}\cR_{pq}\land \begin{pmatrix} \ds -\e+\d\\  \ds -\e+\d \\  0 \end{pmatrix} \Big){\d+ 2z_1\over 2\d}{\d+ 2z_2\over 2\d}+\Big(\Ga_{pq-1}+{1\over 2}\cR_{pq-1}\land \begin{pmatrix} \ds -\e+\d \\  \ds \e-\d\\  0 \end{pmatrix}\Big){\d+ 2z_1\over 2\d}{\d-2z_2\over 2\d}\\
		+\Big(\Ga_{p-1q}+{1\over 2}  \cR_{p-1q} & \land \begin{pmatrix} \ds \e -\d\\  \ds -\e+\d \\  0 \end{pmatrix} \Big){\d- 2z_1\over 2\d}{\d+ 2z_2\over 2\d}+\Big(\Ga_{p-1q-1}+{1\over 2}\cR_{p-1q-1}\land \begin{pmatrix} \ds \e-\d\\  \ds \e-\d \\  0 \end{pmatrix} \Big){\d- 2z_1\over 2\d}{\d-2z_2\over 2\d}, \quad (z_1,z_2)\in \fI^2_{\d}.
\end{aligned}
$$}
Observe that by construction, the fields $\cU$ and $\cR$ satisfy
$$
\partial_\alpha\cU=\cR\land\Ge_\alpha\quad\text{a.e. in $\o^H_{pq}$}\hbox{ for all } (p,q)\in \cK_\e^*\quad \hbox{and}\quad  \cU=\cR=0 \qquad \hbox{a.e. in }\gamma_\d.
$$
Below, we give estimates for $\cU,\;\cR\in H^1(\o_\d)^3$.
\begin{lemma}\label{lem51} We have
\begin{equation}\label{EQ414}
\begin{aligned}
 &  \|\nabla \cR_\a\|_{L^2(\o_\d)}\leq {C\over \d^{3/2}}\GE(u),\quad \|\cR_\a\|_{L^2(\o_\d)}\leq {C\over \e^{1/2}\d}\GE(u),\\
&  \|\p_1\cR_1\|_{L^2(\o_\d)}+\|\p_2\cR_2\|_{L^2(\o_\d)} \leq {C\over \e\d^{1/2}}\GE(u),\\
 &\big\|\partial_1\cU_3+\cR_2\big\|_{L^2(\o_\d)}+\big\|\partial_2\cU_3-\cR_1\big\|_{L^2(\o_\d)}\leq {C\over \d^{1/2}}\GE(u),\\
 &\|\nabla\cR_3\|_{L^2(\o_\d)} \leq {C\over \e\d^{1/2}}\GE(u),\quad  \| \cR_3 \|_{L^2(\o_\d)}\leq {C\over \e^{1/2}}\GE(u),\\
 & \|\nabla \cU_\alpha\|_{L^2(\o_\d)}\leq {C\over \d^{1/2}}\GE(u),\quad \|\cU_\alpha\|_{L^2(\o_\d)}\leq {C\over \e^{1/2}}\GE(u),\quad \|\cU_3\|_{H^1(\o_\d)}\leq {C\over \e^{1/2}\d}\GE(u)
\end{aligned}
\end{equation} and also
\begin{equation}\label{EQ415}
\begin{aligned}
& \|\cR_\a-\cR^\dia_\a\|_{L^2(\o_\d)}\leq C{ \e^{1/2}\over\d}\GE(u),\qquad   \|\cR_3-\cR^\dia_3\|_{L^2(\o_\d)}\leq {C \over \e^{1/2}}\GE(u),\\
&  \|\cU_\a-\cU^\dia_\a\|_{L^2(\o_\d)}\leq C\e^{1/2}\GE(u),\qquad   \|\cU_3-\cU^\dia_3\|_{L^2(\o_\d)}\leq C{\sqrt{\e}\over\d}\GE(u).
\end{aligned}
\end{equation}
 The constants do not depend on $\e$ and $\d$. 
\end{lemma}
\begin{proof}
	The proof is divided into four steps given below.
	
 {\bf Step 1.}  We prove \eqref{EQ414}$_{1,3,4,7}$.\\
In the rectangles $R^2_{pq}$ for $(p,q)\in \cK_\e^{(2)}$, we have
$$
\p_1\cR(p\e+z_1,q\e+z_2)={\cR_{pq}-\cR_{p-1q} \over \d},\quad\forall\,(z_1,z_2)\in\fI_\d\X\Big({\d\over 2},\e-{\d\over 2}\Big).
$$ 
So, we have
\begin{equation}\label{EQ414+}
\|\partial_1\cR_i\|^2_{L^2(R^2_{pq})}\leq {\e\over \d}|\cR_{pq,i}-\cR_{p-1q,i}|^2.
\end{equation}
Similarly, in the rectangles $R^1_{pq}$ for $(p,q)\in\cK_\e^{(1)}$, we have
\begin{equation}\label{EQ414++}
\|\partial_2\cR_i\|^2_{L^2(R^1_{pq})}\leq {\e\over \d}|\cR_{pq,i}-\cR_{pq-1,i}|^2.
\end{equation}
In the squares $R^3_{pq}$ for $(p,q)\in\cK_\e$, we have
\begin{equation}\label{EQ415+}
\begin{aligned}
\p_2\cR(p\e+z_1,q\e+z_2)=&{\cR_{p-1q}-\cR_{p-1q-1}\over \d}{{\d}-2z_1\over 2\d} +{\cR_{pq}-\cR_{pq-1}\over \d}{{\d}+2z_1\over 2\d}\\
\end{aligned} \quad  \forall\,(z_1,z_2)\in\fI_\d^2.
\end{equation}
Similarly, we compute $\partial_2\cR_i$. We therefore obtain 
\begin{equation}\label{EQ416}
\begin{aligned}
	&\|\partial_1\cR_i\|^2_{L^2(R^3_{pq})}\leq  C\big(|\cR_{pq-1,i}-\cR_{p-1q-1,i}|^2+|\cR_{pq,i}-\cR_{p-1q,i}|^2\Big)\\ 
	\hbox{and}\;\; \;	
	&\|\partial_2\cR_i\|^2_{L^2(R^3_{pq})}\leq  C\big(|\cR_{p-1q,i}-\cR_{p-1q-1,i}|^2+|\cR_{pq,i}-\cR_{pq-1,i}|^2\big).
\end{aligned}
\end{equation} 
Then, using the above results \eqref{EQ414+}--\eqref{EQ416}, we get
\begin{equation}\label{INE01}
\begin{aligned}
	&\|\nabla \cR_i\|^2_{L^2(\o_\d)}\leq {C\e\over \d}\Big(\sum_{p=0}^{N_\e-1}\sum_{q=0}^{N_\e}|\cR_{pq,i}-\cR_{p-1q,i}|^2+ \sum_{q=0}^{N_\e}\sum_{p=0}^{N_\e-1}|\cR_{pq,i}-\cR_{pq-1,i}|^2\Big).
\end{aligned}
\end{equation}
The above estimate with \eqref{EQ47}$_3$, \eqref{EQ47}$_4$ and \eqref{EQ47}$_{5,6}$ give \eqref{EQ414}$_{3,4,7}$ respectively. The estimate \eqref{EQ414}$_1$ is a consequence of \eqref{EQ47}$_{1,2,3,4}$ and \eqref{INE01}.\\[1mm]
{\bf Step 2.} We prove \eqref{EQ414}$_{2,8}$ and \eqref{EQ415}$_{1,2}$

Observe that, since  $\cR$ is constant and equal to $\cR_{pq}$ in $\o^H_{pq}$,  due to \eqref{EQ43}$_1$ and \eqref{EQ414}$_1$ we have 
\begin{equation}\label{LOL2}
\begin{aligned}
	\sum_{(p,q)\in \cK^*_\e}\big\|\cR_i -\cR_{pq,i}\big\|^2_{L^2(\wt\o_{pq})} & \leq  C\d^2 \|\nabla \cR_i\|^2_{L^2(\o_\d)}\leq {C\over \d}\GE(u)^2.
\end{aligned}
\end{equation} 
The above  estimate with \eqref{LOL1} give \eqref{EQ415}$_{1,2}$. Then, \eqref{EQ49}$_{2,4}$ and \eqref{EQ415}$_{1,2}$ give \eqref{EQ414}$_{2,8}$. \\[1mm]
{\bf Step 3.} We prove \eqref{EQ414}$_{5,6,7}$.\\ 
Using the definition of $\cU$ and $\cR$, we have $\partial_1\cU-\cR\land\Ge_1$ vanishes in $\o^H_{pq}$ for $(p,q)\in\cK_\e^*$ and  in rectangle $R^1_{pq}$ for $(p,q)\in\cK_\e^{(1)}$. In $R^2_{pq}$ for $(p,q)\in\cK_\e^{(2)}$ we have 
$$
\begin{aligned}
\big(\partial_1\cU-\cR\land \Ge_1\big)(p\e+z_1,q\e+z_2)={1\over \d}\Big(\Ga_{pq}-\Ga_{p-1q}-{\e\over 2}(\cR_{pq}+\cR_{p-1q})\land \Ge_1\Big) - {z_1 \over \d}(\cR_{pq}-\cR_{p-1q})\land \Ge_1\\
+{z_2\over \d}(\cR_{pq}-\cR_{p-1q})\land\Ge_2,\quad (z_1,z_2)\in \fI_\d\X \fI_{\e-\d} 
	\end{aligned}
$$
We also have in $R^3_{pq}$
$$
\begin{aligned}
\big(\partial_1\cU-\cR\land & \Ge_1\big)(p\e+z_1,q\e+z_2)\\
&={1\over \d}\Big(\Big(\Ga_{pq}-\Ga_{p-1q}-{\e\over 2}(\cR_{pq}+\cR_{p-1q})\land \Ge_1\Big) - z_1(\cR_{pq}-\cR_{p-1q})\land \Ge_1\Big){\d+2z_2\over 2\d}\\
&+{1\over \d}\Big(\Big(\Ga_{pq-1}-\Ga_{p-1q-1}-{\e\over 2}(\cR_{pq-1}+\cR_{p-1q-1})\land \Ge_1\Big) - z_1(\cR_{pq-1}-\cR_{p-1q-1})\land \Ge_1\Big){\d-2z_2\over 2\d}\\
&+{1\over 2}{\d+2z_2\over 2\d}(\cR_{pq}-\cR_{p-1q})\land \Ge_2+{1\over 2}{\d-2z_2\over 2\d}(\cR_{pq-1}-\cR_{p-1q-1})\land \Ge_2,\qquad (z_1,z_2)\in \fI_\d^2, 
\end{aligned}
$$
which along with \eqref{EQ47} implies
$$
\begin{aligned}
\sum_{q=0}^{N_\e}\sum_{p=0}^{N_\e}\|\partial_1\cU-\cR\land\Ge_1\|^2_{L^2(R^3_{pq})}\leq  {C\over \e}\GE(u)\|^2_{L^2(\O_\d)}.
\end{aligned}
$$
Finally, we obtain \eqref{EQ414}$_{5}$, in the same way we prove \eqref{EQ414}$_{6}$. The estimate \eqref{EQ414}$_{11}$ is consequence of \eqref{EQ414}$_{2,5,6}$ and Poincar\'e inequality.\\[1mm]
{\bf Step 4.} We prove \eqref{EQ415}$_{3,4}$.\\[1mm]
As a consequence of  \eqref{EQ414}$_{5,6}$ and the $2D$ Korn inequality, we obtain \eqref{EQ414}$_{9}$. Then, using \eqref{EQ49}$_{4}$  and the $\cQ_1$ character of $\cR^\dia_3$ yield
\begin{equation}\label{EXUcM01}
	\|\cU^\dia_\a-\Ga_{pq,\a}\|^2_{L^2(\o^H_{pq})}\leq C\e^2 \|\nabla\cU^\dia_\a\|^2_{L^2(\o_{pq})}\quad \hbox{and}\quad \|\cU_\a-\Ga_{pq,\a}\|^2_{L^2(\o^H_{pq})}\leq C\e^4 |\cR_{pq,3}|^2.
\end{equation}
From \eqref{EQ43}$_1$ and the above, we obtain
$$
\|\cU_\a-\cU^\dia_\a\|^2_{L^2(\o_\d)}\leq C\big(\e^2 \|\nabla\cU^\dia_\a\|^2_{L^2(\o_\d)}+\e^2\|\cR_3\|^2_{L^2(\o_\d)}+\d^2 \|\nabla(\cU_\a-\cU^\dia_\a)\|^2_{L^2(\o_\d)}\big)
$$ which with \eqref{EQ413}$_1$-\eqref{EQ414}$_{8,9}$ lead to 
$$
\|\cU_\a-\cU^\dia_\a\|^2_{L^2(\o_\d)}\leq C\e\GE(u)^2.
$$ Hence, \eqref{EQ414}$_{10}$, \eqref{EQ415}$_{3}$  hold.\\[1mm]
Besides, we also have
$$
\|\cU^\dia_3-\Ga_{pq,3}\|^2_{L^2(\o^H_{pq})}\leq C\e^2 \|\nabla\cU^\dia_3\|^2_{L^2(\o_{pq})}\quad \hbox{and}\quad \|\cU_3-\Ga_{pq,3}\|^2_{L^2(\o^H_{pq})}\leq C\e^4\big(|\cR_{pq,1}|^2+|\cR_{pq,2}|^2\big).
$$ Then, we obtain
$$
\|\cU_3-\cU^\dia_3\|^2_{L^2(\o_\d)}\leq C\big(\e^2 \|\nabla\cU^\dia_3\|^2_{L^2(\o_\d)}+\e^2\|\cR_\a\|^2_{L^2(\o_\d)}+\d^2 \|\nabla(\cU_3-\cU^\dia_3)\|^2_{L^2(\o_\d)}\big)
$$ which with \eqref{EQ413}$_2$-\eqref{EQ414}$_{2,11}$ lead to \eqref{EQ415}$_4$. This completes the proof. 
\end{proof}
As a consequence of above lemmas, we get
\begin{lemma}\label{lem52} We have
\begin{equation}\label{EQ51}
\begin{aligned}
&\big\|\nabla(\p_1\fU_3+\cR_2)\big\|_{L^2(\o_\d)}+\big\|\nabla(\p_2\fU_3-\cR_1)\big\|_{L^2(\o_\d)}\leq {C\over \d^{3/2}}\GE(u),\\
&\big\|\p_1\fU_3+\cR_2\big\|_{L^2(\o_\d)}+\big\|\p_2\fU_3-\cR_1\big\|_{L^2(\o_\d)}\leq {C\over \d^{1/2}}\GE(u),\\
& \|\fU-\cU\|_{L^2(\o_\d)}\leq  C \d^{1/2}\GE(u),\quad \|\nabla (\fU_3-\cU_3)\|_{L^2(\o_\d)}\leq {C\over \d^{1/2}}\GE(u).
\end{aligned}
\end{equation} The constants do not depend on $\e$ and $\d$.
\end{lemma}
\begin{proof}
	The estimates in \eqref{EQ51}$_{1,2}$ follow directly from 
	\eqref{rem+}$_2$ and \eqref{EQ414}$_1$. 
	Next, \eqref{EQ51}$_{3,4}$ are obtained from the equalities 
	\eqref{EQ40}$_{1,2}$ together with the estimates 
	\eqref{EQ51}$_{1,2}$ and \eqref{EQ43}$_1$. 
	
	For the third component, we deduce 
	\[
	\|\fU_3 - \cU_3\|_{L^2(\o_\d)} \leq C\,\delta^{1/2}\,\GE(u)
	\]
	using equality \eqref{EQ40}$_5$ in combination with \eqref{EQ43}$_1$. 
	Similarly, for the in-plane components we have
	\[
	\|\fU_\alpha - \cU_\alpha\|_{L^2(\o_\d)} \leq 
	C\,\delta^{1/2}\,\GE(u),
	\]
	which follows from \eqref{EQ414}$_1$, \eqref{EQ46}$_3$, and again \eqref{EQ43}$_1$. 
	
	Finally, estimate \eqref{EQ51}$_6$ is a consequence of 
	\eqref{EQ414}$_{5,6}$ together with \eqref{EQ51}$_{3,4}$.
\end{proof}

\subsection{The functions \texorpdfstring{$\cU^\bt_\a$ and $\cU^{\bt\bt}_\a$}{U-triangle and U-double-triangle}}\label{SS44}
To the family $\Ga_{pq,\a}$, $(p,q)\in \ov{\cK}_\e$ we associate  the field  $\cU^\bt_\a$ (constructed using \eqref{Dphi}) and we set
$$\cU^{\bt\bt}_\a=\cU_\a-\cU^\bt_\a.$$
From the expression of $\cU_\a$ and $\cU_\a^\bt$, we get
{\small\begin{equation}\label{621}
	\begin{aligned}
		\cU_1^{\bt\bt}(p\e+z_1,q\e+z_2)=&-\cR_{pq,3}z_2,\quad \cU_2^{\bt\bt}(p\e+z_1,q\e+z_2)=\cR_{pq,3}z_1 \quad \hbox{a.e. in } \o^H_{pq},\quad (z_1,z_2)\in \fI_{\e-\d}^2,\\
		\cU_1^{\bt\bt}(p\e+z_1,q\e+z_2)= & {\e-\d\over 2}\left({\d+ 2z_2\over 2\d}\cR_{pq,3}-{\d-2z_2\over 2\d}\cR_{pq-1,3}\right)\quad\hbox{a.e. in } R^1_{pq},\quad (z_1,z_2)\in \fI_{\e-\d}\X\fI_\d,\\ 
		\cU_2^{\bt\bt}(p\e+z_1,q\e+z_2)= & z_1\left({\d+ 2z_2\over 2\d}\cR_{pq,3}+{\d-2z_2\over 2\d}\cR_{pq-1,3}\right)\quad\hbox{a.e. in } R^1_{pq},\quad (z_1,z_2)\in \fI_{\e-\d}\X\fI_\d,\\
		\cU_1^{\bt\bt}(p\e+z_1,q\e+z_2)= & - z_2\left({\d+ 2z_1\over 2\d}\cR_{pq,3}+{\d-2z_1\over 2\d}\cR_{p-1q,3}\right)\quad\hbox{a.e. in } R^2_{pq},\quad (z_1,z_2)\in \fI_{\d}\X\fI_{\e-\d},\\
		\cU_2^{\bt\bt}(p\e+z_1,q\e+z_2)= & {\e-\d\over 2}\left(-{\d+ 2z_1\over 2\d}\cR_{pq,3}+{\d-2z_1\over 2\d}\cR_{p-1q,3}\right)\quad\hbox{a.e. in } R^2_{pq}\quad (z_1,z_2)\in \fI_{\d}\X\fI_{\e-\d},,\\
		\cU_1^{\bt\bt}(p\e+z_1,q\e+z_2)= & {\e-\d\over 2}\left({\d+ 2z_2\over 2\d}\cR_{pq,3}-{\d-2z_2\over 2\d}\cR_{pq-1,3}\right){\d+2z_1\over 2\d}\\
		+{\e-\d\over 2}&\left({\d+ 2z_2\over 2\d}\cR_{p-1q,3}-{\d-2z_2\over 2\d}\cR_{p-1q-1,3}\right){\d-2z_1\over 2\d}\quad\hbox{a.e. in } R^3_{pq},\quad (z_1,z_2)\in \fI_{\d}^2,\\
		\cU_2^{\bt\bt}(p\e+z_1,q\e+z_2)= & -{\e-\d\over 2}\left({\d+ 2z_2\over 2\d}\cR_{pq,3}+{\d-2z_2\over 2\d}\cR_{pq-1,3}\right){\d+2z_1\over 2\d}\\
		+{\e-\d\over 2}&\left({\d+ 2z_2\over 2\d}\cR_{p-1q,3}+{\d-2z_2\over 2\d}\cR_{p-1q-1,3}\right){\d-2z_1\over 2\d}\quad\hbox{a.e. in } R^3_{pq},\quad (z_1,z_2)\in \fI_{\d}^2.
\end{aligned}
\end{equation}}
\begin{lemma} 
	We have
	\begin{equation}\label{EQ414Ubt}
		\begin{aligned}
			& \|\nabla \cU^\bt_\alpha\|_{L^2(\o_\d)}\leq {C\over \d^{1/2}}\GE(u),\quad &&\|\cU^\bt_\alpha\|_{L^2(\o_\d)}\leq {C\over \e^{1/2}}\GE(u),\\
			&\|\nabla \cU^{\bt\bt}_\a\|_{L^2(\o_\d)}\leq  {C\over \d^{1/2}}\GE(u),\quad &&\|\cU^{\bt\bt}_\a\|_{L^2(\o_\d)}\leq C \e^{1/2}\GE(u)
		\end{aligned}
	\end{equation} and 
	\begin{equation}\label{EQ415Ubt}
\|\cU^{\bt\bt}_\a\|_{L^2(\o^S_\ed)}\leq C\d^{1/2}\GE(u).
	\end{equation}
	The constants do not depend on $\e$ and $\d$. 
\end{lemma}
\begin{proof} 
	First, we have using \eqref{EQ47+}$_{1,2}$ and  \eqref{EQ417}$_2$, 
	\begin{equation}\label{EXUcb01}
\sum_{q=0}^{N_\e-1}\sum_{p=1}^{N_\e-1}\big|\big(\Ga_{pq}-\Ga_{p-1q}\big)\cdot\Ge_\a\big|^2+  \sum_{p=0}^{N_\e-1}\sum_{q=1}^{N_\e-1}\big|\big(\Ga_{pq}-\Ga_{pq-1}\big)\cdot\Ge_\a|^2\leq {C\over \e}\GE(u)^2.
	\end{equation}
 Proceeding as in the derivation of the estimates for $\cR$ (see Steps 1 and 2 in the proof of Lemma \ref{lem51}), and using \eqref{EXUcb01}, we obtain \eqref{EQ414Ubt}$_{1,2}$. Then, from  \eqref{EQ414Ubt}$_{1}$ and \eqref{EQ414}$_9$ the estimate \eqref{EQ414Ubt}$_3$ follows.\\
Equalities \eqref{621}$_{1,2}$ and estimate \eqref{EQ417}$_2$ yield
$$\|\cU^{\bt\bt}_\a\|_{L^2(\o^H_\ed)}\leq C\e^{1/2}\GE(u).$$ Then, from the above together with \eqref{EQ43}$_1$ and \eqref{EQ414Ubt}$_{3}$ we obtain \eqref{EQ414Ubt}$_{4}$.\\
 From the expressions \eqref{621} of  $\cU^{\bt\bt}_\a$, the estimates \eqref{EQ47}$_{5,6}$-\eqref{EQ417}$_2$, we obtain  \eqref{EQ415Ubt}. 
\end{proof}
\section{The applied body forces}

In this subsection, we introduce the applied body forces and then using the estimates derived in the previous sections, we give the upper bound for the $L^2$-norm of the linearized strain tensor. 

Let $f$ be in  $L^2(\o)^3$, we set
\begin{equation}\label{Ass01}
	\begin{aligned}
		&f_\ed(x)=\e^\kappa[\d\big(f_1(x')\Ge_1+f_2(x')\Ge_2\big)+\d^2 f_3(x')\Ge_3]\qquad \hbox{for a.e. } x\in \O_\d,\quad \kappa\geq0.
	\end{aligned}		
\end{equation} 
where $f=f_1\Ge_1+f_2\Ge_2+f_3\Ge_3$ (we extend $f$ by $0$ outside of $\o$).
\begin{lemma}\label{Force01}  The solution $u_\ed\in\GU_\ed$ to problem \eqref{Pb} satisfies
	\begin{equation}\label{Main04}
		\GE(u_\ed)\leq C{\e^{\kappa-{1/2}}\d^2}\|f\|_{L^2(\o)},
	\end{equation}
	The constant is independent of $\e$ and $\d$.
\end{lemma}
\begin{proof} First, the solution $u_\ed$ satisfying (choosing $v=0$)
	\begin{equation}\label{Main01}
		c_0\|e(u_\ed)\|^2_{L^2(\O_\d)}\leq \int_{\O^S_{\e\d}}\A_{\ed,ijkl}(x)e_{ij}(u_\ed)e_{kl}(u_\ed)\,dx\leq  \int_{\O_{\d}}f_{\e\d}\cdot u_\ed\,dx,
	\end{equation}
	Using the decomposition \eqref{EQ41D+}, we have
	\begin{equation}\label{Force02}
		\begin{aligned}
			\int_{\O_\ed}f_\ed\cdot u\,dx=\d\int_{\o_\d}f_\ed\cdot\fU\,dx'+\int_{\O^S_\ed}f_\ed\cdot\wt{u}\,dx.
		\end{aligned}
	\end{equation}
	Now, the expression \eqref{Ass01} of the applied forces, the estimates \eqref{EQ46}$_{2,4}$ and equality \eqref{WMC02} lead to \eqref{Main04}.
\end{proof}
{\bf From now on, we assume $\kappa = 1$, so that the total elastic energy is of order $\varepsilon\delta^4$; 
	this is the von K\'arm\'an regime for the composite under consideration, 
	in analogy with the classical von K\'arm\'an scaling $\delta^5$ for a homogeneous thin plate 
	(see~\cite{Amartya02}, where the full hierarchy of energy regimes for this composite 
	is identified in the nonlinear setting).}
\section{Unfolding operators and their properties}\label{S06}
In this section, we introduce two rescaled unfolding operators, which play a central role in deriving the asymptotic limits.
	\begin{definition}[$3$D re-scaled unfolding operators]
	For every measurable function $\psi$ on $\O_\ed^{(1)}$ (resp. $\O_\ed^{(2)}$), we define the measurable functions $\Pi^{(1)}_{\e\d}(\psi)$ (resp. $\Pi^{(2)}_{\e\d}(\psi)$) by  
	\begin{equation}\label{Pi}
	\begin{aligned}
	\Pi^{(1)}_{\e\d}(\psi) (x',y) & \doteq \psi \Big(\e\Big[{x'\over \e}\Big]+ \e y_1\Ge_1+\d y_2\Ge_2,\d y_3\Big)\qquad\text{for a.e.}\;(x',y')\in\o\X\cY_1,\\
\hbox{(resp. }\;	\Pi^{(2)}_{\e\d}(\psi) (x',y) & \doteq \psi \Big(\e\Big[{x'\over \e}\Big]+ \d y_1\Ge_1+\e y_2\Ge_2,\d y_3\Big)\qquad\text{for a.e.}\;(x',y')\in\o\X\cY_2\hbox{)}.
	\end{aligned}
	\end{equation} 
\end{definition} 
Below, we recall some of the properties and inequalities related to the $3$D re-scaled unfolding  operator.\\[1mm]
$\bullet$  For every $\psi\in L^2(\O_{\ed}^{(\alpha)})$, we have
	\begin{equation}\label{ruo01}
		\|\Pi^{(\alpha)}_{\e\d}(\psi)\|_{L^2(\o\X\cY_\alpha)}\leq {\sqrt\e\over \d}\|\psi\|_{L^2(\O_{\ed}^{(\alpha)})}.
	\end{equation}
$\bullet$ For any $\phi\in H^1(\O_{\ed}^{(\alpha)})$, one has
	\begin{equation}\label{ruo02}
		\p_{y_\a}\Pi^{(\alpha)}_{\e\d}(\phi)=\e\Pi^{(\alpha)}_{\e\d}\big(\p_\a\phi \big),\quad 	
		\p_{y_{3-\alpha}}\Pi^{(\alpha)}_{\e\d}(\phi)=\d\Pi^{(\alpha)}_{\e\d}\big(\p_{3-\a} \phi\big),\quad 	\p_{y_3}\Pi^{(\alpha)}_{\e\d}(\phi)=\d\Pi^{(\alpha)}_{\e\d}\big(\p_3 \phi\big).
	\end{equation}
\begin{definition}[$2$D re-scaled unfolding operators]
	For every measurable function $\psi$ on $\o^{(1)}_\ed$  (resp. $\o^{(2)}_\ed$) we define the measurable functions $\Ted^{(1)}(\psi)$ (resp. $\Ted^{(2)}(\psi)$) by  
\begin{equation}\label{Ted}
	\begin{aligned}
	\Ted^{(1)}(\psi) (x',y') & \doteq \psi \Big(\e\Big[{x'\over \e}\Big]+ \e y_1\Ge_1+\d y_2\Ge_2\Big)\qquad\text{for a.e.}\;(x',y)\in\o\X Y_1,\\
\hbox{(resp. }\;	\Ted^{(2)}(\psi) (x',y') & \doteq \psi \Big(\e\Big[{x'\over \e}\Big]+ \d y_1\Ge_1+\e y_2\Ge_2\Big)\qquad\text{for a.e.}\;(x',y)\in\o\X Y_2\hbox{)}.
	\end{aligned}
\end{equation}
\end{definition} 
Below, we briefly give the properties of the $2$D re-scaling unfolding operator.\\[1mm]
$\bullet$  For every $\psi\in L^2(\o_\ed^{(\alpha)})$, we have
	\begin{equation}\label{EQ55}
		\|\Ted^{(\alpha)}(\psi)\|_{L^2(\o\X Y_\alpha)}\leq \sqrt{\e\over \d}\|\psi\|_{L^2(\o_\ed^{(\alpha)})}.
	\end{equation}
$\bullet$ For any $\phi\in H^1(\o_\ed^{(\alpha)})$, one has
	\begin{equation}\label{EQ56}
		\partial_{y_\a}\Ted^{(\alpha)}(\phi)=\e\Ted^{(\alpha)}(\p_\a\phi),\quad 	\p_{y_{3-\a}}\Ted^{(\alpha)}(\phi)=\d\Ted^{(\alpha)}(\partial_{3-\a} \phi).
	\end{equation}
$\bullet$ For $\phi\in L^2(\o^{(\alpha)}_\ed)$, we have
	$$\begin{aligned}
		\Pi_\ed^{(\alpha)}(\phi)(x',y',0)=\Ted^{(\alpha)}(\phi)(x',y').
	\end{aligned}$$

\section{Asymptotic limit of the strain tensor}\label{S07}
In this section, we present the limit of the  strain tensor when both $\e$ and $\d$ tend to zero simultaneously, satisfying \eqref{AS0}. To this end, we consider a sequence of admissible displacements $\{u_\ed\}_{\e,\d}\subset \GU_{\e\d}$ such that
\begin{equation}\label{Eqref01}
	\GE(u_\ed)\leq C {\e^{1/2}\d^2}. 
\end{equation} The constant is independent of $\e$ and $\d$.

Then, we decompose each $u_\ed$ as \eqref{EQ41D+}, then from \eqref{rem+} and \eqref{EQ46} we obtain
\begin{equation}\label{EQ61}
	\begin{aligned}
		&\|\fU_{\ed,3}\|_{H^1(\o_\d)}\leq C\d,\quad \|\p_{\a\b}^2\fU_{\ed,3}\|_{L^2(\o_\d)}\leq C\sqrt{\e\d},\\
		& \|\nabla \fU_{\ed,\a}\|_{L^2(\o_\d)}\leq C\d\sqrt{\ed},\quad\|\fU_{\ed,\a}\|_{L^2(\o_\d)}\leq C{\d^2},\\
		&\|\fu_\ed\|_{L^2(\o_\d)}+\d\|\nabla\fu_\ed\|_{L^2(\o_\d)}\leq C\d^2\sqrt{\ed},
	\end{aligned}
\end{equation}
and
using the properties of the rescaled unfolding operators \eqref{ruo01}-\eqref{ruo02}, we obtain
\begin{equation}\label{55}
	\begin{aligned}
		\|\Pi^{(\alpha)}_{\e\d}(\wt{u}_\ed)\|_{L^2(\o\X\cY_\alpha)}+\|\partial_{y_3}\Pi^{(\alpha)}_{\e\d}( \wt{u}_\ed)\|_{L^2(\o\X\cY_\alpha)}+\|\partial_{y_{3-\alpha}}\Pi_\ed^{(\alpha)}(\wt{u}_\ed)\|_{L^2(\o\X \cY_\alpha)}\leq C\ed^2,\\
		\|\partial_{y_\alpha}\Pi^{(\alpha)}_{\e\d}( \wt{u}_\ed)\|_{L^2(\o\X\cY_\alpha)}\leq C\e^2\d.
	\end{aligned}
\end{equation}
The constants  are independent of $\e$ and $\d$.

In all the lemmas below, we extract  a subsequence  of $\{\ed\}_{\e,\d}$ (still denoted by $\{\ed\}_{\e,\d}$)  in order to get the desired convergences.	
\subsection{Limit behavior of the macroscopic fields}\label{61}
We define the following spaces
$$
\begin{aligned}
	H_{\gamma}^{1}\big(\o\big)&\doteq \big\{\psi\in H^1(\o)\;|\; \psi=0\quad \hbox{a.e. in } \gamma\,\big\},\\
	H_{(0,l)}^{2}\big((0,L)_{x_\alpha}\big)&\doteq \big\{\psi \in H^2(0,L)\;|\; \psi=0\quad \hbox{a.e. in } (0,l)\,\big\}.
\end{aligned}
$$

The following convergences hold for the macroscopic fields:
\begin{lemma}\label{lem61}
	There exist $\fF\in H^2_{(0,l)}\big((0,L)_{x_1}\big)$,  $\fG\in H^2_{(0,l)}\big((0,L)_{x_2}\big)$, $\fU_\a \in H^1_\gamma(\o)$   such that  
	\begin{equation}\label{46}
		\begin{aligned}
			{1\over \d^2}\fU_{\ed,\a}& \to \fU_\a \quad &&\text{strongly in } L^2(\o),\\
			{1\over \d}\fU_{\ed,3}& \to \fU_3=\fF+\fG \quad &&\text{strongly in } H^1_\gamma(\o).
		\end{aligned}
	\end{equation}
Moreover, there exist  $\wh U^{(1)}_\a\in L^2(\o\X(0,1)_{y_1};H^1_0(\fI))$ and $\wh U^{(2)}_\a\in L^2(\o\X(0,1)_{y_2};H^1_0(\fI))$ such that 
	\begin{equation}\label{EQ70}
		\begin{aligned}
		{1\over \ed}\Ted^{(1)}(\p_2 \fU_{\ed,\a}) & \rightharpoonup \p_2\fU_\a+\p_{y_2}\wh U_\a^{(1)}\quad &&\hbox{weakly in } L^2(\o\X Y_1),\\
		{1\over \ed}\Ted^{(1)}(\p_1 \fU_{\ed,\a}) & \rightharpoonup 0 \quad &&\hbox{weakly in } L^2(\o\X Y_1),\\
		{1\over \ed}\Ted^{(2)}(\p_1 \fU_{\ed,\a}) & \rightharpoonup \p_1\fU_\a+\p_{y_1}\wh U_\a^{(2)} \quad &&\hbox{weakly in } L^2(\o\X Y_2),\\
		{1\over \ed}\Ted^{(2)}(\p_2 \fU_{\ed,\a}) & \rightharpoonup 0 \quad &&\hbox{weakly in } L^2(\o\X Y_2),\\
		{1\over \ed}\Ted^{(\beta)}(\p_\a \fu_{\ed}) & \rightharpoonup 0 \quad &&\hbox{weakly in } L^2(\o\X Y_\beta).
		\end{aligned}
	\end{equation}
	Furthermore, there exist $\wh U^{(\a)}_3\in L^2(\o\X(0,1)_{y_\a};H^1_0(\fI))$ such that
	\begin{equation}\label{EQ70+}
		\begin{aligned}
		{1\over \e}\Ted^{(1)}(\p^2_{22}\fU_{\ed,3}) & \rightharpoonup \p^2_{22}\fU_3+\p_{y_{2}}\wh U^{(1)}_3 \quad &&\hbox{weakly in } L^2(\o\X Y_1),\\
		{1\over \e}\Ted^{(1)}(\p^2_{21}\fU_{\ed,3}) & \rightharpoonup 0  \quad &&\hbox{weakly in } L^2(\o\X Y_1),\\
		{1\over \e}\Ted^{(1)}(\p^2_{11}\fU_{\ed,3}) & \rightharpoonup 0 \quad &&\hbox{weakly in } L^2(\o\X Y_1),\\
		{1\over \e}\Ted^{(2)}(\p^2_{11}\fU_{\ed,3}) & \rightharpoonup \p^2_{11}\fU_3+\p_{y_{1}}\wh U^{(2)}_3 \quad &&\hbox{weakly in } L^2(\o\X Y_2),\\
		{1\over \e}\Ted^{(2)}(\p^2_{22}\fU_{\ed,3}) & \rightharpoonup 0 \quad &&\hbox{weakly in } L^2(\o\X Y_2),\\
		{1\over \e}\Ted^{(2)}(\p^2_{12}\fU_{\ed,3}) & \rightharpoonup 0  \quad &&\hbox{weakly in } L^2(\o\X Y_2).
		\end{aligned}
	\end{equation}
\end{lemma}
\begin{proof} 
	The proof is given in the following steps.
	
{\bf Step 1.} We prove \eqref{46}$_1$.
	
	The estimates \eqref{EQ413}$_1$ and \eqref{N3}$_{3}$  give 
	$$
\|\cU^\dia_{\ed,\a}\|_{H^1(\o_\d)}\leq C{\d^2},\qquad \|\fU_{\ed,\a}-\cU^\dia_{\ed,\a}\|_{L^2(\o_\d)}\leq C\e\d^2.
	$$
	Then, there exists $\fU_\a \in H^1_\gamma(\o)$ such that
\begin{equation}\label{Con01}
		{1\over \d^2}\cU^\dia_{\ed,\a} \rightharpoonup \fU_\a\quad \text{weakly in $H^1_\gamma(\o)$ and strongly in } L^2(\o).
\end{equation}
	The above  gives \eqref{46}$_1$.\\
	{\bf Step 2.} We prove \eqref{46}$_2$.
	
	The estimates  \eqref{EQ49}$_{1,2,5,6}$, \eqref{EQ49+}$_{5,6}$, \eqref{EQ413}$_2$ and \eqref{N3}$_{1,2,4}$ give 
	$$
	\begin{aligned}
	& \|\cU^\dia_{\ed,3}\|_{H^1(\o_\d)}\leq C{\d},\qquad \|\cR^\dia_\ed\|_{H^1(\o_\d)}\leq C{\d},\\
	  & \big\|\partial_1\cU^\dia_{\ed,3}+\cR^\dia_{\ed,2}\big\|_{L^2(\o_\d)}+\big\|\partial_2\cU^\dia_{\ed,3}-\cR^\dia_{\ed,1}\big\|_{L^2(\o_\d)}\leq C\ed,\\
 &\|\p_1\cR^\dia_{\ed,1}\|_{L^2(\o_\d)}+\|\p_2\cR^\dia_{\ed,2}\|_{L^2(\o_\d)} \leq C{\d^2\over \e},\\
 &\big\|\p_1\fU_{\ed,3}+\cR^\dia_{\ed,2}\big\|_{L^2(\o_\d)}+\big\|\p_2\fU_{\ed,3}-\cR^\dia_{\ed,1}\big\|_{L^2(\o_\d)}+ \|\fU_{\ed,3}-\cU^\dia_{\ed,3}\|_{H^1(\o_\d)}\leq  C\ed.
	 \end{aligned}
	$$
	Then, there exists $\fU_3,\; \cR_1,\; \cR_2\in H^1_\gamma(\o)$ such that
	\begin{equation}\label{EQ89}
	\begin{aligned}
&{1\over \d}\cU^\dia_{\ed,3}\rightharpoonup \fU_3\quad \text{weakly in $H^1_\gamma(\o)$ and strongly in $L^2(\o)$},\\
&{1\over \d}\cR^\dia_{\ed,\a} \rightharpoonup \cR_\a\quad \text{weakly in $H^1_\gamma(\o)$ and strongly in } L^2(\o)
	\end{aligned}
	\end{equation}
	with
	\begin{equation}\label{EQ90}
	\p_1\fU_3=-\cR_2,\quad \p_2\fU_3=\cR_1,\quad \p_\a\cR_\a=0,\quad \text{a.e. in $\o$}.
	\end{equation}
	So, we get $\fU_3\in H^2(\o)$ with $\p_{12}^2\fU_3=0$ a.e. in $\o$ and 
	$$
	{1\over \d}\cU^\dia_{\ed,3}\to \fU_3\quad \text{strongly in $H^1(\o)$}.
	$$
	This implies  there exist $\fF\in H^2_{(0,l)}((0,L)_{x_1})$ and $\fG\in H^2_{(0,l)}\big((0,L)_{x_2}\big)$ such that
	$$\fU_3(x_1,x_2)=\fF(x_1)+\fG(x_2),\quad\text{for a.e. $(x_1,x_2)\in\o$}.$$
	This completes \eqref{46}$_2$.
	
{\bf Step 3.}  We prove \eqref{EQ70}$_{1,2,3,4}$.\\[1mm]
From the expression \eqref{621} and \eqref{Dphi}, we obtain 
\begin{equation}\label{ExUcI}
	\left.\begin{aligned}
	\p_{y_1}\Ted^{(1)}(\cU_{\ed,1}^{\bt\bt})(x',y_1,y_2)&=0,\\
	\p_{y_2}\Ted^{(1)}(\cU_{\ed,1}^{\bt\bt})(x',y_1,y_2)&=(\e-\d)\int_\fI \Ted^{(1)}(\cR_{\ed,3})(x',y_1,t)dt,\;\\
	\p_{y_1}\Ted^{(1)}(\cU_{\ed,2}^{\bt\bt})(x',y_1,y_2)&=\e\Ted^{(1)}(\cR_{\ed,3}),\\
	\p_{y_2}\Ted^{(1)}(\cU_{\ed,2}^{\bt\bt})(x',y_1,y_2)&=\e y_1\p_{y_2}\Ted^{(1)}(\cR_{\ed,3}),
\end{aligned}\right\}\; \text{ a.e. in}\, \o\X\left({\d\over \e},1-{\d\over \e}\right)\X\fI.
\end{equation}
Using the estimates \eqref{EQ414}$_{7,8}$, \eqref{EQ55},  \eqref{EQ56} and \eqref{EQAp}$_2$ (see Lemma \ref{lemAp1} in Appendix), we have
\begin{equation}\label{EQ813}
\begin{aligned}
&\|\cR_{\ed,3}\|_{L^2(\o_\d)} \leq C{\d^2},\qquad \|\cR_{\ed,3}\|_{L^2(\o^S_\ed)} \leq C{\d^2}\sqrt{\d\over \e},\qquad\|\nabla \cR_{\ed,3}\|_{L^2(\o_\d)}\leq C{\d}\sqrt{\d\over \e},\\
&\|\Ted^{(1)}(\cR_{\ed,3})\|_{L^2(\o\X Y_1)}+\|\p_{y_2}\Ted^{(1)}(\cR_{\ed,3})\|_{L^2(\o\X Y_1)}\leq C{\d^2},\quad \|\p_{y_1}\Ted^{(1)}(\cR_{\ed,3})\|_{L^2(\o\X Y_1)}\leq C\ed.
\end{aligned}
\end{equation} The estimates \eqref{EQ813}$_{1,3}$ can also be read as follows:
$$
\Big\|{\e\over \d^2}\cR_{\ed,3}\Big\|_{L^2(\o_\d)} \leq C\e,\qquad \Big\|{\e\over \d^2}\nabla \cR_{\ed,3}\Big\|_{L^2(\o_\d)}\leq C\sqrt{\e\over \d}.
$$ Now, the above and Lemma \ref{lemAp2} give
$$
\begin{aligned}
	{\e\over \d^2}\cR_{\ed,3}&\to 0\quad &&\text{strongly in $L^2(\o)$},\\
	{1\over \ed}\p_{y_1}\Ted^{(1)}(\cR_{\ed,3})&\rightharpoonup 0\quad &&\text{weakly in $L^2(\o\X Y_1)$},\\
	{1\over \d^2}\p_{y_2}\Ted^{(1)}(\cR_{\ed,3})&\rightharpoonup 0\quad &&\text{weakly in $L^2(\o\X Y_1)$}.
\end{aligned}
$$
The above together with the expressions \eqref{ExUcI} give
\begin{equation}\label{EXUcff01}
	\begin{aligned}
{1\over \e^2\d}\p_{y_1}\Ted^{(1)}(\cU^{\bt\bt}_{\ed,\a})& \rightharpoonup 0\quad&&\text{weakly in $L^2(\o\X Y_1)$},\\
{1\over \e\d^2}\p_{y_2}\Ted^{(1)}(\cU^{\bt\bt}_{\ed,\a})&  \rightharpoonup 0\quad &&\text{weakly in $L^2(\o\X Y_1)$}.
	\end{aligned}
\end{equation}
Below, we consider the asymptotic behavior of $\{\Ted^{(1)}(\cU^\bt_{\ed,\a})\}_{\e,\d}$ and its partial derivatives.\\
First, using the estimates \eqref{EQ414Ubt}$_{1,2}$, \eqref{EQ415}$_3$-\eqref{EQ415Ubt} and \eqref{EQ414Ubt}$_2$, we have
$$\|\cU^\bt_{\ed,\a}\|_{L^2(\o_\d)}\leq C{\d^2},\quad \|\nabla \cU^\bt_{\ed,\a}\|_{L^2(\o_\d)}\leq C{\d}\sqrt{\ed},\quad \|\cU^\bt_{\ed,\a}-\cU^\dia_{\ed,\a}\|_{L^2(\o_\d)}\leq C\e\d^2,$$
which together with the convergence \eqref{Con01}, give
\begin{equation}\label{EQ67}
{1\over \d^2}\cU^\bt_{\ed,\a} \to \fU_\a\quad \text{strongly in $L^2(\o)$}.
\end{equation}
The above estimate of  $\nabla \cU^\bt_{\ed,\a}$ can also be read as follows:
$$\Big\|{1\over \d^2}\nabla \cU^\bt_{\ed,\a}\Big\|_{L^2(\o_\d)}\le C\sqrt{\e\over \d}.$$ Then, from Lemma \ref{lemAp2} and convergence  \eqref{EQ67}, we obtain
\begin{equation}\label{EXUc02}
\begin{aligned}
{1\over \e\d^2}\p_{y_2}\Ted^{(1)}(\cU^\bt_{\ed,\a}) &\rightharpoonup \p_2\fU_\a\quad &&\text{weakly in $L^2(\o\X Y_1)$},\\
{1\over \e^2\d}\p_{y_1}\Ted^{(1)}(\cU^\bt_{\ed,\a}) &\rightharpoonup 0\quad &&\text{weakly in $L^2(\o\X Y_1)$}.
\end{aligned}
\end{equation}
The above convergences and those  in \eqref{EXUcff01} lead to
\begin{equation}\label{EXUc01}
	\begin{aligned}
		{1\over \e\d^2}\p_{y_2}\Ted^{(1)}(\cU_{\ed,\a}) &\rightharpoonup \p_2\fU_\a \quad &&\text{weakly in $L^2(\o\X Y_1)$},\\
		{1\over \e^2\d}\p_{y_1}\Ted^{(1)}(\cU_{\ed,\a}) &\rightharpoonup 0\quad &&\text{weakly in $L^2(\o\X Y_1)$}.
	\end{aligned}	
\end{equation}
Observe that using \eqref{EQ414}$_9$, \eqref{EQ51}$_5$  and \eqref{EQ61}$_3$, we have
\begin{equation}\label{H02}
	\|\nabla(\fU_{\ed,\a}-\cU_{\ed,\a})\|_{L^2(\o_\d)}\leq {\d}\sqrt{\ed},\qquad \|\fU_{\ed,\a}-\cU_{\ed,\a}\|_{L^2(\o_\d)}\leq C\e\d^2\sqrt{\d\over \e}.
\end{equation}
Transforming using $\Ted^{(1)}$ (see \eqref{EQ55}-\eqref{EQ56}) gives
$$\begin{aligned}
	&\|\Ted^{(1)}(\fU_{\ed,\a}-\cU_{\ed,\a})\|_{L^2(\o\X Y_1)}\leq C\e\d^2,\quad  \big\|\p_{y_2}\Ted^{(1)}(\fU_{\ed,\a}-\cU_{\ed,\a})\big\|_{L^2(\o\X Y_1)}\leq C\e\d^2,\\
	&\big\|\p_{y_1}\Ted^{(1)}(\fU_{\ed,\a}-\cU_{\ed,\a})\big\|_{L^2(\o\X Y_1)}\leq C\e^2\d. 
\end{aligned}
$$
From the above inequalities, combined with the fact that  $\fU_{\ed,\a}-\cU_{\ed,\a}$ vanishes in $\o^H_\ed$, we deduce that there exist  $\wh U^{(1)}_\a\in L^2(\o\X(0,1)_{y_1};H^1_0(\fI))$ such that 
$$
\begin{aligned}
	{1\over \e\d^2}\Ted^{(1)}(\fU_{\ed,\a}-\cU_{\ed,\a})&\rightharpoonup \wh U^{(1)}_\a \quad &&\text{weakly in $L^2(\o\X(0,1)_{y_1}; H^1(\cJ))$},\\
	{1\over \e^2\d}\p_{y_1}\Ted^{(1)}(\fU_{\ed,\a}-\cU_{\ed,\a}) &\rightharpoonup 0\quad &&\text{weakly in $L^2(\o\X Y_1)$}.
\end{aligned}
$$
So, the above convergences and those in \eqref{EXUc01} yield \eqref{EQ70}$_{1,2}$.  Similarly, we get \eqref{EQ70}$_{3,4}$. \\[1mm]
{\bf Step 4.}  We prove \eqref{EQ70}$_{5}$.\\[1mm]
From the estimates \eqref{EQ61}$_{5,6}$ and the properties of the operator $\Ted^{(1)}$, we get
$$\|\Ted^{(1)}(\fu_\ed)\|_{L^2(\o\X Y_1)}+\|\p_{y_2}\Ted^{(1)}(\fu_\ed)\|_{L^2(\o\X Y_1)}\leq C\e^2\d^2,\quad \|\p_{y_1}\Ted^{(1)}(\fu_\ed)\|_{L^2(\o\X Y_1)}\leq C \e^2\d.$$
These estimates give \eqref{EQ70}$_{5}$ for $\b=1$. Similarly, we prove the convergence \eqref{EQ70}$_5$ for $\b=2$.\\[1mm]
{\bf Step 5.} We prove \eqref{EQ70+}.
	
	From \eqref{EQ414}$_{1,2}$, \eqref{EQ51}$_{1,2,3,4}$ and \eqref{EQ415}$_1$ we obtain
\begin{equation}\label{EQ610}
	\begin{aligned}
		&\|\cR_{\ed,\a}\|_{L^2(\o_\d)} \leq C{\d}, \qquad \|\nabla \cR_{\ed,\a}\|_{L^2(\o_\d)} \leq C\,\sqrt{\ed},\quad  \|\cR_{\ed,\a}-\cR^\dia_{\ed,\a}\|_{L^2(\o_\d)}\leq C\ed, \\
		&\|\p_1 \fU_{\ed,3} + \cR_{\ed,2}\|_{L^2(\o_\d)}
		+ \|\p_2 \fU_{\ed,3} - \cR_{\ed,1}\|_{L^2(\o_\d)} \leq C\,\d\sqrt{\ed},\\
		&\|\nabla(\p_1 \fU_{\ed,3} + \cR_{\ed,2})\|_{L^2(\o_\d)}
		+ \|\nabla(\p_2 \fU_{\ed,3} - \cR_{\ed,1})\|_{L^2(\o_\d)} \leq C\,\d\sqrt{\ed}.
	\end{aligned}
\end{equation}
Using these inequalities together with the convergence \eqref{EQ89}$_2$ and equalities \eqref{EQ90}, we deduce
\begin{equation*}
	\begin{aligned}
		\frac{1}{\d}\,\cR_{\ed,2} &\to -\p_1 \fU_3 
		&\quad &\text{strongly in $L^2(\o)$}, \\[4pt]
		\frac{1}{\d}\,\cR_{\ed,1} &\to \;\;\p_2 \fU_3 
		&\quad &\text{strongly in $L^2(\o)$}.
	\end{aligned}
\end{equation*}

Next, by the definition of $\cR_{\ed,\a}$ and the estimates \eqref{EQ610}$_{1,2}$ together with Lemma~\ref{lemAp2}, we obtain
\begin{equation}\label{EXRc01}
	\begin{aligned}
		{1\over \e}\Ted^{(1)}(\p_2 \cR_{\ed,2}) 
		&\rightharpoonup -\p^2_{21} \fU_3 = 0 
		&\quad &\text{weakly in $L^2(\o \times Y_1)$}, \\[4pt]
		{1\over \e}\Ted^{(1)}(\p_1 \cR_{\ed,2}) 
		&\rightharpoonup 0, 
		&\quad &\text{weakly in $L^2(\o \times Y_1)$}, \\[4pt]
		{1\over \e}\Ted^{(1)}(\p_2 \cR_{\ed,1}) 
		&\rightharpoonup \p^2_{22} \fU_3 
		&\quad &\text{weakly in $L^2(\o \times Y_1)$}, \\[4pt]
		{1\over \e}\Ted^{(1)}(\p_1 \cR_{\ed,1}) 
		&\rightharpoonup 0
		&\quad &\text{weakly in $L^2(\o \times Y_1)$}.
	\end{aligned}
\end{equation}
From \eqref{EQ610}$_{4,5,6,7}$ and \eqref{EQ55}  we get
$$
\begin{aligned}
	&\|\Ted^{(1)}(\p_1 \fU_{\ed,3} + \cR_{\ed,2})\|_{L^2(\o\X Y_1)} + \|\Ted^{(1)}(\p_2 \fU_{\ed,3} - \cR_{\ed,1})|_{L^2(\o\X Y_1)} \leq C\ed,\\
	&\|\p_{y_2}\Ted^{(1)}\big(\p_1 \fU_{\ed,3} + \cR_{\ed,2}\big)\|_{L^2(\o\X Y_1)} + \|\p_{y_2}\Ted^{(1)}\big(\p_2 \fU_{\ed,3} - \cR_{\ed,1}\big)|_{L^2(\o\X Y_1)} \leq C\ed,\\
		&\|\p_{y_1}\Ted^{(1)}\big(\p_1 \fU_{\ed,3} + \cR_{\ed,2}\big)\|_{L^2(\o\X Y_1)} + \|\p_{y_1}\Ted^{(1)}\big(\p_2 \fU_{\ed,3} - \cR_{\ed,1}\big)|_{L^2(\o\X Y_1)} \leq C\e^2.
\end{aligned}
$$ Besides, we recall that from \eqref{EQ40}$_{1,2}$, the functions $\p_1 \fU_{\ed,3} + \cR_{\ed,2}$ and $\p_2 \fU_{\ed,3} - \cR_{\ed,1}$ vanish in $\o^H_\ed$, so there exist $\wh U_3^{(1)},\; \wh U_3^{(1,2)}\in L^2(\o\X(0,1)_{y_1};H_0^1(\fI))$ such that
\begin{equation}\label{814+}
	\begin{aligned}
		{1\over \ed}\Ted^{(1)}(\p_2 \fU_{\ed,3} - \cR_{\ed,1}) &\rightharpoonup \wh U^{(1)}_3\quad&&\text{weakly in $L^2(\o\X(0,1)_{y_1}; H^1(\fI))$},\\
		{1\over \e^2}\p_{y_1}\Ted^{(1)}\big(\p_2 \fU_{\ed,3} - \cR_{\ed,1}\big)&\rightharpoonup 0\quad &&\text{weakly in $L^2(\o\X Y_1)$},\\
		{1\over \ed}\Ted^{(1)}(\p_1 \fU_{\ed,3} + \cR_{\ed,2}) &\rightharpoonup \wh U^{(1,2)}_3\quad&&\text{weakly in $L^2(\o\X(0,1)_{y_2}; H^1(\fI))$},\\
		{1\over \e^2}\p_{y_1}\Ted^{(1)}\big(\p_1 \fU_{\ed,3} + \cR_{\ed,2}\big)&\rightharpoonup 0\quad &&\text{weakly in $L^2(\o\X Y_1)$}.
	\end{aligned}
\end{equation} 
Note that
$$\begin{aligned}
	\Ted^{(1)}(\p^2_{22}\fU_{\ed,3})={1\over \d}\p_{y_2}\Ted^{(1)}(\p_2\fU_{\ed,3}-\cR_{\ed,1})+\Ted^{(1)}(\p_2\cR_{\ed,1}),\\
	\Ted^{(1)}(\p^2_{11}\fU_{\ed,3})={1\over \e}\p_{y_1}\Ted^{(1)}(\p_1\fU_{\ed,3}+\cR_{\ed,2})-\Ted^{(1)}(\p_1\cR_{\ed,2}),\\
\Ted^{(1)}(\p^2_{12}\fU_{\ed,3})={1\over \e}\p_{y_1}\Ted^{(1)}(\p_2\fU_{\ed,3}-\cR_{\ed,1})+\Ted^{(1)}(\p_1\cR_{\ed,1}),\\
\Ted^{(1)}(\p^2_{12}\fU_{\ed,3})=\Ted^{(1)}(\p^2_{21}\fU_{\ed,3})={1\over \d}\p_{y_2}\Ted^{(1)}(\p_1\fU_{\ed,3}+\cR_{\ed,2})-\Ted^{(1)}(\p_2\cR_{\ed,2}).
\end{aligned}
$$ 
The last two above expressions together with those in \eqref{EXRc01} and \eqref{814+} give 
$$0=\p_{y_2}\wh U^{(1,2)}_3-\p^2_{21} \fU_3=\p_{y_2}\wh U_3^{(1,2)}.$$
So, we have
$$
\p_{y_2}\wh U^{(1,2)}_3(x',y')=0\quad\hbox{for a.e. } (x',y')\in \o\X Y_1.
$$ As a consequence, since $\wh U_3^{(1,2)}\in L^2(\o\X(0,1)_{y_1};H_0^1(\fI))$, we have $\wh U^{(1,2)}_3=0$. 
This and  the convergences \eqref{EXRc01}-\eqref{814+} give \eqref{EQ70+}$_{1,2,3}$. Similarly, we show \eqref{EQ70+}$_{4,5,6}$.
\end{proof}

\subsection{Limit behavior of  the strain tensor}\label{62}
First, we introduce the limit microscopic displacement space by
$$
\begin{aligned}
	\wh{D}^{(1)}(\o \X Y_1\X \fI)  &= \Big\{\wh{v}^{(1)}\in  L^2\big(\o\X(0,1)_{y_1};H^1(\fI^2)\big)^3\;\;|\;\; \text{$\wh{v}^{(1)}$ satisfying \eqref{limW02}$_1$}\Big\},\\
	\wh{D}^{(2)}(\o \X Y_2\X \fI)  &= \Big\{\wh{v}^{(2)}\in  L^2\big(\o\X(0,1)_{y_2}; H^1(\fI^2)\big)^3\;\;|\;\;\text{$\wh{v}^{(2)}$ satisfying \eqref{limW02}$_2$}\Big\},
\end{aligned}
$$
where
\begin{equation}\label{limW02}
\begin{aligned}
\wh{v}^{(1)}(x',y)=0,\quad \text{for a.e. $(x',y)$ in $\o\X \big(0,1\big)\X\Big\{\pm{1\over 2}\Big\}\X\fI$},\\
 \wh{v}^{(2)}(x',y)=0,\quad \text{for a.e. $(x',y)$ in $\o\X\Big\{\pm{1\over 2}\Big\}\X\big(0,1\big)\X\fI$}.
\end{aligned}
\end{equation}
Using the decomposition \eqref{EQ41D+}, we have the following expression for the strain tensor
\begin{equation}\label{EXST}
	e(u_\ed)=\begin{pmatrix*}
		e_{11}(\fU_{\ed,m})-x_3\p^2_{11}\fU_{\ed,3} & *& *\\[2mm]
		e_{12}(\fU_{\ed,m})-x_3\p^2_{12}\fU_{\ed,3}& e_{22}(\fU_{\ed,m})-x_3\p^2_{22}\fU_{\ed,3}& *\\[1mm]
		\p_1\fu_{\ed}  & \p_2\fu_{\ed}   & 0
	\end{pmatrix*}+e(\wt{u}_\ed).
\end{equation}
We have following convergences of the unfolded fields: 
\begin{lemma}\label{L64}
We have the following convergences for the linearized strain tensor:
	\begin{equation}\label{STL01}
		\begin{aligned}
			{1\over \ed}\Pi_\ed^{(1)}\big(e(u_\ed)\big)&\rightharpoonup E^{(1)}(\fU,\wh u^{(1)}) \quad &&\hbox{weakly in } L^2(\o\X \cY_1)^{3\X3},\\
			{1\over \ed}\Pi_\ed^{(2)}\big(e(u_\ed)\big)&\rightharpoonup E^{(2)}(\fU,\wh u^{(2)})  \quad &&\hbox{weakly in } L^2(\o\X \cY_2)^{3\X3},
		\end{aligned}
	\end{equation}
	where
	$$
	\begin{aligned}
		E^{(1)}(\fU,\wh u^{(1)})&= \begin{pmatrix}
			0 & \ds{1\over 2}\p_2\fU_1 +{1\over 2}\p_{y_2}\wh{u}^{(1)}_1 &  *  \\[1mm]
			*&\p_{2}\fU_2 -y_3\partial^2_{22}\fU_3+e_{y,22}(\wh{u}^{(1)}) & * \\[1mm]
			\ds {1\over 2}\p_{y_3}\wh{u}^{(1)}_1 &  e_{y,23}(\wh u^{(1)}) & e_{y,33}(\wh{u}^{(1)})
		\end{pmatrix}  \\[1mm]
		E^{(2)}(\fU,\wh u^{(2)})&= \begin{pmatrix}
			\p_1 \fU_1-y_3\partial^2_{11}\fU_3+e_{y,11}(\wh{u}^{(2)})  &  * & * \\[1mm]
			\ds{1\over 2}\p_1\fU_2 +\ds{1\over 2}\p_{y_1}\wh{u}^{(2)}_2 & 0 & * \\[1mm]
			e_{y,13}(\wh u^{(2)}) & \ds {1\over 2}\p_{y_3}\wh{u}^{(2)}_2  & e_{y,33}(\wh{u}^{(2)})
		\end{pmatrix},
	\end{aligned}
	$$
	with the microscopic displacements $\wh u^{(\a)}\in \wh{D}^{(\a)}(\o \X Y_\a\X \fI\big)$.
	\end{lemma}	
\begin{proof}
	First, from the estimates \eqref{55} and the fact that $\wt{u}_\ed$ vanish in $\o^H_\ed$, we have: There exist  $\wt{u}^{(\alpha)}\in \wh{D}^{(\a)}(\o \X Y_\a \X \fI )$ such that 
	\begin{equation}\label{59}
		\begin{aligned}
			{1\over \e\d^2}\Pi^{(\alpha)}_{\e\d}\big(\wt{u}_\ed\big) & \rightharpoonup \wt{u}^{(\alpha)}\quad&&\text{weakly in $L^2\big(\o\X (0,1)_{y_\alpha};H^1\big(\fI^2\big)\big)^3$},\\
			{1 \over \e^2\d}\partial_{y_\alpha}\Pi^{(\alpha)}_{\e\d}\big(\wt{u}_\ed\big) &\rightharpoonup 0\quad&&\text{weakly in $L^2(\o\X \cY_\alpha)^{3}$}.
		\end{aligned}
	\end{equation}
	The convergences \eqref{EQ70}, \eqref{EQ70+} and \eqref{59} along with the expression \eqref{EXST} give \eqref{STL01} by defining $\wh u^{(\a)}$ as
	\begin{equation}\label{Warping011}
		\begin{aligned}
			\wh u^{(1)}=\wh U^{(1)}_1\Ge_1+(\wh U_2^{(1)}-y_3\wh U^{(1)}_3)\Ge_2+\wt{u}^{(1)},\qquad \wh u^{(2)}=(\wh U_1^{(2)}-y_3\wh U^{(2)}_3)\Ge_1+\wh U^{(2)}_2\Ge_2+\wt{u}^{(2)}.
		\end{aligned}
	\end{equation}
So, we get   $\wh u^{(\a)}\in  \wh{D}^{(\a)}(\o \X Y_\a\X \fI\big)$. This completes the proof.
\end{proof}

\subsection{Limit non-penetration condition}\label{SS73}
	 The non-penetration condition \eqref{NPCM}$_2$ using decomposition \eqref{EQ41D+} and equality \eqref{WMC02}$_1$,  then transforming by $\Ted^{(2)}$ gives 
	\begin{multline*}
		\Ted^{(2)}\big(\fU_{\ed,1}\big) \Big(x',{1\over 2}, y_2\Big)-\Ted^{(2)}\big(\fU_{\ed,1}\big) \Big(x', -{1\over 2}, y_2\Big)\\
		-\d y_3\Big(\Ted^{(2)}\big(\p_1\fU_{\ed,3}\big) \Big(x',{1\over 2}, y_2\Big)-\Ted^{(2)}\big(\p_1\fU_{\ed,3}\big) \Big(x',-{1\over 2}, y_2\Big)\Big)\geq -\d,\;\;
		\text{for a.e. } (x',y_2,y_3)\in \o\X\Big({\d\over 2\e},1-{\d\over 2\e}\Big)\X \fI.
	\end{multline*}
 The above and \eqref{EQ56}$_2$ imply 
\begin{multline*}
\int_\fI \p_{y_1}\Ted^{(2)}\big(\fU_{\ed,1}\big) \big(x',y_1, y_2\big)dy_1-\d y_3\int_\fI \p_{y_1}\Ted^{(2)}\big(\p_{1}\fU_{\ed,3}\big) \big(x',y_1, y_2\big)dy_1 \\
=\int_\fI \d\Ted^{(2)}\big(\p_1\fU_{\ed,1}\big) \big(x',y_1, y_2\big)dy_1-\d^2 y_3\int_\fI\Ted^{(2)}\big(\p^2_{11}\fU_{\ed,3}\big) \big(x',y_1, y_2\big)dy_1 \geq -\d,\\
		\text{for a.e. } (x',y_2,y_3)\in \o\X\Big({\d\over 2\e},1-{\d\over 2\e}\Big)\X \fI.
\end{multline*}
So, we have
\begin{multline}\label{NEPReq}
	\int_\fI \Ted^{(2)}\big(\p_1\fU_{\ed,1}\big) \big(x',y_1, y_2\big)dy_1-\d y_3\int_\fI\Ted^{(2)}\big(\p^2_{11}\fU_{\ed,3}\big) \big(x',y_1, y_2\big)dy_1 \geq -1,\\
	\text{for a.e. } (x',y_2,y_3)\in \o\X\Big({\d\over 2\e},1-{\d\over 2\e}\Big)\X \fI.
\end{multline}
	We divide the above inequality by $\ed$ and then we pass to the limit. Due to the convergences \eqref{EQ70}$_1$ and \eqref{EQ70+}$_1$ we obtain
	\begin{multline*}
		\int_\fI \big(\p_1\fU_1(x')+\p_{y_1}\wh U_1^{(2)}(x',y_1,y_2) \big)dy_1- y_3\int_\fI\big(\p^2_{11}\fU_3(x')+\p_{y_{1}}\wh U^{(2)}_3(x',y_1,y_2)\big)dy_1 \geq -\infty,\\
		\text{for a.e. } (x',y_2,y_3)\in \o\X (0,1)\X \fI,
	\end{multline*} 
	which imply
	$$\p_1\fU_1(x')-y_3\p^2_{11}\fU_3(x')\geq -\infty,\quad \text{for a.e. $(x',y_3)\in\o\X\fI$}.$$
	Similarly, we treat the non-penetration condition \eqref{NPCM}$_1$. There is therefore no additional condition on the limit displacements arising from the non-penetration condition.

	\section{Homogenized limit problem}\label{S09}
	The limit macroscopic displacement set is defined as
$$
\D_0\doteq \big[H^1_\gamma(\o)\big]^2 \X \GH^2(\o),
$$
	where
	$$\GH^2(\o)\doteq \left\{\fU_3\in H^2(\o)\,|\,\fU_3=\fF+\fG,\,\text{with $\fF\in H^2_{(0,l)}\big((0,L)_{x_1}\big)$,  $\fG\in H^2_{(0,l)}\big((0,L)_{x_2}\big)$ }\right\}.$$
	The set of limit displacements is defined as
	$$\D=\D_0\X  \wh{D}^{(1)}(\o \X Y_1\X \fI\big)  \X \wh{D}^{(2)}(\o \X Y_2\X \fI\big).$$
	To define the limit elasticity problem, we assume the following: there exist $\A^{(\alpha)}_{ijkl}\in L^\infty\big(\cY_\alpha\big)$, for $\alpha\in\{1,2\}$ such that
	\begin{equation}\label{HC01}
		\Pi^{(\alpha)}_\ed(\A_{\ed,ijkl}) (x',y)\to \A^{(\alpha)}_{ijkl}(y) \qquad \hbox{a.e. in } \o\X \cY_\alpha, 
	\end{equation}
	Observe that (see Assumption \eqref{CoeCon2}) $\A^{(\alpha)}$ satisfy the following for a.e. $y\in \cY_\alpha$ and for all $\GS\in\GS_3$
	\begin{equation}\label{UL78}
		\begin{aligned}
			&\A^{(\alpha)}_{ijkl}(y)=\A^{(\alpha)}_{ijlk}=\A^{(\alpha)}_{klij}(y),\\
			&\A^{(\alpha)}_{ijkl}(y)\GS_{ij}\GS_{kl}=\lim\Pi_\ed^{(\alpha)}(\A_{\ed,ijkl})(x',y)\GS_{ij}\GS_{kl}\geq c_0|\GS|^2.
		\end{aligned}
	\end{equation}

\subsection{Construction of the recovery sequence}\label{SS81}
In this subsection, we construct the recovery sequence.
\begin{lemma}\label{Lem81}
	For every $(\fV,\wh{v}^{(1)},\wh{v}^{(2)})\in \D$, there exist a sequence $\{v_\ed\}_{\e,\d}$ in $\GU_\ed$ such that
\begin{equation}\label{Test81}
	\begin{aligned}
		{1\over \ed}\Pi_\ed^{(1)}\big(e(v_\ed)\big)&\to E^{(1)}(\fV,\wh v^{(1)}) \quad &&\hbox{strongly in } L^2(\o\X \cY_1)^{3\X3},\\
		{1\over \ed}\Pi_\ed^{(2)}\big(e(v_\ed)\big)&\to E^{(2)}(\fV,\wh v^{(2)})  \quad &&\hbox{strongly in } L^2(\o\X \cY_2)^{3\X3},
	\end{aligned}
\end{equation}
	and
$$
		\lim{1\over \e\d^4}\int_{\O_\d}f_\ed\cdot v_\ed\,dx =\int_\o f\cdot\fV\,dx'.
$$
\end{lemma}

 \begin{figure}
	\centering
	\resizebox{0.4\linewidth}{!}{%
		\begin{tikzpicture}[line cap=round, line join=round]
			\tikzset{soft1/.style={fill=blue!40},
				soft2/.style={fill=red!50},
				junction/.style={fill=purple!60}}
			\def\rod{0.25}\def\eps{1.5}\def\N{4}
			\pgfmathsetmacro\W{\N*\eps}
			\pgfmathsetmacro\H{\N*\eps}
			\draw[black, very thick] (-\rod,-\rod) rectangle (\W+\rod,\H+\rod);
			\foreach \j in {0,...,\N}{
				\pgfmathsetmacro\yc{\j*\eps}
				\fill[soft1] (-\rod,\yc-\rod) rectangle (\W+\rod,\yc+\rod);}
			\foreach \i in {0,...,\N}{
				\pgfmathsetmacro\xc{\i*\eps}
				\fill[soft2] (\xc-\rod,-\rod) rectangle (\xc+\rod,\H+\rod);}
			\foreach \i in {0,...,\N}{
				\foreach \j in {0,...,\N}{
					\pgfmathsetmacro\xc{\i*\eps}
					\pgfmathsetmacro\yc{\j*\eps}
					\fill[junction] (\xc-\rod,\yc-\rod) rectangle (\xc+\rod,\yc+\rod);}}
			\foreach \j in {0,...,\N}{
				\pgfmathsetmacro\yc{\j*\eps}
				\draw[black, thin] (-\rod,\yc-\rod) rectangle (\W+\rod,\yc+\rod);}
			\foreach \i in {0,...,\N}{
				\pgfmathsetmacro\xc{\i*\eps}
				\draw[black, thin] (\xc-\rod,-\rod) rectangle (\xc+\rod,\H+\rod);}
			\draw[black, very thick] (-\rod,-\rod) rectangle (\W+\rod,\H+\rod);
	\end{tikzpicture}}
	\caption{{ Top view of $\omega^S_{\varepsilon\delta} = \omega^{(1)}_{\varepsilon\delta} \cup 
		\omega^{(2)}_{\varepsilon\delta}$ (4 periodicity cells shown). 
		{\color{blue!40}$\blacksquare$} Beams of direction $\Ge_1$ ($\omega^{(1)}_{\varepsilon\delta}$); 
		{\color{red!50}$\blacksquare$} beams of direction $\Ge_2$ ($\omega^{(2)}_{\varepsilon\delta}$); 
		{\color{purple!60}$\blacksquare$} junction squares $ \omega^{(1)}_{\varepsilon\delta} 
		\cap \omega^{(2)}_{\varepsilon\delta}$, of measure $\GO(\delta^2/\varepsilon^2)$;
		white: rigid inclusions $\omega^H_{\varepsilon\delta}$.}}
	\label{fig:junction}
\end{figure}

\begin{proof} 	
		By density argument, it is enough to prove the result for  $\wh{v}^{(\alpha)}\in  \wh{D}^{(\a)}(\o \X Y_\a \X \fI ) \cap \cC^2(\ov{\o}\X \ov{Y}_\a \X \ov{\fI})^3$ such that for all $y_3\in \ov{\fI}$ the support of $\wh{v}^{(\alpha)}(\cdot, y_3)$ is compact in $\o \X  Y_\a$ (denoted by $K_\a$), and $\fV\in\D_0\cap [W^{1,\infty}(\o)]^2\X W^{2,\infty}(\o)$.			
			So, there exist $\fF\in W^{2,\infty}\big((0,L)_{x_1}\big)$ and $\fG\in W^{2,\infty}\big((0,L)_{x_2}\big)$ such that
			$$\fV_3(x_1,x_2)=\fF(x_1)+\fG(x_2),\quad\text{for a.e. $(x_1,x_2)\in\o$}.$$
			{ Since $K_\a$ is a compact subset of $\o\X Y_\a$, there exists $r\in\big(0,{1\over 2}\big)$, independent of $\e$ and $\d$, such that for $\a=1$
				\begin{equation}\label{Esupp}
					\wh v^{(1)}(x',y)=0\quad\hbox{whenever}\quad \hbox{dist}(x',\p\o)<r,\;\hbox{ or }\; y_1\in(0,r)\cup(1-r,1),\;\hbox{ or }\; |y_2|>{1\over 2}-r,
				\end{equation}
				and similarly for $\a=2$, with the roles of $y_1$ and $y_2$ exchanged. In particular, $\wh v^{(\a)}$ vanishes in a neighborhood of the two ends $y_\a=0$ and $y_\a=1$ of the small beams. From now on, $\e$ and $\d$ are small enough so that
				\begin{equation}\label{Esmall}
					{\d\over 2}<r\e,
				\end{equation}
				which is possible since $\d/\e\to0$.}
			
{\bf Step 1.} We define $\fV_{\ed,3}\in W^{2,\infty}(\o_\d)$. 
	
We set
\begin{equation}
	\fV_{\ed,3}(x_1,x_2)=\fF_\ed(x_1)+\fG_\ed(x_2),\quad\text{for a.e. $(x_1,x_2)\in\o_\d$}
\end{equation} where $\fF_\ed$ and $\fG_\ed$ are defined from $\fF$ and $\fG$ using Lemma \ref{lemAp7}.

{\bf Step 2.} In this step we construct the sequence of admissible displacements $\{v_{\ed}\}_{\e,\d}\subset \GU_\ed$.\\[1mm]
The sequence of test displacements $v_\ed$ is			
\begin{equation}\label{Ele03}
	v_{\ed}=\left\{\begin{aligned}
	&V_{e,\ed}+\wh{v}^{(1)}_\ed\qquad &&\hbox{in}\;\; \O^{(1)}_\ed,\\
	&V_{e,\ed}+\wh{v}^{(2)}_\ed\qquad &&\hbox{in}\;\; \O^{(2)}_\ed,\\
	&V_{e,\ed}\qquad &&\hbox{in}\;\; \O^H_\ed
\end{aligned}\right.
\end{equation}
where
{
	\begin{equation*}
		\begin{aligned}
			\wh{v}^{(1)}_{\ed}(x)&= \e\d^2\,\wh{v}^{(1)}\Big(p\e,q\e,{x_1-p\e\over \e},{x_2-q\e\over \d},{x_3\over \d}\Big),\\
			&\hskip 30mm \text{for a.e. } x\in (p\e,p\e+\e)\X\Big(q\e-{\d\over 2},q\e+{\d\over 2}\Big)\X\fI_\d,\quad (p,q)\in \cK^{(1)}_\e,\\
			\wh{v}^{(2)}_{\ed}(x)&= \e\d^2\,\wh{v}^{(2)}\Big(p\e,q\e,{x_1-p\e\over \d},{x_2-q\e\over \e},{x_3\over \d}\Big),\\
			&\hskip 30mm \text{for a.e. } x\in \Big(p\e-{\d\over 2},p\e+{\d\over 2}\Big)\X(q\e,q\e+\e)\X\fI_\d,\quad (p,q)\in \cK^{(2)}_\e,
		\end{aligned}
	\end{equation*}
	and $\wh{v}^{(1)}_{\ed}=0$ (resp. $\wh{v}^{(2)}_{\ed}=0$) if  $x_1$ (resp. $x_2$) in $\big(-{\d\over 2},0\big)\cup\big(L,L+{\d\over 2}\big)$.}
	
The field $V_{e,\ed}$ is defined as follows: In $\O_\d$ by
\begin{equation}\label{Ele01}
	\begin{aligned}
		V_{e,\ed}(x)=\begin{pmatrix*}
			\ds{\d^2}\fV_{\ed,1}-\ds{\d} x_3\p_1\fV_{\ed,3}\\[2mm]
			\ds{\d^2}\fV_{\ed,2}-\ds{\d} x_3\p_2\fV_{\ed,3}\\[2mm]
			\ds{\d}\fV_{\ed,3}
		\end{pmatrix*},
	\end{aligned}
\end{equation}
where 
the field $\ds \fV_{\ed,\a}\in W^{1,\infty}(\o_\d)$ is constructed using Lemma \ref{lemAp5} from $\fV_\a\in W^{1,\infty}(\o)$.

{ By \eqref{Esupp}, $\wh v^{(1)}_\ed(x)=0$ whenever ${x_1-p\e\over \e}\in(0,r)\cup(1-r,1)$ in the cell of index $p$. Hence
	\begin{equation}\label{Ejunc}
		\wh v^{(1)}_\ed(x)=0\quad\hbox{whenever } |x_1-p\e|<r\e\; \hbox{ for some } p\in\{0,\ldots,N_\e\},
	\end{equation}
	and in the same way $\wh v^{(2)}_\ed(x)=0$ whenever $|x_2-q\e|<r\e$ for some $q\in\{0,\ldots,N_\e\}$. In other words, each corrector vanishes near the two ends of every small beam.
	
	Recall from \eqref{Rect} that the overlap of the two families of beams is
	$$
	\O^{(1)}_\ed\cap \O^{(2)}_\ed=\bigcup_{(p,q)\in\cK_\e} R^3_{pq}\X \fI_\d,\qquad R^3_{pq}=(p\e,q\e)+\fI^2_\d,
	$$
	that is, the union of the junction cubes $\big\{|x_1-p\e|<{\d\over 2},\;|x_2-q\e|<{\d\over 2}\big\}\X\fI_\d$ (see Figure \ref{fig:junction}). By \eqref{Esmall} and \eqref{Ejunc}, both $\wh v^{(1)}_\ed$ and $\wh v^{(2)}_\ed$ vanish on this set. Hence the first two lines of \eqref{Ele03} both reduce to $V_{e,\ed}$ on $\O^{(1)}_\ed\cap \O^{(2)}_\ed$. Moreover, $\wh v^{(\a)}_\ed$ vanishes on the interfaces between two consecutive cells of a same beam, by \eqref{Ejunc}, and on the beam-inclusion interfaces $\p\O^H_\ed$, by \eqref{Esupp}. Therefore $v_{\ed} \in H^1(\O_\d)^3$ is well defined.
	
	Furthermore, $v_{\ed}=0$ a.e. on $\Gamma_\d$. Indeed, $V_{e,\ed}=0$ there by the Remark following Lemma \ref{lemAp5} and by Lemma \ref{lemAp7}, and $\Gamma_\d$ is contained in the $\Ge_2$-beam with $p=0$ and the $\Ge_1$-beam with $q=0$, on which $\wh v^{(\a)}_\ed$ is built from the values $\wh v^{(\a)}(0,q\e,\cdot)$, resp. $\wh v^{(\a)}(p\e,0,\cdot)$, which vanish by \eqref{Esupp} since $(0,q\e),\,(p\e,0)\in\p\o$.}

Due to \eqref{Ele01}, we have the following:
$$
e(v_\ed)=e(V_{e,\ed})=0, \quad\text{a.e. in } \O^H_{pq},\quad (p,q)\in\cK_\e^*
$$
since by construction of $\fV_\ed$ we have
$$
\begin{aligned}
	\p_\alpha\fV_{\ed,\b}=\p^2_{\a\b}\fV_{\ed,3}&=0,\quad \text{a.e. in $\o^H_{pq}$ for $(p,q)\in\cK_\e^*$}.
\end{aligned}$$
Observe that the points $\big(p\e\pm{\d\over 2},q\e+z_2,z_3\big)$ with $z_2\in\big({\d\over 2},\e-{\d\over 2}\big)$ lie on the lateral faces $y_1=\pm{1\over 2}$ of the $\Ge_2$-beams, where $\wh v^{(2)}_\ed$ vanishes by \eqref{limW02}; hence $v_\ed=V_{e,\ed}$ there. So, for a.e. $\ds (z_2,z_3)\in \left({\d\over 2},\e-{\d\over 2}\right)\X \fI_\d$ we have
\begin{multline*}
	v_{\ed,1}\Big(p\e+{\d\over 2},q\e+z_2,z_3\Big)-v_{\ed,1}\Big(p\e-{\d\over 2},q\e+z_2,z_3\Big)
	={\d^2}\left( \fV_{\ed,1}\Big(p\e+{\d\over 2},q\e+z_2\Big)-\fV_{\ed,1}\Big(p\e-{\d\over 2},q\e+z_2\Big)\right)\\
	-z_3{\d}\left(\p_1\fV_{\ed,3}\Big(p\e+{\d\over 2},q\e+z_2\Big)-\p_1\fV_{\ed,3}\Big(p\e-{\d\over 2},q\e+z_2\Big)\right).
\end{multline*}
Due to the definition of $\fV_{\ed,3}$ in Lemma \ref{lemAp7}, we have (see \eqref{REDef}) 
$$\left.\begin{aligned}
	&\fV_{\ed,3}(p\e+z_1,q\e+z_2)=\fF_\ed(p\e+z_1)+\fG_\ed(q\e+z_2),\\
	&\p_1\fV_{\ed,3}\Big(p\e+{\d\over 2},q\e+z_2\Big)=d_{1}\fF_\ed\Big(p\e+{\d\over 2}\Big)=d_{1}\fF\Big(p\e+{\e\over 2}\Big),\\
	&\p_1\fV_{\ed,3}\Big(p\e-{\d\over 2},q\e+z_2\Big)=d_{1}\fF_\ed\Big(p\e-{\d\over 2}\Big)=d_{1}\fF\Big(p\e-{\e\over 2}\Big)
\end{aligned}\right\}\quad (z_1,z_2)\in \fI_\d\X\Big({\d\over 2},\e-{\d\over 2}\Big).$$
From the definition of $\fV_{\ed,1}$ in Lemma \ref{lemAp5}, we have (see \eqref{Dphi}$_3$)
$$\fV_{\ed,1}\Big(p\e+{\d\over 2},q\e+z_2\Big)=\fV_1\Big(p\e+{\e\over2},q\e+{\e\over 2}\Big),\quad \fV_{\ed,1}\Big(p\e-{\d\over 2},q\e+z_2\Big)=\fV_1\Big(p\e-{\e\over 2},q\e+{\e\over 2}\Big),\quad \forall z_2\in \Big({\d\over 2},\e-{\d\over 2}\Big).$$
Hence
$$
\begin{aligned}
	&d_{1}\fF\Big(p\e+{\e\over 2}\Big)-d_{1}\fF\Big(p\e-{\e\over 2}\Big)=\int^{p\e+{\e\over 2}}_{p\e-{\e\over 2}}  d^2_{11}\fF(t)\,dt,\\
	&\fV_1\Big(p\e+{\e\over 2},q\e+{\e\over 2}\Big)-\fV_1\Big(p\e-{\e\over2},q\e+{\e\over 2}\Big)=\int^{p\e+{\e\over 2}}_{p\e-{\e\over 2}} \p_1\fV_1\Big(t,q\e+{\e\over 2}\Big)\,dt.
\end{aligned}
$$
So, since $d^2_{11}\fF=\p^2_{11}\fV_3$ is independent of $x_2$ and $|z_3|\leq {\d\over 2}$, for a.e. $\ds (z_2,z_3)\in \left({\d\over 2},\e-{\d\over 2}\right)\X \fI_\d$ we have
$$
\begin{aligned}
	&v_{\ed,1}\Big(p\e+{\d\over 2},q\e+z_2,z_3\Big)-v_{\ed,1}\Big(p\e-{\d\over 2},q\e+z_2,z_3\Big) \\
	=&{\d^2}\int^{p\e+{\e\over 2}}_{p\e-{\e\over 2}} \p_1\fV_1\Big(t,q\e+{\e\over 2}\Big)dt-z_3{\d}\int^{p\e+{\e\over 2}}_{p\e-{\e\over 2}}  \p^2_{11}\fV_3\Big(t,q\e+{\e\over 2}\Big)dt\\
	\geq &{\d^2}\int_{p\e-{\e\over2}}^{p\e+{\e\over 2}}\left(\p_1\fV_1\Big(t,q\e+{\e\over 2}\Big)-{1\over 2}\Big|\p^2_{11}\fV_3\Big(t,q\e+{\e\over 2}\Big)\Big|\right)\,dt\\
	\geq  &-{\d^2}\int_{p\e-{\e\over2}}^{p\e+{\e\over 2}} \Big(\|\p_1\fV_1\|_{L^\infty(\o)}+{1\over 2}\|\p^2_{11}\fV_3\|_{L^\infty(\o)}\Big)\,dt\\
	= & -\e\d^2 \Big(\|\p_1\fV_1\|_{L^\infty(\o)}+{1\over 2}\|\p^2_{11}\fV_3\|_{L^\infty(\o)}\Big).
\end{aligned}
$$
Therefore, if $\e$ and $\d$ are small enough, the non-penetration condition \eqref{NPCM}$_2$ is satisfied by $v_\ed$. Similarly, $v_\ed$ satisfies \eqref{NPCM}$_1$. This implies $v_\ed\in \GU_\ed$.

{\bf Step 3.} In this step, we give convergence of the fields.\\[1mm]
Observe that
\begin{equation}\label{TWL01}
	{1\over \e\d^2}\Pi^{(\a)}_\ed(\wh{v}^{(\a)}_\ed)  \to \wh{v}^{(\a)}\quad \text{strongly in }  \wh{D}^{(\a)}(\o \X Y_\a \X \fI).
\end{equation}
	{ By definition of $\wh v^{(\a)}_\ed$ and \eqref{Pi}, we have $\ds{1\over \e\d^2}\Pi^{(\a)}_\ed(\wh{v}^{(\a)}_\ed)(x',y)=\wh v^{(\a)}\Big(\e\Big[{x'\over \e}\Big],y\Big)$ for a.e. $(x',y)\in\o\X\cY_\a$. Since $\wh v^{(\a)}$ and its first derivatives are uniformly continuous on $\ov{\o}\X \ov{Y}_\a \X \ov{\fI}$ and $\big|\e[x'/\e]-x'\big|\leq \sqrt 2\,\e$, the convergence \eqref{TWL01} holds, even uniformly. The derivative along the beam axis satisfies, by \eqref{ruo02},
	\begin{equation}\label{TWL02}
		{1\over \e\d}\Pi^{(\a)}_\ed\big(\p_\a\wh v^{(\a)}_\ed\big)(x',y)={\d\over \e}\,\p_{y_\a}\wh v^{(\a)}\Big(\e\Big[{x'\over \e}\Big],y\Big),
	\end{equation}
	whose norm in $L^\infty(\o\X \cY_\a)$ is bounded by ${\d\over \e}\|\p_{y_\a}\wh v^{(\a)}\|_{L^\infty(\o\X\cY_\a)}$, hence tends to $0$.}
	
The definition of  $\fV_{\ed,3}$ (see Lemma \ref{lemAp7} in Appendix \ref{SAppB}) give the following convergences:  
\begin{equation}\label{Test01}
	\begin{aligned}
		{\d\over \e}\Ted^{(2)}(\p^2_{11}\fV_{\ed,3}) & \to \p^2_{11}\fF \quad &&\hbox{strongly in } L^2(\o\X Y_2),\\
		{\d\over \e}\Ted^{(2)}(\p^2_{22}\fV_{\ed,3}) & \to 0 \quad &&\hbox{strongly in } L^2(\o\X Y_2),\\
		{\d\over \e}\Ted^{(1)}(\p^2_{22}\fV_{\ed,3}) & \to \p^2_{22}\fG \quad &&\hbox{strongly in } L^2(\o\X Y_1),\\
		{\d\over \e}\Ted^{(1)}(\p^2_{11}\fV_{\ed,3}) & \to 0 \quad &&\hbox{strongly in } L^2(\o\X Y_1),\\
		{\d\over \e}\Ted^{(\a)}(\p_{12}\fV_{\ed,3}) & \to 0 \quad &&\hbox{strongly in } L^2(\o\X Y_\a).
	\end{aligned}
\end{equation}
We also have from Lemma \ref{lemAp5} in Appendix \ref{SAppB}
\begin{equation}\label{EQ810}
	\begin{aligned}
		{\d\over \e}\Ted^{(2)}(\p_1 \fV_{\ed,\a}) & \to \p_1\fV_\a \quad &&\hbox{strongly in } L^2(\o\X Y_2),\\
		{\d\over \e}\Ted^{(2)}(\p_2 \fV_{\ed,\a}) & \to 0 \quad &&\hbox{strongly in } L^2(\o\X Y_2),\\
		{\d\over \e}\Ted^{(1)}(\p_2 \fV_{\ed,\a}) & \to \p_2\fV_\a \quad &&\hbox{strongly in } L^2(\o\X Y_1),\\
		{\d\over \e}\Ted^{(1)}(\p_1 \fV_{\ed,\a}) & \to 0 \quad &&\hbox{strongly in } L^2(\o\X Y_1).
	\end{aligned}
\end{equation}
{\bf Step 4.} In this step, we present the convergence of the strain tensor and the right hand side.\\[1mm]
Observe that
\begin{equation}\label{eq:s4}
	 \Pi_\ed^{(2)}\big(e(V_{e,\ed})\big)= {\d^2}\begin{pmatrix*}
		\Ted^{(2)}(e_{11}(\fV_{\ed}))-y_3\Ted^{(2)}(\p^2_{11}\fV_{\ed,3}) & *& *\\[2mm]
		 \Ted^{(2)}(e_{12}(\fV_{\ed}))-y_3\Ted^{(2)}(\p^2_{12}\fV_{\ed,3})& \Ted^{(2)}(e_{22}(\fV_{\ed}))-y_3\Ted^{(2)}(\p^2_{22}\fV_{\ed,3})& *\\[1mm]
		0 & 0 & 0
	\end{pmatrix*},
\end{equation}
a.e. in $\o\X \cY_2$, which along with the convergences \eqref{TWL01}, \eqref{TWL02}, \eqref{Test01}$_{1,2,5}$ and \eqref{EQ810}$_{1,2}$ give
\begin{equation*}
	\begin{aligned}
		{1\over \e\d}\Pi_\ed^{(2)}\big(e(v_\ed)\big) \to E^{(2)}(\fV,\wh v^{(2)})\quad \text{strongly in $L^2(\o\X \cY_2)^{3\X3}$}.
	\end{aligned}
\end{equation*}
Similarly, using the convergences \eqref{TWL01}, \eqref{TWL02}, \eqref{Test01}$_{3,4,5}$ and \eqref{EQ810}$_{3,4}$
$$		{1\over \e\d}\Pi_\ed^{(1)}\big(e(v_\ed)\big) \to E^{(1)}(\fV,\wh v^{(1)})\quad \text{strongly in $L^2(\o\X \cY_1)^{3\X3}$}.$$
With definition of forces \eqref{Ass01} and $V_{e,\ed}$ \eqref{Ele01}, we get
$$
\lim{1\over \e\d^4}\int_{\O_\d}f_\ed\cdot v_{\ed}\,dx=\int_\o f\cdot\fV\,dx',
$$
since from Lemma \ref{lemAp5} and \ref{lemAp7}, we have
$$\begin{aligned}
	\fV_{\ed,\a}&\to \fV_\a,\quad &&\text{strongly in $L^2(\o)$},\\
	\fV_{\ed,3} &\to\fV_3,\quad &&\text{strongly in $H^1(\o)$}.
\end{aligned}$$
This completes the proof. 
\end{proof}

\subsection{Unfolded limit problem}
\begin{theorem}\label{Th82}
	Let $u_{\ed}\in\GU_\ed$ be the solution of the elasticity problem \eqref{Pb}. Then, for the whole sequence $\{\ed\}_{\e,\d}$ the following convergences hold:
	\begin{equation}\label{814}
		\begin{aligned}
			{1\over \ed}\Pi_\ed^{(1)}\big(e(u_\ed)\big)&\to E^{(1)}(\fU,\wh u^{(1)}) \quad &&\hbox{strongly in } L^2(\o\X \cY_1)^{3\X3},\\
			{1\over \ed}\Pi_\ed^{(2)}\big(e(u_\ed)\big)&\to E^{(2)}(\fU,\wh u^{(2)})  \quad &&\hbox{strongly in } L^2(\o\X \cY_2)^{3\X3},
		\end{aligned}
	\end{equation} 
	and $(\fU,\wh u^{(1)},\wh u^{(2)})\in \D$ is the solution of the following rescaled unfolded problem:
	\begin{equation}\label{ULP01}
	\sum_{\a=1}^2\int_{\o\X \cY_\a}\A^{(\a)}_{ijkl}E^{(\a)}_{ij}(\fU,\wh u^{(\a)})E^{(\a)}_{kl}(\fV, \wh v^{(\a)})\,dx'dy =\int_\o f\cdot \fV\,dx',\quad\,\forall\,(\fV,\wh v^{(1)},\wh v^{(2)})\in\D.
	\end{equation}
\end{theorem}
\begin{proof}
	The proof is given as follows.
	
{\bf Step 1.} We prove the convergences \eqref{814} for the whole sequence in weak sense and \eqref{ULP01}.

Since $u_\ed\in\GU_\ed$, then using Theorem \ref{TH42} and the assumption on forces \eqref{Ass01}, we have \eqref{Main04} and \eqref{EQ61}. Then, we have the convergences as in Lemmas \ref{lem61}--\ref{L64}, so we got the convergence at least for a subsequence
\begin{equation}\label{814++}
	\begin{aligned}
	{1\over \ed}\Pi_\ed^{(1)}\big(e(u_\ed)\big)&\rightharpoonup E^{(1)}(\fU,\wh u^{(1)}) \quad &&\hbox{weakly in } L^2(\o\X \cY_1)^{3\X3},\\
	{1\over \ed}\Pi_\ed^{(2)}\big(e(u_\ed)\big)&\rightharpoonup E^{(2)}(\fU,\wh u^{(2)})  \quad &&\hbox{weakly in } L^2(\o\X \cY_2)^{3\X3},
\end{aligned}	
\end{equation}
and $(\fU,\wh u^{(1)},\wh u^{(2)})\in\D$. 

Multiplying \eqref{Pb} by $\ds{1\over \e\d^4}$ then, using Lemma \ref{Lem81} and the convergences \eqref{EQ70}, \eqref{EQ70+} and \eqref{STL01}, the weak lower semi-continuity of the LHS of \eqref{Pb}, we get
\begin{equation}\label{WLSC01}
	\begin{aligned}
		\sum_{\a=1}^2&\int_{\o\X \cY_\a}\A^{(\a)}_{ijkl}E^{(\a)}_{ij}(\fU,\wh u^{(\a)})E^{(\a)}_{kl}(\fU,\wh u^{(\a)})\,dx'dy\\
			&\leq\liminf\sum_{\a=1}^2\int_{\o\X\cY_\a} \1_{(\o\setminus\o^J_\ed)\X\cY_\a}
			\Pi^{(\a)}_\ed(\A_{\ed,ijkl}){1\over \ed}\Pi^{(\a)}_\ed(e_{ij}(u_\ed)){1\over \ed}\Pi^{(\a)}_\ed(e_{kl}(u_\ed))\,dx'dy\\
&\leq \liminf{1\over \e\d^4}\int_{\O^S_{\e\d}}\A_{\ed,ijkl}(x)e_{ij}(u_\ed)e_{kl}(u_\ed)\,dx\\
&\leq  \limsup{1\over \e\d^4}\int_{\O^S_{\e\d}}\A_{\ed,ijkl}(x)e_{ij}(u_\ed)e_{kl}(u_\ed)\,dx\\ 
		&\leq \limsup{1\over \e\d^4}\int_{\O_{\d}}f_{\e\d}\cdot (u_\ed-v_\ed)\,dx+\limsup{1\over \e\d^4}\int_{\O^S_{\e\d}}\A_{\ed,ijkl}(x)e_{ij}(u_\ed)e_{kl}(v_\ed)\,dx\\
		&= \int_\o f\cdot(\fU-\fV)\,dx'+\sum_{\a=1}^2\int_{\o\X \cY_\a}\A^{(\a)}_{ijkl}E^{(\a)}_{ij}(\fU,\wh u^{(\a)})E^{(\a)}_{kl}(\fV,\wh v^{(\a)})\,dx'dy,\quad\,\forall\,(\fV,\wh v^{(1)},\wh v^{(2)})\in\D.
	\end{aligned}
\end{equation}
Since,
\begin{multline*}
	{1\over \e\d^4}\int_{\O^S_{\e\d}}\A_{\ed,ijkl}(x)e_{ij}(u_\ed)e_{kl}(v_\ed)\,dx\\
	=\int_{\o\X\cY_1} \Pi^{(1)}_\ed(\A_{\ed,ijkl}){1\over \ed}\Pi^{(1)}_\ed(e_{ij}(u_\ed)){1\over \ed}\Pi^{(1)}_\ed(e_{kl}(v_\ed))\,dx'dy\\
	+\int_{\o\X\cY_2} \1_{(\o\setminus\o^J_\ed)\X\cY_2}\Pi^{(2)}_\ed(\A_{\ed,ijkl}){1\over \ed}\Pi^{(2)}_\ed(e_{ij}(u_\ed)){1\over \ed}\Pi^{(2)}_\ed(e_{kl}(v_\ed))\,dx'dy
\end{multline*}
convergences by weak–strong 
convergence of the solution sequence and the recovery sequence, respectively and  the junction domain term $\o^J_\ed\X\cY_\a$ has measure of order $\GO(\delta^2/\varepsilon^2)$, where $\o^J_\ed\defeq(\o^{(1)}_{\e\d}\cap\o^{(2)}_{\e\d})$, so $\1_{(\o\setminus\o^J_\ed)\X\cY_\a}\to 1$ a.e.~in $\o\X\cY_\a$, see the detailed analogous argument in~\cite{Amartya02}, Proposition~7.1 to get the first inequality in \eqref{WLSC01}.

Observe that \eqref{WLSC01} gives
\begin{multline*}
		\sum_{\a=1}^2\int_{\o\X \cY_\a}\A^{(\a)}_{ijkl}E^{(\a)}_{ij}(\fU,\wh u^{(\a)})E^{(\a)}_{kl}(\fU-\fV,\wh u^{(\a)}-\wh v^{(\a)})\,dx'dy\\
	\leq \int_\o f\cdot(\fU-\fV)\,dx',\quad\,\forall\,(\fV,\wh v^{(1)},\wh v^{(2)})\in\D.
\end{multline*}
Since the set of limit displacement $\D$ is a subspace without any limit non-penetration condition, we can take $2\fU-\fV$ and $2\wh u^{(\a)}-\wh v^{(\a)}$ as test function to get
	 \begin{multline}\label{eq:814}
		\sum_{\a=1}^2\int_{\o\X \cY_\a}\A^{(\a)}_{ijkl}E^{(\a)}_{ij}(\fU,\wh u^{(\a)})E^{(\a)}_{kl}(\fU-\fV,\wh u^{(\a)}-\wh v^{(\a)})\,dx'dy\\
		\geq \int_\o f\cdot(\fU-\fV)\,dx',\quad\,\forall\,(\fV,\wh v^{(1)},\wh v^{(2)})\in\D.
	\end{multline} 
 So, we have proved \eqref{ULP01}, by replacing $\fU-\fV$ and $\wh u^{(\a)}-\wh v^{(\a)}$ by $\fV$ and $\wh v^{(\a)}$ respectively.
	
Using the coercivity result \eqref{CoeL02} together with \eqref{UL78} and using Lax Milgram’s lemma, we have unique solution to the unfolded problem \eqref{ULP01}. So, the convergences \eqref{814++} hold for the whole sequence.

{\bf Step 2.} We prove the strong convergence \eqref{814}.

 Multiplying \eqref{Pb} by $\ds{1\over \e\d^4}$ then, using Lemma \ref{Lem81} and the convergences \eqref{EQ70}, \eqref{EQ70+} and \eqref{STL01}, the weak lower semi-continuity of the LHS of \eqref{Pb} (with $v=0$), we get
	\begin{equation*}
		\begin{aligned}
			\sum_{\a=1}^2&\int_{\o\X \cY_\a}\A^{(\a)}_{ijkl}E^{(\a)}_{ij}(\fU,\wh u^{(\a)})E^{(\a)}_{kl}(\fU,\wh u^{(\a)})\,dx'dy\\
			&\leq\liminf\sum_{\a=1}^2\int_{\o\X\cY_\a}
			\1_{(\o\setminus\o^J_\ed)\X\cY_\a}
			 \Pi^{(\a)}_\ed(\A_{\ed,ijkl}){1\over \ed}\Pi^{(\a)}_\ed(e_{ij}(u_\ed)){1\over \ed}\Pi^{(\a)}_\ed(e_{kl}(u_\ed))\,dx'dy\\
			&\leq \liminf{1\over \e\d^4}\int_{\O^S_{\e\d}}\A_{\ed,ijkl}(x)e_{ij}(u_\ed)e_{kl}(u_\ed)\,dx\\
			&\leq  \limsup{1\over \e\d^4}\int_{\O^S_{\e\d}}\A_{\ed,ijkl}(x)e_{ij}(u_\ed)e_{kl}(u_\ed)\,dx\\ 
			&\leq \lim{1\over \e\d^4}\int_{\O_{\d}}f_{\e\d}\cdot u_\ed\,dx
			= \int_\o f\cdot \fU\,dx'.
		\end{aligned}
	\end{equation*} 
	According to \eqref{ULP01}, the above inequalities are equalities, which leads to the strong convergences \eqref{814}. 

This completes the proof.
\end{proof}

\subsection{Cell problems}
In this subsection, we derive the homogneized problem. For that we express the microscopic displacements $\wh u^{(\a)}$ in-terms of the macroscopic displacement $\fU$ and some correctors. 

In the problem \eqref{ULP01} taking $\fV=0$, we get
\begin{equation*}
\int_{\cY_\a}\A^{(\a)}_{ij}E^{(\a)}_{ijkl}(0,\wh u^{(\a)})E^{(\a)}_{kl}(0,\wh v^{(\a)})=-\int_{\cY_\a}\A^{(\a)}_{ijkl}E^{(\a)}_{ij}(\fU,0)E^{(\a)}_{kl}(0,\wh v^{(\a)})\,dy,\quad
 \forall\, \wh v^{(\a)}\in \wh{D}^{(\a)}(\o \X Y_\a \X \fI ).
\end{equation*}
 We define the following $3\X3$ symmetric matrices by
$$
\begin{aligned}
\GM^{(1)}_1&=\begin{pmatrix}
0 & \ds{1/ 2} & 0\\
\ds{1/2} & 0 & 0\\
0 & 0 & 0
\end{pmatrix},\quad &&\GM^{(1)}_2=\begin{pmatrix}
0 & 0 & 0\\
0 & 1 & 0\\
0 & 0 & 0
\end{pmatrix},\quad &&\GM^{(1)}_3=\begin{pmatrix}
0 & 0 & 0\\
0 & -y_3 & 0\\
0 & 0 & 0
\end{pmatrix},\\
\GM^{(2)}_1&=\begin{pmatrix}
1 & 0 & 0\\
0 & 0 & 0\\
0 & 0 & 0
\end{pmatrix},\quad &&\GM^{(2)}_2=\begin{pmatrix}
0 & {1/2} & 0\\
{1/2} & 0 & 0\\
0 & 0 & 0
\end{pmatrix},\quad &&\GM^{(2)}_3=\begin{pmatrix}
-y_3 & 0 & 0\\
0 & 0 & 0\\
0 & 0 & 0
\end{pmatrix}.
\end{aligned}
$$ Let  $\ds\chi^{(\a)}_r$ be in $\wh{D}^{(\a)}(\o \X Y_\a \X \fI )$,  $r=1,2,3$, the solutions of the following cell problems ($\a\in\{1,2\}$):
\begin{equation}\label{Cell01}
	\left.\begin{aligned}
		\int_{\cY_\a}\A_{ijkl}^{(\a)}(y)\big(\GM_{r,ij}^{(\a)}+E^{(\a)}_{ij}(0,\chi^{(\a)}_r)\big)E^{(\a)}_{kl}(0,\wh v^{(\a)})= 0,		
	\end{aligned}\right.\quad\forall\,\wh v^{(\a)}\in \wh{D}^{(\a)}(\o \X Y_\a \X \fI ).
\end{equation}
The six  correctors above allow us to express $\wh u^{(\a)}$ in terms of the partial derivatives of $\fU$.  We obtain
\begin{equation*}
	\begin{aligned}
		\wh u^{(1)}(x',y)&=\p_2\fU_1\chi^{(1)}_{1}(y)+\p_2\fU_2\chi^{(1)}_{2}(y)+\p^2_{22}\fU_3\chi^{(1)}_{3}(y),\quad &&\text{a.e. in $\o\X\cY_1$},\\
		\wh u^{(2)}(x',y)&=\p_1\fU_2\chi^{(2)}_{1}(y)+\p_1\fU_1\chi^{(2)}_{2}(y)+\p^2_{11}\fU_3\chi^{(2)}_{3}(y),\quad &&\text{a.e. in $\o\X\cY_2$},
	\end{aligned}
\end{equation*}
We define the homogenized coefficients as
\begin{equation}\label{HomC01}
	    \begin{aligned}
	    \cA_{rs}^{(\a)}&=\int_{\cY_\a}\A^{(\a)}_{ijkl}(y)\left[\GM_{r,ij}^{(\a)}+E^{(\a)}_{ij}(0,\chi_r^{(\a)})\right]\GM_{s,kl}^{(\a)}\,dy\\
	    &=\int_{\cY_\a}\A^{(\a)}_{ijkl}(y)\left[\GM_{r,ij}^{(\a)}+E^{(\a)}_{ij}(0,\chi_r^{(\a)})\right]\left[\GM_{s,kl}^{(\a)}+E^{(\a)}_{kl}(0,\chi_s^{(\a)})\right]\,dy
	\end{aligned}\qquad\text{for $r,s\in\{1,2,3\}$}.
\end{equation}
Hence, by setting $\cA^{(\a)}=[\cA^{(\a)}_{rs}]\in\GM_3$, we get the homogenized varitional inequality, given by
\begin{equation*}
	\begin{aligned}
		&\sum_{\a=1}^2\int_{\o}\cA^{(\a)}E^{(\a)}(\fU)\cdot E^{(\a)}(\fV)\,dx'= \int_\o f\cdot\fV\,dx',\quad\forall\,\fV\in\D_0,
	\end{aligned}
\end{equation*}
where
\begin{equation}\label{HomH01}
	\begin{aligned}
		E^{(1)}(\fV)=\begin{pmatrix*}
			\p_2\fV_1\\
			\p_2\fV_2\\
			\p^2_{22}\fV_3
		\end{pmatrix*},\quad E^{(2)}(\fV)=\begin{pmatrix*}
		\p_1\fV_2\\
		\p_1\fV_1\\
		\p^2_{11}\fV_3
		\end{pmatrix*},\quad\forall\,\fV\in\D_0.
	\end{aligned}
\end{equation}
The following is the final result of this paper
\begin{theorem}
	The limit macroscopic displacements $\fU\in\D_0$ is the unique solution to the homogenized problem
	\begin{equation*}
		\begin{aligned}
			\sum_{\a=1}^2\int_{\o}\cA^{(\a)}E^{(\a)}(\fU)\cdot E^{(\a)}(\fV)\,dx'= \int_\o f\cdot\fV\,dx',\quad\forall\,\fV\in\D_0,
		\end{aligned}
	\end{equation*}
	where the homogenized tensor is given by \eqref{HomH01} with the coefficients \eqref{HomC01}.
\end{theorem}
\begin{proof}
	Observe that for $\xi\in\R^3$, we have, using  \eqref{Cell01}-\eqref{HomC01}) 
	$$
	\begin{aligned}
	\cA^{(\a)}\xi\cdot \xi =\int_{\cY_\a}\A^{(\a)}_{ijkl}(y)\left[\xi_r\GM_{r,ij}^{(\a)}+E^{(\a)}_{ij}(0,\xi_r\chi_r^{(\a)})\right]\left[\xi_s\GM_{s,kl}^{(\a)}+E^{(\a)}_{kl}(0,\xi_s\chi_s^{(\a)})\right]\,dy.
	\end{aligned}
	$$
	Then, using \eqref{HC01}--\eqref{UL78}, we get $\cA^{(\a)}$ is symmetric, non-negative and coercive. Then, the estimate \eqref{CoeL02} and using Lax-Milgram's lemma with boundary condition satisfied by elements in the subspace $\D_0$ give that there exist a unique solution to the homogenized problem.
	Finally, using the Theorem \eqref{Th82} along with cell problems completes the proof.
\end{proof}

\bibliographystyle{plain}
\bibliography{reference}

\appendix
\section{Appendix: Proof of Theorem \ref{TH42}}\label{AppA01}
\begin{proof} 
	The proof is divided into three steps given below.
	
	{\it Step 1.} Below we recall two theorems concerning the decomposition of plate displacements.\\[1mm]
The first statements of the  theorem below is found in \cite[Theorem 2.3]{GGDecomplPlate} and \cite[Theorem 4.1]{GDecomp}, and it was completed by \cite[Theorem 11.4]{PUM}.
\begin{theorem}[First decomposition of a plate displacement]\label{THA1} Any  displacement $u$ in $H^1(\O_{\d})^3$ can be decomposed as follows: 
\begin{equation}\label{EQ41DA1}
u(x)= \begin{pmatrix}
\ds \cU_1(x')+x_3 \cR_1 (x')\\[1.5mm]
\ds\cU_2(x')+x_3 \cR_2 (x')\\[1.5mm]
\cU_3(x')\\
\end{pmatrix}+\ov{u}(x)\quad  \hbox{ for a.e. $x$ in }  \Omega_\delta.
\end{equation} where $\cU\in H^1(\o_{\d})^3$, $\cR_\a \in H^2(\o_{\d})$ and  $\ov{u}\in H^1(\O_\d)^3$ satisfies
\begin{equation}\label{WMC01A1}
	\int_{\fI_\d} \ov{u}\, dx_3=0,\qquad \int_{\fI_\d} x_3\ov{u}_\alpha\, dx_3=0,\qquad \hbox{a.e. in }\o_\d.
\end{equation}
We set 
$$\cU_m=\cU_1\Ge_1+\cU_2\Ge_2,\qquad \cR=\cR_1\Ge_1+\cR_2\Ge_2.$$ 
The following estimates hold:
	\begin{equation}\label{remA1}
	\begin{aligned}
	&\|e_{\alpha\beta}(\cU_m)\|_{L^2(\o)}+\d\big\|\nabla \cR\big\|_{L^2(\o)} \leq {C\over \d^{1/2}}\|e(u)\|_{L^2(\O_\d)},\\
	&\big\|\nabla \cU_3+\cR\big\|_{L^2(\o_\d)}\leq {C\over \d^{1/2}}\|e(u)\|_{L^2(\O_\d)},\\
	&\|\overline{u}\|_{L^2(\O_\d )}+\d\|\nabla \overline{u}\|_{L^2(\O_\d)}\le C \d \|e(u)\|_{L^2(\O_\d)}.
	\end{aligned}
	\end{equation}
	The constants  depend only on $\o$.\\
	Moreover, if $u=0$ a.e. on $\Gamma_\d$ then, the terms of the decomposition \eqref{EQ41DA1} satisfy
$$\cU=0,\quad \cR=0\quad \hbox{and}\quad \ov{u}=0\qquad \hbox{a.e. on } \gamma.$$	
\end{theorem}
Theorem below  is an immediate consequence of  \cite[Theorem 2]{GGKL}.
\begin{theorem}[Complete decomposition of a plate displacement]\label{THA2} Any  displacement $u$ in $H^1(\O_{\d})^3$ can be decomposed as follows: 
\begin{equation}\label{EQ41DA2}
u(x)= \begin{pmatrix}
\ds \fU_1(x')-x_3 \p_1\fU_3 (x')+x_3\fr_1(x')\\[1.5mm]
\ds\fU_2(x')-x_3 \p_2 \fU_3 (x')+x_3\fr_2(x')\\[1.5mm]
\fU_3(x')+\fu(x')\\
\end{pmatrix}+\ov{u}(x)\quad  \hbox{ for a.e. $x$ in }  \Omega_\delta.
\end{equation} where $\fU_\a \in H^1(\o_{\d})$, $\fU_3 \in H^2(\o_{\d})$,  $\fr_\a,\; \fu \in H^1(\o_\d)$ and  $\ov{u}\in H^1(\O_\d)^3$ satisfies
\begin{equation}\label{WMC01A2}
	\int_{\fI_\d} \ov{u}\, dx_3=0,\qquad \int_{\fI_\d} x_3\ov{u}_\alpha\, dx_3=0,\qquad \hbox{a.e. in }\o_\d.
\end{equation}
We set
$$\fU_m=\fU_1\Ge_1+\fU_2\Ge_2,\qquad  \fr=\fr_1\Ge_1+\fr_2\Ge_2 \in H^1(\o_\d)^2.$$ 
The following estimates hold:
	\begin{equation}\label{remA2}
	\begin{aligned}
	&\|e_{\alpha\beta}(\fU_m)\|_{L^2(\o)}+\d\big\|\p^2_{\alpha\beta}\fU_3\big\|_{L^2(\o)}\le {C\over \d^{1/2}}\|e(u)\|_{L^2(\O_\d)},\\
	&\|\fr\|_{L^2(\o_\d)}+\d \big\|\nabla \fr\big\|_{L^2(\o_\d)} \leq {C\over \d^{1/2}}\|e(u)\|_{L^2(\O_\d)},\\
	&\|\fu\|_{L^2(\o_\d)}+\d \big\|\nabla \fu\big\|_{L^2(\o_\d)} \leq {C \d^{1/2}}\|e(u)\|_{L^2(\O_\d)},\\
	&\|\overline{u}\|_{L^2(\O_\d )}+\d\|\nabla \overline{u}\|_{L^2(\O_\d)}\le C \d \|e(u)\|_{L^2(\O_\d)}.
	\end{aligned}
	\end{equation}
	The constants  depend only on $\o$.\\
	Moreover, if $u=0$ a.e. on $\Gamma_\d$ then, the terms of the decomposition \eqref{EQ41DA2} satisfy
$$\fU_m=0,\quad \fU_3=0,\quad \nabla\fU_3=0, \quad \fr=0,\quad \fu=0,\quad \ov{u}=0\qquad \hbox{a.e. on } \gamma.$$	
\end{theorem}
We proved Theorem \ref{THA1} from Theorem \ref{THA2} by setting $ \fU_\a=\cU_\a$ and using the same deformation of the fibers $\ov{u}$, then constructed the field $\fU_3\in H^2(\o_\d)$ and defined $\fr\in H^1(\o_\d)^2$ and $\fu\in H^1(\o_\d)$ to get
$$\cU_3=\fU_3+\fu,\qquad \cR=-\nabla \fU_3+\fr.$$

\begin{figure}[ht]
		\centering
		\begin{tikzpicture}[line cap=round,line join=round, scale=0.3]
			\def\sideD{16}   
			\def\sideC{12}   
			\def\gap{1.0}    
			
			\pgfmathsetmacro{\sideB}{\sideC - 2*\gap}
			\pgfmathsetmacro{\sideA}{\sideC - 4*\gap}
			
			\path[fill=red!50, even odd rule]
			(-\sideC/2,-\sideC/2) rectangle (\sideC/2,\sideC/2)
			(-\sideA/2,-\sideA/2) rectangle (\sideA/2,\sideA/2);
			\fill[blue!50] (-\sideA/2,-\sideA/2) rectangle (\sideA/2,\sideA/2);
			
			\draw[thin]       (-\sideD/2,-\sideD/2) rectangle (\sideD/2,\sideD/2); 
			\draw[thin]       (-\sideC/2,-\sideC/2) rectangle (\sideC/2,\sideC/2); 
			\draw[very thick] (-\sideB/2,-\sideB/2) rectangle (\sideB/2,\sideB/2); 
			\draw[thin]       (-\sideA/2,-\sideA/2) rectangle (\sideA/2,\sideA/2); 
			
			\draw[->] (\sideD/2,0) -- ++(2,0)  node[right] {$\cO^{E}_{\ed}$};
			\draw[->] (0,-\sideC/2) -- ++(0,-2) node[below] {$\cO_\ed$};
			\draw[->] (-\sideB/2,0) -- ++(-2,0) node[left]  {$\cO_{\e}$};
			\draw[->] (0,\sideA/2) -- ++(0,2)   node[above] {$\cO^{H}_{\ed}$};
		\end{tikzpicture}
	\caption{The different open sets: $\cO^H_\ed$ (blue), $\cO_\e$, $\cO_\ed$ (red+blue) and $\cO^E_\ed$ (empty+red+blue).}
	\label{F3}
\end{figure}

{\it Step 2.} In this step we explain how to modify  $\cU_3$ in a cell.\\[1mm]
Denote (see Figure \ref{F3})
$$
\cO_\e\doteq (0,\e)^2,\qquad \cO_\ed\doteq \Big(-{\d\over 2},\e+{\d\over 2}\Big)^2,\qquad \cO^H_\ed\doteq \Big({\d\over 2},\e-{\d\over 2}\Big)^2,\qquad  \cO^E_\ed\doteq \Big(-{3\d\over 2},\e+{3\d\over 2}\Big)^2.
$${\it \large For the sake of simplicity we assume that $\ds {\e\over \d}$ is an integer greater than 4. } \\[1mm]
We recall that there exists an extension operator $P$ from  $H^1(\cO_\ed\X \fI_\d)^3$ into $H^1(\cO^E_\ed\X\fI_\d)^3$ such that (see \cite[Lemma 4.2]{GDecomp})
$$\forall v\in H^1(\cO_\ed\X \fI_\d)^3,\qquad \big\|e\big(P(v)\big)\big\|_{L^2(\cO^E_\ed\X \fI_\d)} \leq C \big\|e(v)\big\|_{L^2(\cO_\ed\X \fI_\d)}.$$
The constant is independent of $\e$ and $\d$.

Below, since $v$ satisfies $e(v)=0$ a.e. in $\cO^H_\ed\X \fI_\d$ we have $v=\GR$ in $\cO^H_\ed\X \fI_\d$, where $\GR$ is a rigid displacement. To extend $v$, we consider $v-\GR$ and  first make reflections with respect to two opposite sides of   $\p\cO_\ed\X \fI_\d$, then reflections with respect to the  other two sides of this domain, finally adding the rigid displacement $\GR$ to the result.\footnotemark\\[1mm]
\footnotetext{$ $ In the third step, using this result, we will extend the restriction of a displacement $u$ to $\ds \O_{pq,\ed}\doteq  \Big(p\e-{\d\over 2}, p\e+\e+{\d\over 2}\Big)\X\Big(q\e-{\d\over 2}, q\e+\e+{\d\over 2}\Big)\X\fI_\d$, $(p,q)\in \cK^*_\e$, into a displacement of $\ds \Big(p\e-{3\d\over 2}, p\e+\e+{3\d\over 2}\Big)\X\Big(q\e-{3\d\over 2}, q\e+\e+{3\d\over 2}\Big)\X\fI_\d$, if necessary. For simplicity, the extensions will  always be denoted $u$.}
Let  $v$ be a displacement in $H^1(\cO_\ed\X \fI_\d)^3$, satisfying $e(v)=0$ a.e. in $\cO^H_{\ed}\X \fI_\d$ and extended in a displacement, still denoted $v$, belonging to $H^1(\cO^E_{\ed}\X \fI_\d)^3$. Note that with this displacement, applying  Theorem \ref{THA1}, gives estimates with constants independent of $\e$ and $\d$. We have $v=\GR$ in $\cO^H_\ed\X \fI_\d$, where $\GR$ is a rigid displacement. Below, we consider $v$ as a displacement of $\cO^E_\ed\X \fI_\d$. Theorem \ref{THA1} allows to decompose $v$ as 
\begin{equation}\label{EQ41DA1+}
v(x)= \begin{pmatrix}
\ds \cV_1(x')+x_3 \cS_1 (x')\\[1.5mm]
\ds\cV_2(x')+x_3 \cS_2 (x')\\[1.5mm]
\cV_3(x')\\
\end{pmatrix}+\ov{v}(x)\quad  \hbox{ for a.e. $x$ in }  \cO_\ed\X \fI_\d.
\end{equation}
We have 
$$\cV(x')+x_3\cS(x')=\GR(x)\quad \hbox{ a.e. in } x \in \cO^H_\ed, \qquad \ov{v}=0\quad \hbox{a.e. in } \cO^H_\ed\X \fI_\d.$$

We cover $\cO_\ed$ with squares whose edges have length $\d$. 
	\begin{figure}[ht]
		\centering
		
		\begin{subfigure}{0.28\textwidth}
			\centering
			\begin{tikzpicture}[line cap=round,line join=round, scale=0.3]
				\def\sideC{12}   
				\def\gap{1.0}    
				\def\n{6}        
				
				\pgfmathsetmacro{\sideB}{\sideC - 2*\gap}
				\pgfmathsetmacro{\sideA}{\sideC - 4*\gap}
				
				\path[fill=red!50, even odd rule]
				(-\sideC/2,-\sideC/2) rectangle (\sideC/2,\sideC/2)
				(-\sideA/2,-\sideA/2) rectangle (\sideA/2,\sideA/2);
				
				\fill[blue!50]
				(-\sideA/2,-\sideA/2) rectangle (\sideA/2,\sideA/2);
				
				\foreach \i in {1,...,\numexpr\n-1}{
					\draw (-\sideC/2 + \i*\sideC/\n, -\sideC/2) --
					(-\sideC/2 + \i*\sideC/\n,  \sideC/2);
					\draw (-\sideC/2, -\sideC/2 + \i*\sideC/\n) --
					(\sideC/2,  -\sideC/2 + \i*\sideC/\n);
				}
				
				\draw[thin]       (-\sideC/2,-\sideC/2) rectangle (\sideC/2,\sideC/2); 
				\draw[very thick] (-\sideB/2,-\sideB/2) rectangle (\sideB/2,\sideB/2); 
				\draw[thin]       (-\sideA/2,-\sideA/2) rectangle (\sideA/2,\sideA/2); 
			\end{tikzpicture}
			\caption{Big square with nested regions}
		\end{subfigure}
		\hspace{1cm}
		\begin{subfigure}{0.28\textwidth}
			\centering
			\begin{tikzpicture}[scale=0.8]
				\def\s{1}   
				\def\dx{3.0}  
				\def\dy{3.0}  
				
				\begin{scope}[shift={(0*\dx,0*\dy)}]
					\draw (0,0) grid (2*\s,2*\s);
					\fill[red!50] (0,0) rectangle (\s,\s);
				\end{scope}
				
				\begin{scope}[shift={(1*\dx,0*\dy)}]
					\draw (0,0) grid (2*\s,2*\s);
					\fill[red!50] (\s,0) rectangle (2*\s,\s);
				\end{scope}
				
				\begin{scope}[shift={(0*\dx,-1*\dy)}]
					\draw (0,0) grid (2*\s,2*\s);
					\fill[red!50] (\s,\s) rectangle (2*\s,2*\s);
				\end{scope}
				
				\begin{scope}[shift={(1*\dx,-1*\dy)}]
					\draw (0,0) grid (2*\s,2*\s);
					\fill[red!50] (0,\s) rectangle (\s,2*\s);
				\end{scope}
			\end{tikzpicture}
			\caption{Four \(2\times2\) grids arranged compactly.}
		\end{subfigure}
		\hspace{1cm}
			\begin{subfigure}{0.28\textwidth}
			\centering
			\begin{tikzpicture}[line cap=round,line join=round, scale=0.3]
				\def\sideC{12}
				\def\gap{1.0}
				\def\n{6}
				\pgfmathsetmacro{\sideB}{\sideC - 2*\gap}
				\pgfmathsetmacro{\sideA}{\sideC - 4*\gap}
				\def\cell{\sideC/\n}
				
				\path[fill=red!50, even odd rule]
				(-\sideC/2,-\sideC/2) rectangle (\sideC/2,\sideC/2)
				(-\sideA/2,-\sideA/2) rectangle (\sideA/2,\sideA/2);
				\fill[blue!50] (-\sideA/2,-\sideA/2) rectangle (\sideA/2,\sideA/2);
				
				\fill[pattern=north east lines, pattern color=black]
				(-\sideA/2, \sideA/2-\cell) rectangle (\sideA/2, \sideA/2);
				\fill[pattern=north east lines, pattern color=black]
				(\sideA/2-\cell, -\sideA/2) rectangle (\sideA/2, \sideA/2);
				
				\foreach \i in {1,...,\numexpr\n-1}{
					\draw (-\sideC/2 + \i*\cell, -\sideC/2) --
					(-\sideC/2 + \i*\cell,  \sideC/2);
					\draw (-\sideC/2, -\sideC/2 + \i*\cell) --
					(\sideC/2,  -\sideC/2 + \i*\cell);
				}
				
				\draw[thin]       (-\sideC/2,-\sideC/2) rectangle (\sideC/2,\sideC/2); 
				\draw[very thick] (-\sideB/2,-\sideB/2) rectangle (\sideB/2,\sideB/2); 
				\draw[thin]       (-\sideA/2,-\sideA/2) rectangle (\sideA/2,\sideA/2); 
			\end{tikzpicture}
			\caption{First strategy: hatched region where no rigid displacement is obtained.}
			\label{F4}
		\end{subfigure}
		\caption{$\mathcal{O}_{\ed}$ is covered with squares; below are four strategies for defining $\mathcal{V}_3$.}
		\label{F2}
	\end{figure}
To define $\fV_3$, we have four possible strategies. In the first, we consider four cells included in $\cO^E_\ed$, the first four cells under the domain $\cO_\ed$ (see Figure \ref{F2}), such that the red cell is included in $\cO_\ed$. Then, we consider the mean values in these four cells of $\cV_3$, $\cS_1$ and $\cS_2$. Thanks to these values we define $\fV_3$ in the red cell (see  \cite[Section 8]{GGKL}). In this way, and because we have extended $v$ to the domain $\cO^E_\ed\X \fI_\d$, we obtain a function denoted $\fV_{3,A}$. The values obtained for $\fV_{3,A}$ do not allow to obtain a Kirchhoff-Love displacement equals to $\GR$  in the hatched part of  $\cO^H_\ed\X \fI_\d$ (see Figure \ref{F4}).

Then, we use the other possibilities to construct $\fV_3$ (see Figure \ref{F2}). This gives the  fields $\fV_{3,B}$, $\fV_{3,C}$ and $\fV_{3,D}$. Of course, the values of these fields do not allow us to obtain  Kirchhoff-Love displacements equal to $\GR$  in some cells close the boundary of  $\cO^H_\ed\X \fI_\d$.\\
Now, let $\psi$ be in $\cC^2(\R)$ satisfying
\begin{equation}\label{C2Basis}
	\left\{
	\begin{aligned}
		& \psi(t)\in [0,1], && \forall t\in \R,\\
		& \psi(t)=1, && \forall t\leq 0,\\
		& \psi(t)=0, && \forall t\geq 1,\\
		&\psi'(0)=\psi'(1)=0
	\end{aligned}
	\right.
\end{equation}
 and
$$ \psi_\d(t)  =\psi\Big({t\over \d}+{3\over 2}\Big),\qquad \forall t\in \R.$$ It satisfies
\begin{equation}\label{EQpsi}
\|\psi_\d\|_{L^\infty(\R)}\leq 1,\qquad \|\psi_\d^{'}\|_{L^\infty(\R)}\leq {C\over \d},\qquad \|\psi_\d^{''}\|_{L^\infty(\R)}\leq {C\over \d^2}.
\end{equation} The constant is independent of $\d$.\\
We set
$$\cV_{3}=\fV_{3,Z}+\fv_Z,\qquad \cS=-\nabla \fV_{3,Z}+\fs_Z,\qquad Z\in\big\{A,\, B,\, C,\, D\big\}.$$
We have
\begin{equation}\label{EQA8}
v(x)= \begin{pmatrix}
\ds \cV_1(x')-x_3 \p_1\fV_{3,Z} (x')+x_3\fs_{1,Z}(x')\\[1.5mm]
\ds\cV_2(x')-x_3 \p_2 \fV_{3,Z} (x')+x_3\fs_{2,Z}(x')\\[1.5mm]
\fV_{3,Z}(x')+\fv_Z(x')\\
\end{pmatrix}+\ov{v}(x)\quad  \hbox{ for a.e. $x$ in }  \cO_\ed\X\fI_\d.
\end{equation} where  $\fV_{3,Z} \in H^2(\cO_\ed)$ and  $\fs_{\a,Z},\; \fv_Z \in H^1(\cO_\ed)$. The estimates are given by Theorem \ref{THA2} with constants independent of $\e$ and $\d$. They are for any $Z\in\big\{A,\, B,\, C,\, D\big\}$
$$
\d\big\|\p^2_{\alpha\beta}\fV_{3,Z}\big\|_{L^2(\cO_\ed)}+\|\fs_Z\|_{L^2(\cO_\ed)}+\d \big\|\nabla \fs_Z\big\|_{L^2(\cO_\ed)} +{1\over \d}\|\fv_Z\|_{L^2(\cO_\ed)}+\big\|\nabla \fv_Z\big\|_{L^2(\cO_\ed)} \leq {C\over  \d^{1/2}}\|e(v)\|_{L^2(\cO_\ed\X\fI_\d)}.
$$ By construction, we have
$$\fV_{3,Z}-\cV_3=\fs_{\a,Z}=\fv_Z=0\qquad \hbox{a.e. in } \cO^H_\ed,\qquad Z\in\big\{A,\, B,\, C,\, D\big\}.$$
Now, we define $\fV_3$, $\fs$ and $\fv$. We set
$$
\begin{aligned}
\fV_3=&\fV_{3,A}\psi_{\d,A}+\fV_{3,B}\psi_{\d,B}+\fV_{3,C}\psi_{\d,C}+\fV_{3,D}\psi_{\d,D},\\
\fv=&\fv_A\psi_{\d,A}+\fv_B\psi_{\d,B}+\fv_C\psi_{\d,C}+\fv_D\psi_{\d,D},\\
\fs_\a=&\fs_{\a,A}\psi_{\d,A}+\fs_{\a,B}\psi_{\d,B}+\fs_{\a,C}\psi_{\d,C}+\fs_{\a,D}\psi_{\d,D}-\big(\fV_{3,A}\p_\a\psi_{\d,A}+\fV_{3,B}\p_\a\psi_{\d,B}+\fV_{3,C}\p_\a\psi_{\d,C}+\fV_{3,D}\p_\a\psi_{\d,D}\big),\\
=&\fs_{\a,A}\psi_{\d,A}+\fs_{\a,B}\psi_{\d,B}+\fs_{\a,C}\psi_{\d,C}+\fs_{\a,D}\psi_{\d,D}-\big(\fv_{A}\p_\a\psi_{\d,A}+\fv_{B}\p_\a\psi_{\d,B}+\fv_{C}\p_\a\psi_{\d,C}+\fv_{D}\p_\a\psi_{\d,D}\big),
\end{aligned}
$$ where
$$
\begin{aligned}
\psi_{\d,A}(x')&= \psi_\d(x_2)\psi_\d(x_1),\qquad && \psi_{\d,B}(x')=\psi_\d(x_2)\big(1-\psi_\d(x_1)\big),\\
\psi_{\d,C}(x')&=\big(1-\psi_\d(x_2)\big)\big(1-\psi_\d(x_1)\big),\quad &&\psi_{\d,D}(x')=\big(1-\psi_\d(x_2)\big)\psi_\d(x_1),
 \end{aligned}\qquad \forall x'\in \ov{\cO_\ed}.
$$ Observe that by construction 
$$
\begin{aligned}
&(\fV_{3,A}-\cV_3)\psi_{\d,A}+(\fV_{3,B}-\cV_3)\psi_{\d,B}+(\fV_{3,C}-\cV_3)\psi_{\d,C}+(\fV_{3,D}-\cV_3)\psi_{\d,D}=0,\\
&\fV_{3,A}\p_1\psi_{\d,A}+\fV_{3,B}\p_1\psi_{\d,B}+\fV_{3,C}\p_1\psi_{\d,C}+\fV_{3,D}\p_1\psi_{\d,D}=0,
\end{aligned}\qquad \hbox{a.e. in } \cO^H_\ed.
$$  So, we obtain the desired decomposition of $v$.
\begin{equation}\label{EQA10}
v(x)= \begin{pmatrix}
\ds \cV_1(x')-x_3 \p_1\fV_3 (x')+x_3\fs_1(x')\\[1.5mm]
\ds\cV_2(x')-x_3 \p_2 \fV_3 (x')+x_3\fs_2(x')\\[1.5mm]
\fV_3(x')+\fv(x')\\
\end{pmatrix}+\ov{v}(x)\quad  \hbox{ for a.e. $x$ in }  \cO_\ed\X\fI_\d.
\end{equation} where $\fV_3 \in H^2(\cO_\ed)$ and  $\fs_\a,\; \fv \in H^1(\cO_\ed)$. By construction, we have
$$\fV_{3}-\cV_3=\fs_{\a}=\fv=0\qquad \hbox{a.e. in } \cO^H_\ed.$$
 Bellow, we give the estimates satisfied by $\fV_3$ and  $\fs_\a,\; \fv$.
First, from \eqref{EQpsi} we have 
$$
\begin{aligned}
\big\|\fv\|_{L^2(\cO_\ed)}\leq &\big\|\fv_A\|_{L^2(\cO_\ed)}+\big\|\fv_B\|_{L^2(\cO_\ed)}+\big\|\fv_C\|_{L^2(\cO_\ed)}+\big\|\fv_D\|_{L^2(\cO_\ed)}\leq C \d^{1/2} \|e(v)\|_{L^2(\cO_\ed\X \fI_\d)},\\
\big\|\nabla\fv\|_{L^2(\cO_\ed)}\leq &\big\|\nabla\fv_A\|_{L^2(\cO_\ed)}+\big\|\nabla\fv_B\|_{L^2(\cO_\ed)}+\big\|\nabla\fv_C\|_{L^2(\cO_\ed)}+\big\|\nabla \fv_D\|_{L^2(\cO_\ed)}\\
+& {C\over \d}\big(\big\|\fv_A\|_{L^2(\cO_\ed)}+\big\|\fv_B\|_{L^2(\cO_\ed)}+\big\|\fv_C\|_{L^2(\cO_\ed)}+\big\|\fv_D\|_{L^2(\cO_\ed)}\big)\leq {C\over \d^{1/2} } \|e(v)\|_{L^2(\cO_\ed\X \fI_\d)}.
\end{aligned}
$$ Then, again with \eqref{EQpsi} we get 
$$
\begin{aligned}
\big\|\fs\|_{L^2(\cO_\ed)}\leq &\big\|\fs_A\|_{L^2(\cO_\ed)}+\big\|\fs_B\|_{L^2(\cO_\ed)}+\big\|\fs_C\|_{L^2(\cO_\ed)}+\big\|\fs_D\|_{L^2(\cO_\ed)}\\
+& {C\over \d}\big(\big\|\fv_A\|_{L^2(\cO_\ed)}+\big\|\fv_B\|_{L^2(\cO_\ed)}+\big\|\fv_C\|_{L^2(\cO_\ed)}+\big\|\fv_D\|_{L^2(\cO_\ed)}\big)\leq {C\over  \d^{1/2} }\|e(v)\|_{L^2(\cO_\ed\X \fI_\d)},\\
\big\|\nabla\fs\|_{L^2(\cO_\ed)}\leq &\big\|\nabla\fs_A\|_{L^2(\cO_\ed)}+\big\|\nabla\fs_B\|_{L^2(\cO_\ed)}+\big\|\nabla\fs_C\|_{L^2(\cO_\ed)}+\big\|\nabla \fs_D\|_{L^2(\cO_\ed)}\\
+& {C\over \d}\big(\big\|\fs_A\|_{L^2(\cO_\ed)}+\big\|\fs_B\|_{L^2(\cO_\ed)}+\big\|\fs_C\|_{L^2(\cO_\ed)}+\big\|\fs_D\|_{L^2(\cO_\ed)}\big)\\
+& {C\over \d^2}\big(\big\|\fv_A\|_{L^2(\cO_\ed)}+\big\|\fv_B\|_{L^2(\cO_\ed)}+\big\|\fv_C\|_{L^2(\cO_\ed)}+\big\|\fv_D\|_{L^2(\cO_\ed)}\big)\\
+&{C\over \d}\big\|\nabla\fv_A\|_{L^2(\cO_\ed)}+\big\|\nabla\fv_B\|_{L^2(\cO_\ed)}+\big\|\nabla\fv_C\|_{L^2(\cO_\ed)}+\big\|\nabla \fv_D\|_{L^2(\cO_\ed)} \leq {C\over \d^{3/2}}\|e(v)\|_{L^2(\cO_\ed\X \fI_\d)}.
\end{aligned}
$$ Below, we estimate $\p^2_{\a\b}\fV_3$. We have
$$
\begin{aligned}
\p^2_{\a\b}\fV_3=&\sum_{Z\in \{A,B,C,D\}}\big(\p^2_{\a\b}\fV_{3,Z}\psi_{\d,Z}+\p_\a\fV_{3,Z}\p_\b\psi_Z+\p_\b\fV_{3,Z}\p_\a\psi_Z+\fV_{3,Z}\p^2_{\a\b}\psi_{\d,Z}\big)\\
=&\sum_{Z\in \{A,B,C,D\}}\big(\p^2_{\a\b}\fV_{3,Z}\psi_{\d,Z}+\p_\a\fv_{Z}\p_\b\psi_Z+\p_\b\fv_{Z}\p_\a\psi_Z+\fv_{Z}\p^2_{\a\b}\psi_{\d,Z}\big).
\end{aligned}
$$ The estimates \eqref{remA2} and \eqref{EQpsi} lead to the following estimate:
$$\big\|\p^2_{\alpha\beta}\fV_3\big\|_{L^2(\o)}\le {C\over \d^{3/2}}\|e(u)\|_{L^2(\cO_\ed)}.$$
The constants are independent of $\e$ and $\d$.\\[1mm]
{\it Step 3.} We apply Step 2 in every cell $\O_{pq,\ed}$. With the help of function $\psi$ we construct the desired decomposition.\\[1mm]
Denote
$$
\begin{aligned}
&\psi_{pq,\ed}(x')=\psi\Big({x_1-(p+1)\e\over \d}-{1\over 2}\Big)\Big(1-\psi\Big({x_1-p\e\over \d}+{1\over 2}\Big)\Big)\psi\Big({x_2-(q+1)\e\over \d}-{1\over 2}\Big)\Big(1-\psi\Big({x_2-q\e\over \d}+{1\over 2}\Big)\Big),\\
&\hskip 85mm (p,q)\in\{1,\ldots,N_\e-2\}\X\{1,\ldots,N_\e-2\},\\
& \psi_{0q,\ed}(x')=\psi\Big({x_1-\e\over \d}-{1\over 2}\Big)\psi\Big({x_2-(q+1)\e\over \d}-{1\over 2}\Big)\Big(1-\psi\Big({x_2-q\e\over \d}+{1\over 2}\Big)\Big),\qquad q\in \{1,\ldots,N_\e-2\},\\
& \psi_{N_\e-1q,\ed}(x')=\Big(1-\psi\Big({x_1-(N_\e-1)\e\over \d}+{1\over 2}\Big)\Big)\psi\Big({x_2-(q+1)\e\over \d}-{1\over 2}\Big)\Big(1-\psi\Big({x_2-q\e\over \d}+{1\over 2}\Big)\Big),\\ &\hskip 115mm  q\in \{1,\ldots,N_\e-2\},\\
&\psi_{p0,\ed}(x')=\psi\Big({x_1-(p+1)\e\over \d}-{1\over 2}\Big)\Big(1-\psi\Big({x_1-p\e\over \d}+{1\over 2}\Big)\Big)\psi\Big({x_2-\e\over \d}-{1\over 2}\Big),\qquad  p\in\{1,\ldots,N_\e-1\},\\
&\psi_{pN_\e-1,\ed}(x')=\psi\Big({x_1-(p+1)\e\over \d}-{1\over 2}\Big)\Big(1-\psi\Big({x_1-p\e\over \d}+{1\over 2}\Big)\Big)\Big(1-\psi\Big({x_2-(N_\e-1)\e\over \d}-{1\over 2}\Big)\Big),\\ &\hskip 115mm  p\in\{1,\ldots,N_\e-1\}.
\end{aligned}
$$ Similarly, we define 
$$
\begin{aligned}
	&\psi_{00,\ed}(x')=\psi\Big({x_1-\e\over \d}-{1\over 2}\Big)\Big(1-\psi\Big({x_1\over \d}+{1\over 2}\Big)\Big)\psi\Big({x_2-\e\over \d}-{1\over 2}\Big)\Big(1-\psi\Big({x_2\over \d}+{1\over 2}\Big)\Big),\\
	& \psi_{0N_\e-1,\ed}(x')=\psi\Big({x_1-\e\over \d}-{1\over 2}\Big)\Big(1-\psi\Big({x_1\over \d}+{1\over 2}\Big)\Big)\psi\Big({x_2-N_\e\e\over \d}-{1\over 2}\Big)\Big(1-\psi\Big({x_2-(N_\e-1)\e\over \d}+{1\over 2}\Big)\Big),\\
	& \psi_{N_\e-10,\ed}(x')=\psi\Big({x_1-N_\e\e\over \d}-{1\over 2}\Big)\Big(1-\psi\Big({x_1-(N_\e-1)\e\over \d}+{1\over 2}\Big)\Big)\psi\Big({x_2-\e\over \d}-{1\over 2}\Big)\Big(1-\psi\Big({x_2\over \d}+{1\over 2}\Big)\Big),\\ 
	&\psi_{N_\e-1N_\e-1,\ed}(x')\\
	&\hskip 5mm =\psi\Big({x_1-N_\e\e\over \d}-{1\over 2}\Big)\Big(1-\psi\Big({x_1-(N_\e-1)\e\over \d}+{1\over 2}\Big)\Big)\psi\Big({x_2-N_\e\e\over \d}-{1\over 2}\Big)\Big(1-\psi\Big({x_2-(N_\e-1)\e\over \d}-{1\over 2}\Big)\Big).
\end{aligned}
$$
We have
$$\sum_{(p,q)\in \cK^*_\e}\psi_{pq,\ed}(x')=1\qquad \forall x'\in \o_\d$$ and
\begin{equation}\label{EQpsied}
\|\psi_{pq,\ed}\|_{L^\infty(\o_\d)}\leq 1,\qquad \|\nabla \psi_{pq,\ed}\|_{L^\infty(\o_\d)}\leq {C\over \d},\qquad \|\p^2_{\a\b}\psi\|_{L^\infty(\o_\d)}\leq {C\over \d^2}.
\end{equation} The constants are  independent of $\e$ and $\d$.\\[1mm]
Now, let $u$ be in $\cH^1(\O_\d)^3$, we decompose $u$ as \eqref{EQ41DA1+}. Then, we apply the result obtained in Step 2 to all the restrictions to $u$ to the sets $\O_{pq,\ed}$, $(p,q)\in \cK_\e$. This gives
\begin{equation}\label{EQA8+}
u(x)= \begin{pmatrix}
\ds \cU_1(x')-x_3 \p_1\fU_{3pq} (x')+x_3\fr_{1,pq}(x')\\[1.5mm]
\ds\cU_2(x')-x_3 \p_2 \fU_{3,pq} (x')+x_3\fr_{2,pq}(x')\\[1.5mm]
\fU_{3,pq}(x')+\fu_{pq}(x')\\
\end{pmatrix}+\ov{v}(x)\quad  \hbox{ for a.e. $x$ in }  \O_{pq,\ed}\X\fI_\d.
\end{equation} where  $\fU_{3,pq} \in H^2(\O_{pq,\ed})$ and  $\fr_{\a,pq},\; \fu_{pq} \in H^1(\O_{pq,\ed})$. Of course, the fields $\fU_{3,pq}$ and  $\fr_{\a,pq},\; \fu_{pq}$ satisfy the estimates obtained in the previous step  with constants independent of $\e$ and $\d$. By construction, we have
$$\fV_{3,pq}-\cV_3=\fs_{\a,pq}=\fv_{pq}=0\qquad \hbox{a.e. in } \o^H_{pq}\qquad \forall(p,q)\in \cK_\e.$$
Now, we define $\fU_3$, $\fr$ and $\fu$. We set
$$
\begin{aligned}
\fU_3=&\sum_{(p,q)\in \cK_\e} \fU_{3,pq}\psi_{pq,\ed},\qquad \fu=\sum_{(p,q)\in \cK_\e} \fu_{pq}\psi_{pq,\ed},\\
\fr_\a=&\sum_{(p,q)\in \cK_\e} \fr_{\a,pq}\psi_{pq,\ed}-\sum_{(p,q)\in \cK_\e} \fU_{3,pq}\p_\a\psi_{pq,\ed},\\
=&\sum_{(p,q)\in \cK_\e} \fr_{\a,pq}\psi_{pq,\ed}-\sum_{(p,q)\in \cK_\e} \fu_{pq}\p_\a\psi_{pq,\ed},\\
\end{aligned}\qquad \hbox{a.e. in }\O_{pq,\ed}.
$$
So, $\fU_3 \in H^2(\o_\d)$ and  $\fr_\a,\; \fu \in H^1(\o_\d)$ and by construction, we have
$$\fU_{3}-\cU_3=\fr_{\a}=\fu=0\qquad \hbox{a.e. in } \o^H_\ed.$$ Bellow, we give the estimates satisfied by $\fU_3$ and  $\fr_\a,\; \fu$.
First, from \eqref{EQpsied} we have 
$$
\begin{aligned}
\big\|\fu\|_{L^2(\o_\d)}\leq &\sum_{(p,q)\in \cK_\e} \big\|\fu_{pq}\|_{L^2(\wt{\o}_{pq})} \leq C \d^{1/2} \|e(u)\|_{L^2(\O_\d)},\\
\big\|\nabla\fu\|_{L^2(\o_\d)}\leq & \sum_{(p,q)\in \cK_\e}\big\|\nabla\fu_{pq}\|_{L^2(\wt{\o}_{pq})} + {C\over \d}\sum_{(p,q)\in \cK_\e} \big\|\fu_{pq}\|_{L^2(\wt{\o}_{pq})} \leq {C\over \d^{1/2} } \|e(u)\|_{L^2(\O_\d)}.
\end{aligned}
$$ Then, again with \eqref{EQpsi} we get 
$$
\begin{aligned}
\big\|\fr\|_{L^2(\wt{\o}_{pq})}\leq & \sum_{(p,q)\in \cK_\e} \big\|\fr_{pq}\|_{L^2(\wt{\o}_{pq})}+ {C\over \d} \sum_{(p,q)\in \cK_\e} \big\|\fu_{pq}\|_{L^2(\wt{\o}_{pq})}\leq {C\over  \d^{1/2} }\|e(u)\|_{L^2(\o_\d))},\\
\big\|\nabla\fr\|_{L^2(\wt{\o}_{pq})}\leq &  \sum_{(p,q)\in \cK_\e}\big\|\nabla\fr_{pq}\|_{L^2(\wt{\o}_{pq})}+ {C\over \d}  \sum_{(p,q)\in \cK_\e}\big\|\fr_{pq}\|_{L^2(\wt{\o}_{pq})}+ {C\over \d^2} \sum_{(p,q)\in \cK_\e}\big(\big\|\fu_{pq}\|_{L^2(\wt{\o}_{pq})}\\
+&{C\over \d} \sum_{(p,q)\in \cK_\e}\big\|\nabla\fu_{pq}\|_{L^2(\wt{\o}_{pq})}  \leq {C\over \d^{3/2} } \|e(u)\|_{L^2(\O_\d)}.
\end{aligned}
$$ Below, we estimate $\p^2_{\a\b}\fU_3$. We have
$$
\begin{aligned}
\p^2_{\a\b}\fU_3=& \sum_{(p,q)\in \cK_\e} \big(\p^2_{\a\b}\fU_{3,pq}\psi_{pq,\ed}+\p_\a\fU_{3,pq}\p_\b\psi_{pq,\ed}+\p_\b\fU_{3,pq}\p_\a\psi_{pq,\ed}+\fU_{3,pq}\p^2_{\a\b}\psi_{pq,\ed}\big)\\
=& \sum_{(p,q)\in \cK_\e} \big(\p^2_{\a\b}\fU_{3,pq}\psi_{pq,\ed}+\p_\a\fu_{pq}\p_\b\psi_{pq,\ed}+\p_\b\fu_{pq}\p_\a\psi_{pq,\ed}+\fu_{pq}\p^2_{\a\b}\psi_{pq,\ed}\big).
\end{aligned}
$$ The estimates obtained in Step 2 together with  \eqref{EQpsied} lead to the following estimate of $\p^2_{\alpha\beta}\fU_3$:
$$\big\|\p^2_{\alpha\beta}\fU_3\big\|_{L^2(\o)}\le {C\over \d^{3/2}}\|e(u)\|_{L^2(\O_\d)}.$$
The constants are independent of $\e$ and $\d$. Finally, we set
$$\wt{u}=\ov{u}+x_3\fr.$$
This ends the proof of Theorem \ref{TH42}.
\end{proof}

\section{Appendix: Some Lemmas}\label{SAppB}
Denote 
$$
\begin{aligned}
\GV & \doteq\big\{\phi\in H^1(\o_\d)\;|\; \nabla \phi=0\quad\hbox{a.e. in } \; \o^H_{\ed},\;\hbox{ and } \; \phi=0 \hbox{ a.e. on $\gamma$ and }\; \phi \hbox{ defined by \eqref{Dphi}} \big\},\\
\end{aligned}
$$
Below, we recall a few results from \cite{GGPumed} (see Lemmas 5.1, 5.2, 6.1 and 6.3 in this article).
\begin{lemma}\label{lemAp1}[Poincar\'e inequality] For any $\phi\in \GV$  we have the following  inequalities:
	\begin{equation}\label{EQAp}
\|\phi\|_{L^2(\o^S_\ed)}\leq C\sqrt{\d\over \e}\|\phi\|_{L^2(\o_\d)},\qquad  \|\phi\|_{L^2(\o_\d)}\leq C\sqrt{\d\over \e}\|\nabla\phi\|_{L^2(\o_\d)}.
	\end{equation}
	The constants do not depend on $\e$ and $\d$.
\end{lemma}
\begin{lemma}\label{lemAp2} Let $\{\phi_\ed\}_{\e,\d}$ be a sequence of functions belonging to $\GV$ and satisfying 
	\begin{equation}\label{EQ83}
		\|\nabla\phi_\ed\|_{L^2(\o_\d)}\leq C\sqrt{\e\over \d}.
	\end{equation}
	There exist a subsequence of $\{(\e,\d)\}$, still denoted $\{(\e,\d)\}$, and $\phi\in H^1(\o)$  satisfying $\phi=0$ a.e. on $\gamma$  such that 
	\begin{equation}\label{EQ72}
		\begin{aligned}
		\phi_\ed&\to \phi\quad &&\hbox{strongly in } L^2(\o),\\
	\sqrt{\e\over \d}\Ted^{(1)}(\phi_\ed)&\to  \phi \quad&& \hbox{strongly in } L^2(\o\X Y_1),\\
	{\d\over \e}\Ted^{(1)}(\p_2\phi_\ed)={1\over \e}\p_{y_2} \Ted^{(1)}(\phi_\ed)&\rightharpoonup \p_2\phi \quad&& \hbox{weakly in } L^2(\o\X Y_1),\\
	{\d\over \e}\Ted^{(1)}(\p_1\phi_\ed)={\d\over \e^2}\p_{y_1} \Ted^{(1)}(\phi_\ed)&\rightharpoonup 0 \quad&& \hbox{weakly in } L^2(\o\X Y_1),\\
	{\d\over \e}\Ted^{(2)}(\p_1\phi_\ed)={1\over \e}\p_{y_1}\Ted^{(2)}(\phi_\ed)&\rightharpoonup \p_1\phi \quad &&\hbox{weakly in } L^2(\o\X Y_2),\\
	{\d\over \e}\Ted^{(2)}(\p_2\phi_\ed)={\d\over \e^2}\p_{y_2}\Ted^{(2)}(\phi_\ed)&\rightharpoonup 0 \quad&& \hbox{weakly in } L^2(\o\X Y_2).
		\end{aligned}
	\end{equation}
\end{lemma}

\begin{lemma}\label{lemAp5}[Test function] For any $\phi\in W^{1,\infty}(\o)$ there exists $\phi_\ed\in W^{1,\infty}(\o_\d)\cap \GV$ satisfying
	\begin{equation}\label{EQ85}
\begin{aligned}
&\phi_\ed \longrightarrow \phi\quad \hbox{strongly in } L^2(\o),\\
{\d\over \e} &\Ted^{(1)}(\p_2\phi_\ed) \longrightarrow  \p_2\phi  \quad \hbox{strongly in } L^2(\o\X Y_1),\\
{\d\over \e} & \Ted^{(2)}(\p_1\phi_\ed) \longrightarrow  \p_1\phi  \quad \hbox{strongly in } L^2(\o\X Y_2).
		\end{aligned}
	\end{equation}
\end{lemma}
\begin{proof} First, thanks to symmetries with respect to the edges of $\o$, we extend $\phi$ into a function, still denoted $\phi$, belonging to $W^{1,\infty}(\o_\d)$. Then, proceeding as in Subsection \ref{SS42} to construct $\cR$,  setting
	$$\phi_{\ed,pq}=\phi\Big(p\e+{\e\over 2},q\e+{\e\over 2}\Big)\qquad \forall (p,q)\in \ov{\cK}_\e$$ 
	we construct the function $\phi_\ed$ belonging to $W^{1,\infty}(\o_\d)$. We easily show that
	$$
	\|\phi_\ed\|_{L^\infty(\o_\d)}\leq \|\phi\|_{L^\infty(\o)},\qquad \sum_{(p,q)\in \cK_\e} \big\|\nabla\phi_\ed\big\|^2_{L^2(R^3_{pq})}\leq C\|\nabla\phi\|^2_{L^\infty(\o)}
	$$ and (similarly as in Lemma \ref{lemAp1})
	$$
	\sum_{(p,q)\in \cK^{(1)}_\e} \big\|\p_2\phi_\ed\big\|^2_{L^2(R^1_{pq})}\leq C{\e\over \d}\|\p_2\phi\|^2_{L^\infty(\o)},\qquad  \sum_{(p,q)\in \cK^{(2)}_\e} \big\|\p_1\phi_\ed\big\|^2_{L^2(R^2_{pq})}\leq C{\e\over \d}\|\p_1\phi\|^2_{L^\infty(\o)}.
	$$
	We recall that (see \eqref{EXRc01}) $\p_1\phi_\ed=0$ in $R^1_{pq}$ and $\p_2\phi_\ed=0$ in $R^2_{pq}$. \\
	From \eqref{EQ55}-\eqref{EQ56} we obtain
	$$
	\begin{aligned}
		&\|\Ted^{(1)}(\p_2\phi_\ed)\|_{L^2(\o\X Y_1)}\leq C{\e\over \d}\|\p_2\phi\|_{L^\infty(\o)},\qquad \|\Ted^{(2)}(\p_1\phi_\ed)\|_{L^2(\o\X Y_2)}\leq C{\e\over \d}\|\p_1\phi\|_{L^\infty(\o)},\\
		\Longrightarrow\quad &\Big\|{\p\Ted^{(1)}(\phi_\ed)\over \p y_2}\Big\|_{L^2(\o\X Y_1)}\leq C\e\|\p_2\phi\|_{L^\infty(\o)},\qquad \Big\|{\p \Ted^{(2)}(\phi_\ed)\over \p y_1}\Big\|_{L^2(\o\X Y_2)}\leq C\e\|\p_1\phi\|_{L^\infty(\o)}.
	\end{aligned}
	$$ By construction of $\phi_\ed$, we have 
$$
\begin{aligned}
	{\p\Ted^{(1)}(\phi_\ed)\over \p y_2}(x',y')=\phi\Big(\e\Big[{x'\over \e}\Big]+{\e\over 2}\Ge_2\Big)-\phi\Big(\e\Big[{x'\over \e}\Big]-{\e\over 2}\Ge_2\Big),\quad \hbox{for a.e. } (x',y')\in \o \X \Big({\d\over 2\e}, 1- {\d\over 2\e}\Big)\X \fI,\\
	{\p\Ted^{(2)}(\phi_\ed)\over \p y_1}(x',y')=\phi\Big(\e\Big[{x'\over \e}\Big]+{\e\over 2}\Ge_1\Big)-\phi\Big(\e\Big[{x'\over \e}\Big]-{\e\over 2}\Ge_1\Big),\quad \hbox{for a.e. } (x',y')\in \o \X \fI  \X \Big({\d\over 2\e}, 1- {\d\over 2\e}\Big).
\end{aligned}
$$ 
As a consequence of the above estimates and equalities we obtain the convergences \eqref{EQ85}$_{2,3}$.
\end{proof}
\begin{remark} In the above lemma, if $\phi$ vanishes on $\gamma$ we construct $\phi_\ed$ by setting
	$$\phi_{pq}=0\qquad \hbox{for any }\; (p,q)\in \{-1,0\}\X\{0,\ldots,n_\e\}\cup\{0,\ldots, n_\e\}\X\{-1,0\}$$ in order to obtain a test function vanishing in the neighborhood of $\gamma$. The results of Lemma \ref{lemAp5} remain valid.
\end{remark}
\begin{lemma}\label{lemAp7} To any $\Phi\in W^{2,\infty}(0,L)$ we associate a sequence of functions $\Phi_\ed \in\ds W^{2,\infty}\Big(-{\d\over 2},L+{\d\over 2}\Big)$such that
\begin{equation}\label{B5}
	\begin{aligned}
		\Phi_\ed &\to \Phi\quad &&\text{strongly in $H^1(0,L)$},\\
		{\d\over \e}\Ted^{(2)}(d^2\Phi_\ed) &\to d^2\Phi\quad &&\text{strongly in $L^2(\o\X Y_2)$},\\
		{\d\over \e}\Ted^{(1)}(d^2\Phi_\ed) &\to 0\quad &&\text{strongly in $L^2(\o\X Y_1)$}.
	\end{aligned}
\end{equation}
Moreover, if $\Phi=d\Phi=0$ a.e. on $(0,l)$, then we have $\Phi_\ed=d\Phi_\ed=0$ a.e on $(-{\d\over 2},l)$.
\end{lemma}
\begin{proof} 
	 First, we extend $\Phi$ to a function belonging to $\ds W^{2,\infty}\Big(-{\e\over 2},L+{\e\over 2}\Big)$ by setting
		$$\left\{\begin{aligned}
			&\Phi(0)+t\Phi'(0),\quad &&t\in(-{\e\over 2},0],\\
			&\Phi(t),\quad &&t\in(0,L),\\
			&\Phi(L)+(t-L)\Phi'(L),\quad && t\in[L,L+{\e\over 2}).
		\end{aligned}\right.
		$$ This extension of $\Phi$ is still denoted $\Phi$. Observe that
		$$\|d\Phi\|_{L^\infty(-{\e\over 2},L+{\e\over 2})}=\|d\Phi\|_{L^\infty(0,L)},\qquad \|d^2\Phi\|_{L^\infty(-{\e\over 2},L+{\e\over 2})}=\|d^2\Phi\|_{L^\infty(0,L)}.$$
		The extension is needed on $(-{\e\over 2},L+{\e\over 2})$, and not only on $(-{\d\over 2},L+{\d\over 2})$, since \eqref{REDef}$_1$ evaluates $d\Phi$ at $p\e+{\e\over \d}z_1\in\big(p\e-{\e\over 2},p\e+{\e\over 2}\big)$ for $p\in\{0,\ldots,N_\e\}$.
	
	Now, we define  the function $\Phi_\ed$. For that first we set 
	\begin{equation}\label{REDef}
		d\Phi_\ed(p\e+z_1)=\left\{\begin{aligned}
			&d\Phi\left(p\e+{\e\over \d}z_1\right)&& z_1\in \fI_\d,\qquad &&p\in\{0,\ldots,N_\e\},\\
			&d\Phi\left(p\e+{\e\over 2}\right)  && z_1\in\left({\d\over 2},\e-{\d\over 2}\right),\quad &&p\in\{0,\ldots,N_\e-1\}
		\end{aligned}\right.
	\end{equation} and
	$$\Phi_\ed(z_1)=\Phi(0)+\int_0^{z_1}d\Phi_\ed(t)\,dt,\qquad z_1\in \left(-{\d\over 2},L+{\d\over 2}\right).$$
	 Note that $d\Phi_\ed$ is continuous: at $p\e+{\d\over 2}$ both lines of \eqref{REDef} give $d\Phi\big(p\e+{\e\over 2}\big)$, and at $(p+1)\e-{\d\over 2}$ both give $d\Phi\big((p+1)\e-{\e\over 2}\big)$. Moreover, $d\Phi_\ed$ is Lipschitz on each interval of \eqref{REDef}.
	So, we have $\ds\Phi_\ed\in W^{2,\infty}\Big(-{\d\over 2},L+{\d\over 2}\Big)$ and 
	\begin{equation}\label{B9}
		\begin{aligned}
			&d^2\Phi_\ed(p\e+z_1)=0,\quad  z_1\in \left({\d\over 2},\e-{\d\over 2}\right),\quad p\in\{0,\ldots,N_\e-1\},\\
			&d^2\Phi_\ed(p\e+z_1)={\e\over \d}d^2\Phi\left(p\e+{\e\over \d}z_1\right),\quad  z_1\in\fI_\d,\quad p\in\{0,\ldots,N_\e\}.
		\end{aligned}
	\end{equation} Moreover
	\begin{equation}\label{B7}
		\begin{aligned}
			&\|d\Phi_\ed\|_{L^\infty(-{\d\over 2},L+{\d\over 2})}\leq { \|d\Phi\|_{L^\infty(0,L)}},\quad \|d^2\Phi_\ed\|_{L^\infty(-{\d\over 2},L+{\d\over 2})}\leq {\e\over \d}{ \|d^2\Phi\|_{L^\infty(0,L)}},\\
			&\|d\Phi_\ed-d\Phi\|_{L^2(-{\d\over 2},L+{\d\over 2})}\leq C\e,\quad \|\Phi_\ed-\Phi\|_{L^2(-{\d\over 2},L+{\d\over 2})}\leq C\e.
		\end{aligned}
	\end{equation}
	For the last estimate, we use $1$D anchored Poincar\'e inequality, since $\Phi_\ed(0)=\Phi(0)$. So, we get \eqref{B5}$_1$. 
	
	 Below, we consider the function $x_1\longmapsto \Phi_\ed(x_1)$, which we regard as a function on $\o_\d$ independent of $x_2$, and we set $\ds p=\Big[{x_1\over \e}\Big]$ for $x_1\in(0,L)$, so that $p\in\{0,\ldots,N_\e-1\}$.
	By \eqref{Ted}, for a.e. $(x',y')\in\o\X Y_2$ we have $\Ted^{(2)}(d^2\Phi_\ed)(x',y')=d^2\Phi_\ed(p\e+\d y_1)$. Hence \eqref{B9}$_2$ gives
		\begin{equation}\label{B10}
			{\d\over \e}\Ted^{(2)}(d^2\Phi_\ed)(x',y')=d^2\Phi\Big(\e\Big[{x_1\over \e}\Big]+\e y_1\Big),\qquad \text{for a.e. $(x',y')\in\o\X Y_2$}.
		\end{equation}
		Observe that $d^2\Phi=0$ on $\big(-{\e\over 2},0\big)\cup\big(L,L+{\e\over 2}\big)$. Hence, extending $d^2\Phi$ by $0$ outside $\big(-{\e\over 2},L+{\e\over 2}\big)$, we have $d^2\Phi\in L^2(\R)$, and this extension does not change the right-hand side of \eqref{B10}, since $\e[x_1/\e]+\e y_1\in\big(-{\e\over 2},L-{\e\over 2}\big)$. Writing $x_1=p\e+z_1$ with $z_1\in(0,\e)$ and then $s=\e y_1-z_1\in\big(-{3\e\over 2},{\e\over 2}\big)$, we obtain
		$$
		\begin{aligned}
			\Big\|{\d\over \e}\Ted^{(2)}(d^2\Phi_\ed)-d^2\Phi\Big\|^2_{L^2(\o\X Y_2)}
			&=L\sum_{p=0}^{N_\e-1}\int_0^\e\int_{\fI}\big|d^2\Phi(p\e+\e y_1)-d^2\Phi(p\e+z_1)\big|^2\,dy_1\,dz_1\\
			&\leq {L\over \e}\int_{-{3\e\over 2}}^{{\e\over 2}}\sum_{p=0}^{N_\e-1}\int_0^\e\big|d^2\Phi(p\e+z_1+s)-d^2\Phi(p\e+z_1)\big|^2\,dz_1\,ds\\
			&= {L\over \e}\int_{-{3\e\over 2}}^{{\e\over 2}}\big\|d^2\Phi(\cdot+s)-d^2\Phi\big\|^2_{L^2(0,L)}\,ds
			\leq 2L\sup_{|s|\leq {3\e\over 2}}\big\|d^2\Phi(\cdot+s)-d^2\Phi\big\|^2_{L^2(0,L)}.
		\end{aligned}
		$$
		The right-hand side tends to $0$ as $\e\to0$, by the continuity of translations in $L^2(\R)$. This gives \eqref{B5}$_2$.
		
		By \eqref{Ted}, for a.e. $(x',y')\in\o\X Y_1$ we have $\Ted^{(1)}(d^2\Phi_\ed)(x',y')=d^2\Phi_\ed(p\e+\e y_1)$. Using \eqref{B9}, we obtain, for a.e. $(x',y')\in\o\X Y_1$,
		$$
		{\d\over \e}\Ted^{(1)}(d^2\Phi_\ed)(x',y')=\left\{\begin{aligned}
			&d^2\Phi\Big(p\e+{\e^2\over \d}y_1\Big),\quad &&y_1\in\Big(0,{\d\over 2\e}\Big),\\
			&0,\quad &&y_1\in\Big({\d\over 2\e},1-{\d\over 2\e}\Big),\\
			&d^2\Phi\Big((p+1)\e+{\e^2\over \d}(y_1-1)\Big),\quad &&y_1\in\Big(1-{\d\over 2\e},1\Big).
		\end{aligned}\right.
		$$
		The first and third lines come from \eqref{B9}$_2$ with $z_1=\e y_1\in\big(0,{\d\over 2}\big)$ in the cell $p$, resp. $z_1=\e(y_1-1)\in\big(-{\d\over 2},0\big)$ in the cell $p+1$; the second line comes from \eqref{B9}$_1$. Since the two boundary layers have total measure ${\d\over \e}$ in $Y_1$, we get
		$$
		\Big\|{\d\over \e}\Ted^{(1)}(d^2\Phi_\ed)\Big\|_{L^2(\o\X Y_1)}\leq L\,\sqrt{\d\over \e}\,\|d^2\Phi\|_{L^\infty(0,L)}\longrightarrow 0,
		$$
		which is \eqref{B5}$_3$.
		
		Finally, if $\Phi=d\Phi=0$ on $(0,l)$, then $\Phi(0)=d\Phi(0)=0$, so the extension of $\Phi$ vanishes on $\big(-{\e\over 2},0\big]$. For $z_1\in\big(-{\d\over 2},l\big)$, \eqref{REDef} evaluates $d\Phi$ only at points of $\big(-{\e\over 2},l\big)$, where it vanishes. Hence $d\Phi_\ed=0$ a.e. on $\big(-{\d\over 2},l\big)$, and then $\Phi_\ed(z_1)=\Phi(0)+\int_0^{z_1}d\Phi_\ed(t)\,dt=0$ there.
	
	This completes the proof.
\end{proof}

\section{Appendix: Coercivity results}\label{SAppC}

In this subsection, we give a coercivity result required for the proof of the existence of a unique weak solution to the limit unfolded problem.
\begin{lemma}
	For every $(\cV,\wh v^{(1)},\wh v^{(2)})\in \D$, the following estimate hold:
	\begin{equation}\label{CoeL02}
		\|\cV_1\|^2_{H^1(\o)}+\|\cV_2\|^2_{H^1(\o)}+\|\cV_3\|^2_{H^2(\o)}+\sum_{\a=1}^2\|\wh v^{(\a)}\|^2_{L^2(\o\X(0,1)_{y_\a};H^1(\fI^2))}\leq C\sum_{\a=1}^2\|E^{(\a)}(\cV,\wh v^{(\a)})\|^2_{L^2(\o\X\cY_\a)}.
	\end{equation}
\end{lemma}
\begin{proof}
	Since, we have
	$$\wh v^{(1)}=0,\quad\text{a.e. on $\o\X(0,1)\X\{\pm{1\over 2}\}\X\fI$},\quad \wh v^{(2)}=0,\quad\text{a.e. on $\o\X\{\pm{1\over 2}\}\X(0,1)\X\fI$},$$
	after a straightforward calculation we get
	$$
	\begin{aligned}
	&\|\nabla \fU_1\|^2_{L^2(\o)}+\|\nabla \fU_2\|^2_{L^2(\o)}+\|\p^2_{11}\cV_3\|^2_{L^2(\o)}+\|\p^2_{22}\cV_3\|^2_{L^2(\o)}+\|\p_{y_2}\wh{u}^{(1)}_1\|^2_{L^2(\o\X \cY_1)}+\|\p_{y_3}\wh{u}^{(1)}_1\|^2_{L^2(\o\X \cY_1)}\\
	+&\|\p_{y_1}\wh{u}^{(2)}_2\|^2_{L^2(\o\X \cY_2)}+\|\p_{y_3}\wh{u}^{(2)}_2\|^2_{L^2(\o\X \cY_2)} +\|e_{y,22}(\wh{u}^{(1)})\|^2_{L^2(\o\X \cY_1)}+\|e_{y,23}(\wh{u}^{(1)})\|^2_{L^2(\o\X \cY_1)}\\
	+&\|e_{y,33}(\wh{u}^{(1)})\|^2_{L^2(\o\X \cY_1)}+\|e_{y,11}(\wh{u}^{(2)})\|^2_{L^2(\o\X \cY_2)}+\|e_{y,13}(\wh{u}^{(2)})\|^2_{L^2(\o\X \cY_2)}+\|e_{y,33}(\wh{u}^{(2)})\|^2_{L^2(\o\X \cY_2)}\\
&\leq C\sum_{\a=1}^2\|E^{(\a)}(\cV,\wh v^{(\a)})\|^2_{L^2(\o\X\cY_\a)}
	\end{aligned}
	$$ from which we easily deduce \eqref{CoeL02}.
\end{proof}

%
%
\end{document}